\documentclass[11pt,a4paper]{article}

\usepackage{soul}
\usepackage{tikz}
\usepackage{amsthm}
\usepackage{amsmath}
\usepackage{amssymb}
\usepackage{mathrsfs}

\usepackage{geometry}
\usepackage{graphicx}
\usepackage{amsfonts}
\usepackage{epstopdf}
\usepackage{placeins}
\usepackage{enumerate}
\usepackage{enumitem}
\usepackage{subcaption}
\usepackage{microtype}
\usepackage[hidelinks]{hyperref} 
\usepackage[normalem]{ulem} 

\usetikzlibrary{calc,arrows.meta, positioning}

\usetikzlibrary{decorations.markings}

\numberwithin{equation}{section}

\allowdisplaybreaks

\newcommand{\esup}{\operatorname*{ess\,sup}}
\newcommand{\einf}{\operatorname*{ess\,inf}}

\newcommand{\Rmnum}[1]{\uppercase\expandafter{\romannumeral#1}} 

\def\Xint#1{\mathchoice
{\XXint\displaystyle\textstyle{#1}}%
{\XXint\textstyle\scriptstyle{#1}}%
{\XXint\scriptstyle\scriptscriptstyle{#1}}%
{\XXint\scriptscriptstyle\scriptscriptstyle{#1}}%
\!\int}
\def\XXint#1#2#3{{\setbox0=\hbox{$#1{#2#3}{\int}$ }
\vcenter{\hbox{$#2#3$ }}\kern-.6\wd0}}

\def\dashint{\Xint-}

\theoremstyle{plain}
\newtheorem{theorem}{Theorem}[section]
\newtheorem{proposition}[theorem]{Proposition}
\newtheorem{lemma}[theorem]{Lemma}

\newtheorem{definition}[theorem]{Definition}

\theoremstyle{definition}
\newtheorem{remark}[theorem]{Remark}

\renewcommand{\thefootnote}{}

\AtEndDocument{
\bigskip{\footnotesize%
   \textsc{Department of Mathematics, Aarhus University, 8000 Aarhus C, Denmark} \par
  \textit{E-mail address}: \texttt{yangmengqh@gmail.com}\par
  \textit{E-mail address}: \texttt{yang@math.au.dk}
}}

\begin{document}

\title{On the uniqueness of $p$-energy on the Sierpi\'nski carpet}
\author{Meng Yang}
\date{}

\maketitle

\abstract{For any $p\in(1,+\infty)$, it is proved that the self-similar $p$-energy on the Sierpi\'nski carpet is unique, up to multiplication by a positive constant.}

\footnote{\textsl{Date}: \today}

\renewcommand{\thefootnote}{\arabic{footnote}}
\setcounter{footnote}{0}

\section{Introduction}\label{sec_intro}

The construction of Brownian motion and its associated Dirichlet form is a fundamental problem in analysis on fractals. In the absence of a smooth differential structure, local regular Dirichlet forms provide an analog of the classical Dirichlet integral and a framework for defining Laplace operators and diffusion processes. Early developments include the work of Barlow and Perkins \cite{BP88} on the Sierpi\'nski gasket (see Figure \ref{fig_SG}), Lindstr{\o}m \cite{Lindstrom90} on nested fractals, and Kigami \cite{Kig93} on post-critically finite (p.c.f.) self-similar sets; see also the lecture notes by Barlow \cite{Bar98} and the monograph by Kigami \cite{Kig01book}. On many p.c.f. self-similar sets, finite ramification allows self-similar energies to be constructed by renormalizing energies on finite graphs.

\begin{figure}[ht]
\centering
\begin{minipage}[t]{0.48\textwidth}
\centering
\includegraphics[width=\textwidth]{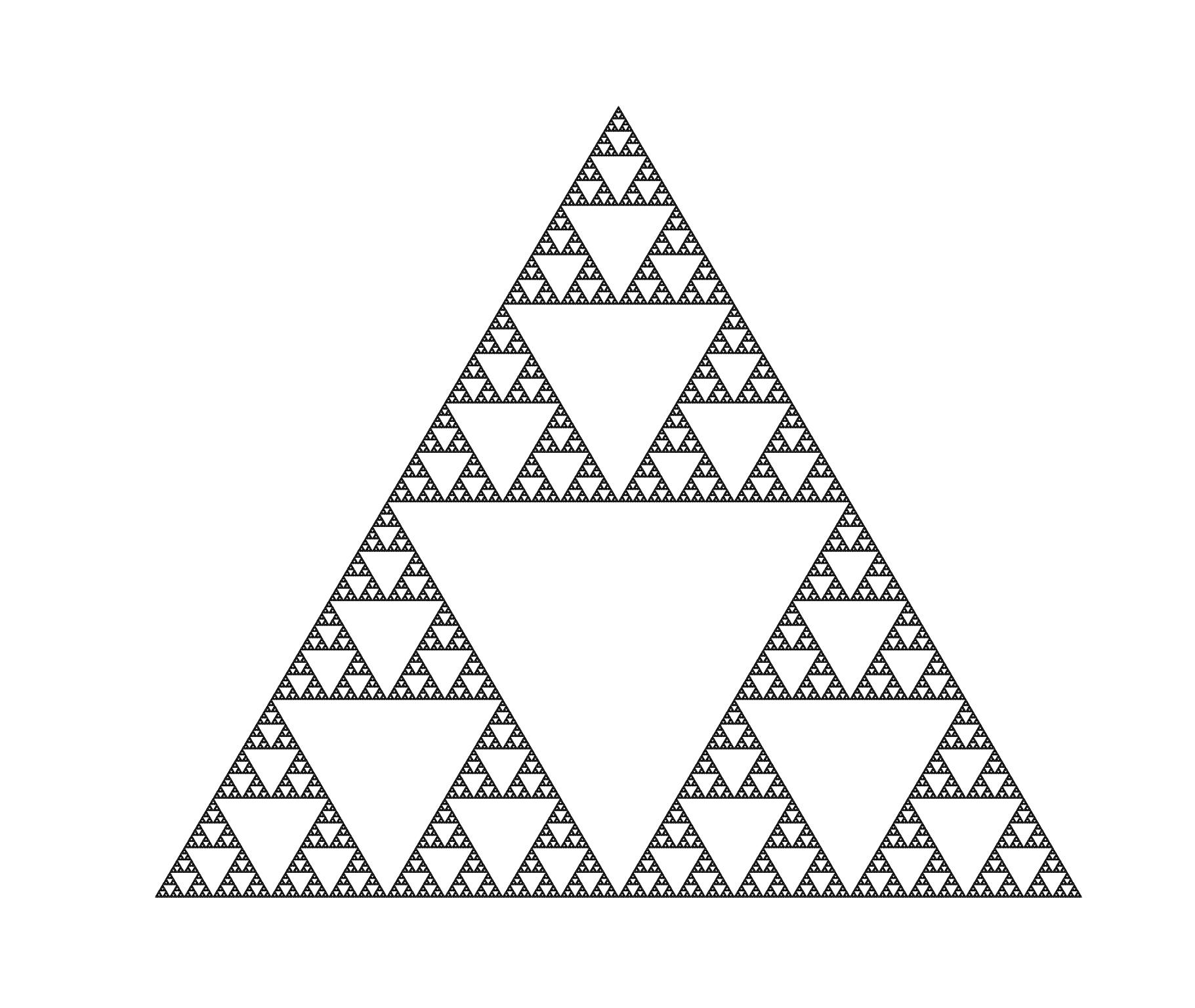}
\caption{The Sierpi\'nski gasket}\label{fig_SG}
\end{minipage}
\begin{minipage}[t]{0.48\textwidth}
\centering
\includegraphics[width=\textwidth]{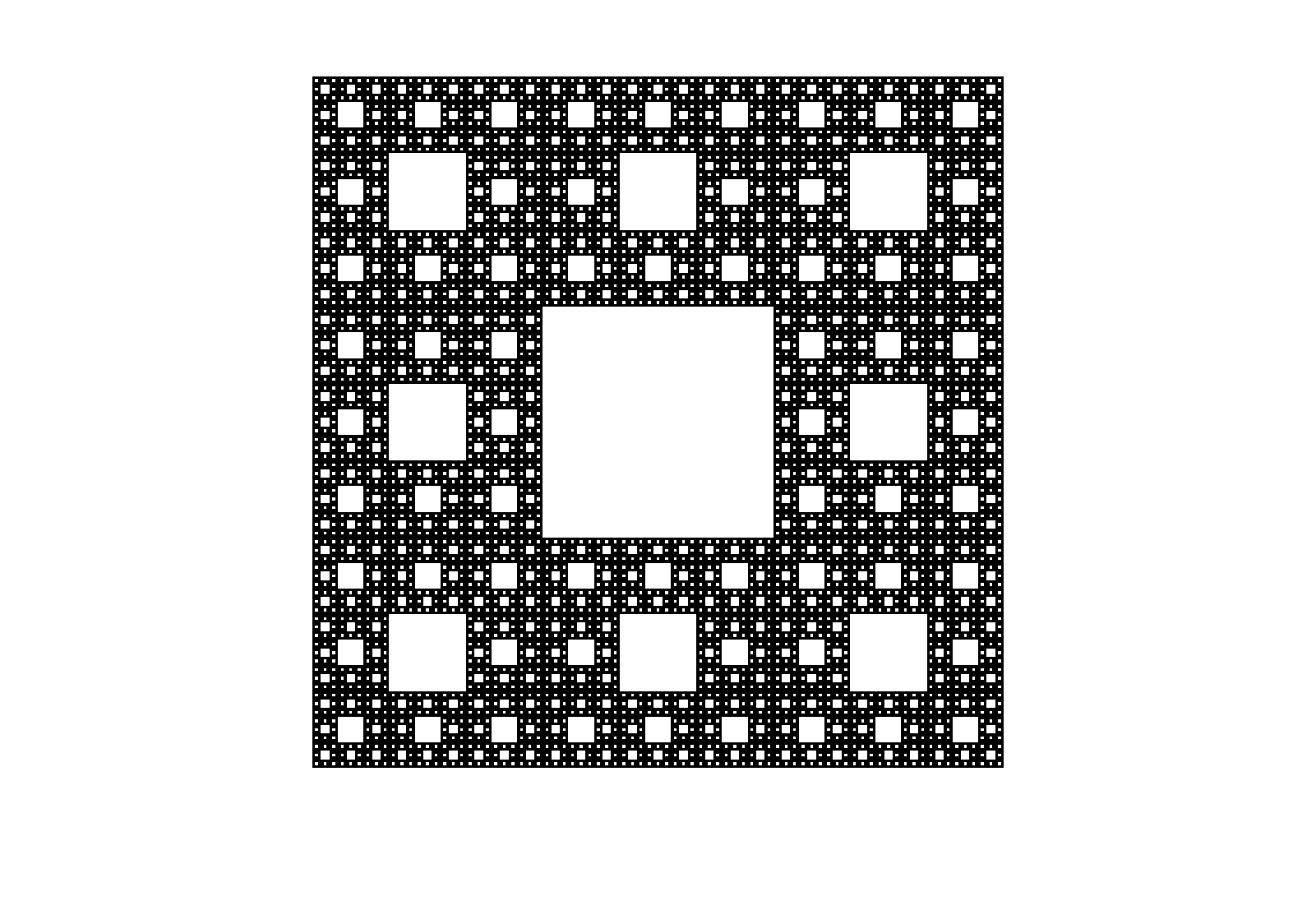}
\caption{The Sierpi\'nski carpet}\label{fig_SC}
\end{minipage}
\end{figure}

The Sierpi\'nski carpet presents substantially greater difficulties (see Figure \ref{fig_SC}). It is infinitely ramified and is not p.c.f., so the finite-dimensional renormalization methods available in the p.c.f. setting do not apply. Barlow and Bass \cite{BB89} first constructed Brownian motion on the carpet by taking subsequential limits of suitably time-changed reflected Brownian motions on planar domains approximating the carpet. Kusuoka and Zhou \cite{KZ92} subsequently developed a construction based on approximating graphs and Markov chains, which also yielded a self-similar Dirichlet form. These constructions required estimates uniform across the approximations and left open whether the resulting Brownian motions coincide, up to a constant rescaling of time.

The uniqueness problem was resolved by Barlow, Bass, Kumagai and Teplyaev \cite{BBKT10}. They proved that, up to multiplication by a positive constant, there is a unique non-zero local regular conservative Dirichlet form invariant under the local symmetries of a generalized Sierpi\'nski carpet. Their proof uses Harnack inequalities and heat kernel estimates to establish uniform comparison estimates for admissible forms. Applying these estimates to suitable differences of such forms, they show that any two such forms differ only by a constant factor.

More recently, Cao and Qiu \cite{CQ24} constructed symmetric self-similar Dirichlet forms on unconstrained Sierpi\'nski carpets. On these carpets, the constituent cells need not align with the usual grids, and intersections between cells may be line segments of irrational length. In \cite{CQ26}, they also proved uniqueness, up to multiplication by a positive constant, of self-similar $D_4$-symmetric resistance forms on such carpets. Their uniqueness proof follows the comparison argument in \cite{BBKT10}, which involves suitable differences of forms.

For $p\in(1,+\infty)$, the classical integral $\int_{\mathbb{R}^d}|\nabla f(x)|^p\mathrm{d}x$ suggests a non-linear counterpart of the Dirichlet form framework. The construction of such energies on fractals was initiated by Herman, Peirone and Strichartz \cite{HPS04} in 2004. They constructed $p$-energies on a class of Sierpi\'nski gasket-type fractals through renormalization of discrete energies. Recent advances in the study of $p$-energies on fractals and general metric measure spaces include the work of Cao, Gu and Qiu \cite{CGQ22} on p.c.f. self-similar sets, Baudoin and Chen \cite{BC23} on the Vicsek set, Chen, Gao, Yu and Zhang \cite{CGYZ26} on scale-irregular Vicsek sets, Kigami \cite{Kig23} on compact metric spaces, Alonso-Ruiz and Baudoin \cite{AB25} on Cheeger spaces, and Anttila, Eriksson-Bique and Shimizu \cite{AEBS25} on a class of Laakso-type fractal spaces.

For generalized Sierpi\'nski carpets, Shimizu \cite{Shi24} constructed self-similar $p$-energies and associated energy measures when $p$ is greater than the Ahlfors regular conformal dimension. Murugan and Shimizu \cite{MS25} subsequently constructed self-similar $p$-energies and energy measures on the standard planar Sierpi\'nski carpet for any $p\in(1,+\infty)$, removing this restriction on $p$. The principal difficulties on the Sierpi\'nski carpet include its infinite ramification and the possible discontinuity of finite-energy functions. These constructions lead naturally to the following uniqueness question, formulated in \cite[Conjecture 10.3]{MS25}:

\vspace{5pt}
\noindent\textbf{Question}: Is the self-similar $p$-energy on the Sierpi\'nski carpet unique, up to multiplication by a positive constant?
\vspace{5pt}

We answer this question affirmatively. More precisely, it is proved that for any $p\in(1,+\infty)$, any self-similar $p$-energy on the standard Sierpi\'nski carpet satisfying conditions \ref{item_SSEa}--\ref{item_SSEf} in Definition \ref{def_SSE} is a positive constant multiple of the canonical energy constructed in \cite{MS25}; see Theorem \ref{thm_unique}. In particular, its domain and self-similar scaling factor are uniquely determined.

For $p\ne2$, however, the non-linear nature of the problem creates additional obstructions. The argument in \cite{BBKT10} uses the quadratic structure of Dirichlet forms, particularly the fact that suitable differences of comparable local forms are still Dirichlet forms. This step does not extend to $p$-energies: the difference of two ``well-behaved" energies need not even satisfy the required convexity condition. Thus their uniqueness argument cannot be applied, and a different approach is needed.

A rough outline of the proof is as follows. We start with the quasi-uniqueness result in \cite{MS23arXiv}, which establishes the comparability of $p$-energies and their associated energy measures. We then use the notion of $p$-independence in \cite{EM25}, which corresponds to the notion of index in \cite{Hin10}, to show that the canonical $p$-energy has index one. This yields a common density relating the two families of energy measures. Finally, self-similarity and $D_4$-symmetry force this density to be constant, establishing uniqueness.

\section{Statement of main results}\label{sec_statement}

Before stating our main results, we fix some notation. Throughout this paper, $p\in(1,+\infty)$ is fixed. The letters $C$, $C_1$, $C_2$, $C_A$, $C_B$ will always refer to some positive constants and may change at each occurrence. The sign $\asymp$ means that the ratio of the two sides is bounded from above and below by positive constants. The sign $\lesssim$ ($\gtrsim$) means that the LHS is bounded by a positive constant times the RHS from above (below). We use $x_+$ and $x_-$ to denote the positive and negative parts of $x\in \mathbb{R}$, respectively, that is, $x_+=\max\{x,0\}$ and $x_-=\max\{-x,0\}$. For $x,y\in\mathbb{R}$, we write $x\wedge y=\min\{x,y\}$ and $x\vee y=\max\{x,y\}$. For two $\sigma$-finite Borel measures $\mu, \nu$, the notion $\mu\le\nu$ means that $\mu\ll \nu$ and $\frac{\mathrm{d}\mu}{\mathrm{d}\nu}\le1$, that is $\mu$ is absolutely continuous with respect to $\nu$ with Radon-Nikodym derivative bounded by 1, and the notion $\mu\lesssim \nu$ means that $\mu\le C\nu$ for some positive constant $C$. We use $\#A$ to denote the cardinality of a set $A$. We use $\mathcal{L}^d$ to denote the Lebesgue measure on $\mathbb{R}^d$. For any $x\in \mathbb{R}$, we define the sign function by $\mathrm{sgn}(x)=1$ if $x>0$, $\mathrm{sgn}(x)=0$ if $x=0$, and $\mathrm{sgn}(x)=-1$ if $x<0$.

Let $(K,d,m)$ be a complete metric measure space, that is, $(K,d)$ is a complete locally compact separable metric space and $m$ is a positive Radon measure on $K$ with full support. Throughout this paper, we always assume that all metric balls are relatively compact. For any $x\in K$ and any $r>0$, denote $B(x,r)=\{y\in K:d(x,y)<r\}$ and $V(x,r)=m(B(x,r))$. If $B=B(x,r)$, then denote $\delta B=B(x,\delta r)$ for any $\delta>0$. Let $\mathcal{B}(K)$ be the family of all Borel measurable subsets of $K$. Let $C(K)$ be the family of all continuous functions on $K$. Let $C_c(K)$ be the family of all continuous functions on $K$ with compact support. Denote $\dashint_A=\frac{1}{m(A)}\int_A$ and $u_A=\dashint_Au \mathrm{d} m$ for any measurable set $A$ with $m(A)\in(0,+\infty)$ and any function $u$ such that the integral $\int_Au \mathrm{d}m$ is well-defined.

In $\mathbb{R}^2$, let $\square:=[0,1]\times[0,1]\subseteq \mathbb{R}^2$ denote the unit square, and let $\ell_L$, $\ell_R$, $\ell_B$, and $\ell_T$ denote its left, right, bottom, and top sides, respectively. Let
\begin{align*}
&p_0=(0,0),p_1=(\frac{1}{2},0),p_2=(1,0),p_3=(1,\frac{1}{2}),\\
&p_4=(1,1),p_5=(\frac{1}{2},1),p_6=(0,1),p_7=(0,\frac{1}{2}),
\end{align*}
and let
$$F_i(x)=\frac{1}{3}(x-p_i)+p_i,x\in \mathbb{R}^2,i=0,1,\ldots,7.$$
Then the Sierpi\'nski carpet is the unique non-empty compact subset $K\subseteq \mathbb{R}^2$ satisfying
$$K=\bigcup_{i=0}^{7}F_i(K).$$
Let $D_4$ denote the dihedral group of the unit square $\square$. The Sierpi\'nski carpet $K$ is invariant under $D_4$. We say that $K^\partial:=\ell_L\cup \ell_R\cup \ell_B\cup \ell_T$ is the boundary of $K$, and $K^{\mathrm{int}}:=K\backslash K^\partial$ is the interior of $K$. Let $d$ be the Euclidean metric, then the Hausdorff dimension of $(K,d)$ is
$$d_h:=\frac{\log8}{\log3}.$$

For any $n\ge0$, let
$$W_n:=\left\{w=w_1\cdots w_n:w_i\in\{0,1,\ldots,7\}\text{ for any }i=1,\dots,n\right\}$$
be the set of words of length $n$ over the alphabet $\{0,1,\ldots,7\}$, where $W_0=\{\emptyset\}$ and $\emptyset$ is the empty word. Let $W_*:=\cup_{n\ge0}W_n$.

Let
$$\Sigma:=\left\{0,1,\ldots,7\right\}^\infty=\left\{ w=w_1w_2\cdots:w_i\in\{0,1,\ldots,7\}\text{ for any }i=1,2,\dots\right\}$$
be endowed with the product topology, then $\Sigma$ is a compact metrizable space. For any $i=0,1,\ldots,7$, let $\sigma_i:\Sigma\to \Sigma$ be given by $w_1w_2\cdots\mapsto iw_1w_2\cdots$. Then there exists a continuous surjective map $\pi:\Sigma\to K$ such that $\pi\circ \sigma_i=F_i\circ \pi$ for any $i=0,1,\ldots,7$.

For any $n,m\ge0$ and any
$$w^{(1)}=w^{(1)}_1\cdots w^{(1)}_n\in W_n,\qquad w^{(2)}=w^{(2)}_1\cdots w^{(2)}_m\in W_m,$$
define the concatenation $w^{(1)}w^{(2)}$ of $w^{(1)}$ and $w^{(2)}$ by
$$w^{(1)}w^{(2)}:=w^{(1)}_1\cdots w^{(1)}_n w^{(2)}_1\cdots w^{(2)}_m\in W_{n+m}.$$
For any $n\ge0$ and any
$$w^{(1)}=w^{(1)}_1\cdots w^{(1)}_n\in W_n,\qquad w^{(2)}=w^{(2)}_1w^{(2)}_2\cdots \in \Sigma,$$
define the concatenation $w^{(1)}w^{(2)}$ of $w^{(1)}$ and $w^{(2)}$ by
$$w^{(1)}w^{(2)}:=w^{(1)}_1\cdots w^{(1)}_n w^{(2)}_1 w^{(2)}_2\cdots\in \Sigma.$$
For any $m\ge n\ge0$ and any $w=w_1\cdots w_m\in W_m$, define the $n$-th truncation $[w]_n$ of $w$ by
$$[w]_n:=w_1\cdots w_n\in W_n.$$
For any $n\ge0$ and any $w=w_1w_2\cdots \in \Sigma$, define the $n$-th truncation $[w]_n$ of $w$ by
$$[w]_n:=w_1\cdots w_n\in W_n.$$

For any $n\ge0$ and any word
$$w=w_1\cdots w_n\in W_n,$$
let $\lvert w\rvert=n$ be the length of the word $w$, and define
$$F_w=F_{w_1}\circ\cdots\circ F_{w_n},\qquad K_w=F_w(K),\qquad K_w^{\mathrm{int}}=F_w(K^{\mathrm{int}}),\qquad K_w^{\partial}=F_w(K^\partial),$$
$$\sigma_w=\sigma_{w_1}\circ\cdots\circ \sigma_{w_n},\qquad \Sigma_w=\sigma_w(\Sigma).$$
We say that $K_w$ is a level-$n$ cell, $K_w^{\mathrm{int}}$ is the interior of $K_w$, $K_w^{\partial}$ is the boundary of $K_w$, and $\{F_w(i,j):i,j=0,1\}$ is the set of corners of $K_w$. In particular, for $n=0$, $\lvert \emptyset\rvert=0$, $F_\emptyset=\mathrm{Id}_{\mathbb{R}^2}$ is the identity map in $\mathbb{R}^2$, $K_\emptyset=K$, $K^{\mathrm{int}}_\emptyset=K^{\mathrm{int}}$, $K^\partial_\emptyset=K^\partial$, $\sigma_\emptyset=\mathrm{Id}_\Sigma$ is the identity map in $\Sigma$, and $\Sigma_\emptyset=\Sigma$. For any $n\ge0$, we have
$$K=\bigcup_{w\in W_n}K_w.$$
Let
$$K^\partial_n:=\bigcup_{w\in W_n}K^\partial_w,\qquad K^\partial_*:=\bigcup_{n\ge0}K^\partial_n,\qquad K^{\mathrm{int}}_n=\bigcup_{w\in W_n}K^{\mathrm{int}}_w.$$
For any $x\in K\backslash K^\partial_*$ and any $n\ge0$, there exists a unique word in $W_n$, denoted by $[x]_n$, such that $x\in K_{[x]_n}$.

Let $m$ be the unique positive Radon measure on $K$ such that $m(K_w)=8^{-|w|}$ for any $w\in W_*$, then $m$ is the normalized $d_h$-dimensional Hausdorff measure.

\begin{definition}\label{def_SSE}
We say that $(\mathcal{E},\mathcal{F})$ is a self-similar $p$-energy on $K$ if $\mathcal{F}\subseteq L^p(K;m)$ and $\mathcal{E}:\mathcal{F}\to[0,+\infty)$ satisfy the following conditions.
\begin{enumerate}[label=(SSE\alph*),ref=(SSE\alph*)]
\item\label{item_SSEa} $\{f\in \mathcal{F}:\mathcal{E}(f)=0\}=\{f\in L^p(K;m):f\text{ is constant }m\text{-a.e. on }K\}$. For any $f\in \mathcal{F}$ and $a\in \mathbb{R}$,
$$\mathcal{E}(f+a)=\mathcal{E}(f),\qquad \mathcal{E}(af)=\lvert a\rvert^p \mathcal{E}(f).$$
\item\label{item_SSEb} $\mathcal{E}^{\frac{1}{p}}$ is a seminorm on $\mathcal{F}$. Moreover, $(\mathcal{F},(\mathcal{E}(\cdot)+\lVert \cdot\rVert_{L^p(K;m)}^p)^{\frac{1}{p}})$ is a uniformly convex Banach space.
\item\label{item_SSEc} (Regularity) $\mathcal{F}\cap C(K)$ is uniformly dense in $C(K)$ and $(\mathcal{E}(\cdot)+\lVert \cdot\rVert_{L^p(K;m)}^p)^{\frac{1}{p}}$-dense in $\mathcal{F}$.
\item\label{item_SSEd} (Symmetry) For any $\psi\in D_4$ and any $f\in \mathcal{F}$, $f\circ \psi\in \mathcal{F}$ and $\mathcal{E}(f\circ \psi)=\mathcal{E}(f)$.
\item\label{item_SSEe} (Self-similarity) There exists $\rho\in(0,+\infty)$ such that for any $f\in \mathcal{F}$ and any $i=0,1,\ldots,7$, $f\circ F_i\in \mathcal{F}$ and
$$\mathcal{E}(f)=\rho\sum_{i=0}^7 \mathcal{E}(f\circ F_i).$$
Moreover,
$$\mathcal{F}\cap C(K)=\left\{f\in C(K):f\circ F_i\in \mathcal{F}\text{ for any }i=0,1,\ldots,7\right\}.$$
\item\label{item_SSEf} (Unit contractivity) For any $f\in \mathcal{F}$, $(f\vee0)\wedge1\in \mathcal{F}$ and $\mathcal{E}((f\vee0)\wedge1)\le \mathcal{E}(f)$.
\end{enumerate}
\end{definition}

Let
$$\mathfrak{E}:=\left\{(\mathcal{E},\mathcal{F}):(\mathcal{E},\mathcal{F})\text{ is a self-similar $p$-energy}\right\}.$$

Fix $(\mathcal{E},\mathcal{F})\in \mathfrak{E}$. Let
$$\mathcal{E}_1(\cdot):=\mathcal{E}(\cdot)+\lVert \cdot\rVert_{L^p(K;m)}^p,\qquad\beta_p:=\frac{\log(8\rho)}{\log3},$$
where $\rho$ is the parameter in \ref{item_SSEe}.

Under \ref{item_SSEa}--\ref{item_SSEf}, we can extend unit contractivity \ref{item_SSEf} to:
\begin{enumerate}[label=(SSE\alph*'),ref=(SSE\alph*'),start=6]
\item\label{item_SSEfp} (Lipschitz contractivity) For any $f\in\mathcal F$ and any Lipschitz function $\varphi:\mathbb R\to\mathbb R$, $\varphi\circ f\in\mathcal F$ and $\mathcal E(\varphi\circ f)\le \mathrm{Lip}(\varphi)^p\mathcal E(f)$.
\end{enumerate}
The proof follows by a standard piecewise affine approximation argument, using weak compactness and Mazur’s lemma.

By the self-similar property, for any $f\in \mathcal{F}$, there exists a positive Radon measure $\Gamma(f)$ on $K$, called the $\mathcal{E}$-energy measure of $f$, given as follows. First, we define a measure $\nu_f$ on $\Sigma$. For any $w\in W_*$, let $\nu_f(\Sigma_w)=\rho^{\lvert w\rvert}\mathcal{E}(f\circ F_w)$. By the self-similar property, we may apply Kolmogorov’s consistency theorem to obtain a measure $\nu_f$ on $\Sigma$ satisfying
$$\nu_f(\Sigma_w)=\sum_{i=0}^7\nu_f(\Sigma_{wi})\text{ for any }w\in W_*.$$
Then $\Gamma(f):=\pi_*\nu_f$ is defined as the pushforward of $\nu_f$ under $\pi$. Obviously, $\Gamma(f)(K)=\mathcal{E}(f)$.

The conditions in Definition \ref{def_SSE} imply the following spectral gap condition.

\begin{proposition}\label{prop_gap}
Let $(\mathcal{E},\mathcal{F})\in \mathfrak{E}$. Then the following condition holds.
\begin{enumerate}[label=(SSE\alph*),ref=(SSE\alph*),start=7]
\item\label{item_SSEg} (Spectral gap) There exists a constant $C_{gap}>0$ such that
$$\lVert f-f_K\rVert_{L^p(K;m)}^p\le C_{gap}\mathcal{E}(f)\text{ for any }f\in \mathcal{F}.$$
\end{enumerate}
\end{proposition}

We henceforth include \ref{item_SSEg}, which follows from \ref{item_SSEa}--\ref{item_SSEf}, among the conditions in Definition \ref{def_SSE}. Under this additional condition, we obtain the following Poincar\'e inequality (see \cite[Proposition 10.31]{MS23arXiv}, \cite[Theorem 8.21]{MS25}): there exist constants $C_{PI}>0$ and $A_{PI}\ge1$ such that for any ball $B(x,r)$ and any $f\in\mathcal{F}$,
\begin{equation*}\label{eq_PI}\tag*{$\text{PI}(\beta_p)$}
\int_{B(x,r)}\lvert f-f_{B(x,r)}\rvert^p\mathrm{d} m\le C_{PI}r^{\beta_p}\int_{{B(x,A_{PI}r)}}\mathrm{d}\Gamma(f).
\end{equation*}

By \cite[Theorem 1.1]{MS25}, there exists a canonical self-similar $p$-energy $(\mathcal{E}^{(c)},\mathcal{F}^{(c)})\in \mathfrak{E}$, which is constructed by the $\Gamma$-convergence of discrete energies on a sequence of approximating graphs, and satisfies the generalized $p$-contraction property introduced in \cite{KS24aV5}, so that $\mathfrak{E}\ne \emptyset$. We use $\rho^{(c)}$ to denote the parameter associated with $(\mathcal{E}^{(c)},\mathcal{F}^{(c)})$ in \ref{item_SSEe}. The parameter
$$\beta_p^{(c)}:=\frac{\log(8\rho^{(c)})}{\log3}$$
is called the $p$-walk dimension of $(\mathcal{E}^{(c)},\mathcal{F}^{(c)})$.

The main result of this paper is as follows.

\begin{theorem}\label{thm_unique}
For any $(\mathcal{E},\mathcal{F})\in \mathfrak{E}$, let $\rho$ be the parameter associated with $(\mathcal{E},\mathcal{F})$ in \ref{item_SSEe}, then $\rho=\rho^{(c)}$, $\mathcal{F}=\mathcal{F}^{(c)}$, and there exists a constant $c>0$ such that
$$\mathcal{E}(f)=c\mathcal{E}^{(c)}(f)\text{ for any }f\in \mathcal{F}=\mathcal{F}^{(c)}.$$
\end{theorem}

The idea of the proof is as follows. First, one can obtain a quasi-uniqueness result from \cite[Propositions 10.26 and 10.29]{MS23arXiv}: any $(\mathcal{E},\mathcal{F})\in \mathfrak{E}$ has the same parameter $\rho=\rho^{(c)}$, the same domain $\mathcal{F}=\mathcal{F}^{(c)}$, and $\mathcal{E}$ is comparable to $\mathcal{E}^{(c)}$, that is, there exists a constant $C\ge 1$ such that
$$\frac{1}{C}\mathcal{E}^{(c)}\le \mathcal{E}\le C\mathcal{E}^{(c)},\qquad \frac{1}{C}\Gamma^{(c)}\le \Gamma\le C\Gamma^{(c)}.$$

We then use the notion of $p$-independence from \cite[Definition 2.25]{EM25}; see also the references therein for related results. In the linear Dirichlet form setting, this notion is related to the index defined in \cite[Definition 2.9]{Hin10}. We show that any two functions $f_1,f_2\in\mathcal{F}^{(c)}$ are $p$-dependent with respect to $(\mathcal{E}^{(c)},\mathcal{F}^{(c)})$, or equivalently, $(\mathcal{E}^{(c)},\mathcal{F}^{(c)})$ has index one; see Theorem \ref{thm_R1}. Roughly speaking, at almost every point, the space of energy measures is $1$-dimensional. The proof requires the fact that $\frac{d_h}{2}<\frac{\beta^{(c)}_p}{p}$, which follows from \cite[Corollary 2.2]{Yan25e} or \cite[Theorem 9.8]{KS24aV5}.

By the index-one property and the comparability of $\Gamma$ and $\Gamma^{(c)}$, we then find a Borel function $a$ such that
$$\Gamma(f)=a\Gamma^{(c)}(f)\ \text{ for any }f\in\mathcal{F}=\mathcal{F}^{(c)}.$$
We show that $a$ is invariant under the contraction maps $F_i$ and
the isometries in $D_4$, which implies the corresponding invariance properties of its level sets $\{a>r\}$ for $r\in \mathbb{R}$. A zero-one law in Theorem \ref{thm_01law} then forces $a$ to be constant and yields the uniqueness in Theorem \ref{thm_unique}.

This paper is organized as follows. In Section \ref{sec_gap}, we prove Proposition \ref{prop_gap}. In Sections \ref{sec_harmonic}--\ref{sec_01law}, we are restricted to the canonical $p$-energy. In Section \ref{sec_harmonic}, we prove some results on harmonic functions for later use. In Section \ref{sec_R1}, we prove that the canonical $p$-energy has index one (Theorem \ref{thm_R1}). In Section \ref{sec_01law}, we prove a zero-one law (Theorem \ref{thm_01law}). In Section \ref{sec_unique}, we prove Theorem \ref{thm_unique}. In Section \ref{sec_Hin310}, we prove a result in Section \ref{sec_R1}.

\section{Proof of Proposition \ref{prop_gap}}\label{sec_gap}

Throughout this section, we assume that $(\mathcal{E},\mathcal{F})\in \mathfrak{E}$. Figure \ref{fig_depend} gives the dependencies among the main results in this section.

\begin{figure}[ht]
  \centering
\begin{tikzpicture}[
  x=1cm,y=1cm,
  font=\normalsize,
  box/.style={
    draw=black!75,line width=.7pt,rounded corners=1.2mm,
    fill=white,align=center,minimum width=2.7cm,
    minimum height=.8cm,inner sep=2mm
  },
  arrow/.style={-{Stealth[length=2.4mm,width=1.7mm]},draw=black!70,line width=.8pt}
]

\node[box] (l31)  at (0,0)      {\text{Lemma \ref{lem_rhochi}}};
\node[box] (l311) at (9.2,0)    {Lemma \ref{lem_BAM}};

\node[box] (l37)  at (4.6,-1.6) {Lemma \ref{lem_restrict}};
\node[box] (l310) at (2.3,-3.2) {Lemma \ref{lem_path01}};
\node[box] (l39)  at (2.3,-4.8) {Lemma \ref{lem_inter4}};
\node[box] (l38)  at (9.2,-4.8) {Lemma \ref{lem_bdd_Lp}};

\node[box] (l35)  at (0,-6.4)   {Lemma \ref{lem_gh}};
\node[box] (p36)  at (4.6,-6.4) {Prop. \ref{prop_gap_pre}};
\node[box] (l34)  at (9.2,-6.4) {Lemma \ref{lem_ug}};
\node[box] (p22)  at (4.6,-8)   {Prop. \ref{prop_gap}};

\draw[arrow] (l31) -- (l35);
\draw[arrow] (l31) -- (l310);
\draw[arrow] (l311) -- (l37);
\draw[arrow] (l311) -- (l38);
\draw[arrow] (l37) -- (l310);
\draw[arrow] (l37) -- (p36);
\draw[arrow] (l310) -- (l39);
\draw[arrow] (l39) -- (p36);
\draw[arrow] (l38) -- (p36);
\draw[arrow] (l35) -- (p22);
\draw[arrow] (p36) -- (p22);
\draw[arrow] (l34) -- (p22);

\end{tikzpicture}
\caption{Dependencies among the main results in Section \ref{sec_gap}}\label{fig_depend}
\end{figure}
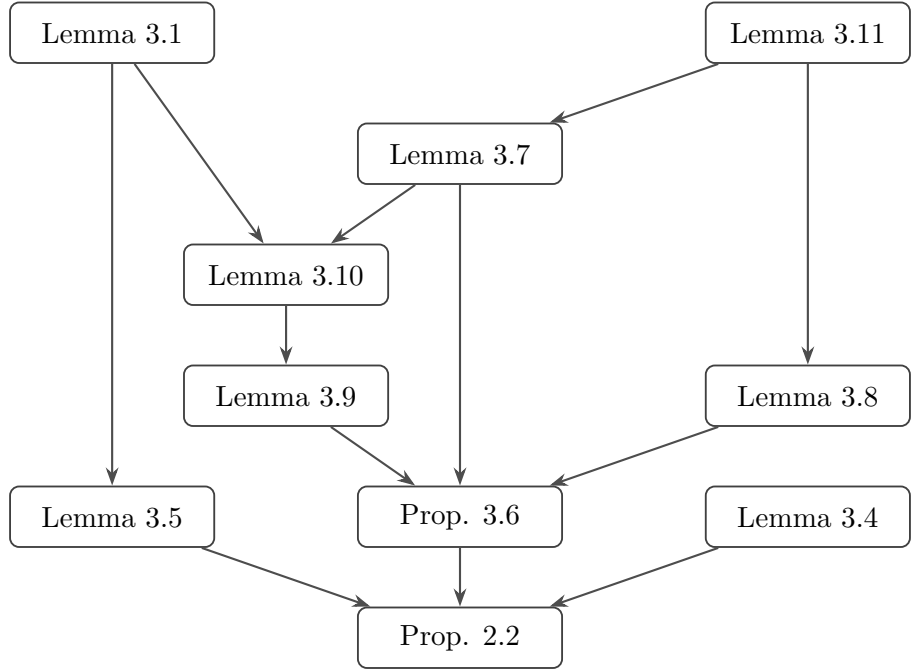

As a preliminary step, we have the following capacity estimates.

\begin{lemma}[{\cite[Lemma 10.24]{MS23arXiv}}]\label{lem_rhochi}
We have $\rho\ge \frac12$ and
\begin{align}
\chi:=&\inf \left\{\mathcal{E}(f):f\in \mathcal{F}\cap C(K),f\le 0\text{ on }{\ell_L},f\ge1\text{ on }{\ell_R}\right\}\nonumber\\
=&\inf \left\{\mathcal{E}(f):f\in \mathcal{F}\cap C(K),f|_{\ell_L}\equiv0,f|_{\ell_R}\equiv1\right\}\in(0,+\infty).\label{eq_rhochi}
\end{align}
\end{lemma}

\begin{remark}
In the proof of {\cite[Lemma 10.24]{MS23arXiv}}, in particular, in the derivation of Equation (10.24) therein, only \ref{item_SSEa}, \ref{item_SSEb}, \ref{item_SSEc}, \ref{item_SSEe}, and \ref{item_SSEf} are needed. First, \ref{item_SSEc} and \ref{item_SSEf} give the equality in the second line of (\ref{eq_rhochi}) and $\chi<+\infty$. Second, assuming \ref{item_SSEa}, \ref{item_SSEb}, and \ref{item_SSEe}, by gluing an approximating sequence $\{u_n\}$ to get a sequence $\{g_n\}$ with $\mathcal{E}(g_n)=2\rho \mathcal{E}(u_n)$, choosing an appropriate subsequence, and applying Mazur's lemma, one can obtain a sequence $\{\widetilde{g}_n\}\subseteq \mathcal{F}\cap C(K)$ such that $\widetilde{g}_n|_{\ell_L}\equiv0$ and $\widetilde{g}_n|_{\ell_R}\equiv1$, hence $\chi\le \mathcal{E}(\widetilde{g}_n)$, and $\{\widetilde{g}_n\}$ is $\mathcal{E}$-convergent to a function $g_\infty\in \mathcal{F}$ with $g_\infty=0$ $m$-a.e. on $K_0\cup K_6\cup K_7$ and $g_\infty=1$ $m$-a.e. on $K_2\cup K_3\cup K_4$. Therefore,
$$\chi\le \mathcal{E}(g_\infty)=\lim_{n\to+\infty}\mathcal{E}(\widetilde{g}_n)\le2\rho \lim_{n\to+\infty} \mathcal{E}(u_n)\le 2\rho\chi,$$
then \ref{item_SSEa} and the fact that $g_\infty$ is non-constant give $\chi>0$, which implies $\rho\ge \frac12$.
\end{remark}

Roughly speaking, the above capacity estimates mean that any function in $\mathcal{F}\cap C(K)$ cannot ``oscillate" too much between two opposite sides of $K$. Suppose that \ref{item_SSEg} does not hold, so that some function may oscillate too much somewhere. We use self-similarity and a blowup argument to construct a function with excessive oscillation between opposite sides of $K$, yielding a contradiction.

We list some properties of energy measures under the assumptions \ref{item_SSEa}--\ref{item_SSEf} alone.

\begin{lemma}\label{lem_meas}
\begin{enumerate}[label=(\alph*),ref=(\alph*)]
\item\label{item_meas1} For any $A\in \mathcal{B}(K)$, $\Gamma(\cdot)(A)^{\frac{1}{p}}$ is a seminorm on $\mathcal{F}$.
\item\label{item_meas2} For any $f\in \mathcal{F}$ and any Lipschitz function $\varphi:\mathbb{R}\to\mathbb{R}$, $\Gamma(\varphi\circ f)\le \mathrm{Lip}(\varphi)^p \Gamma(f)$.
\item\label{item_meas3} For any $f\in \mathcal{F}$ and any $n\ge0$,
\begin{equation}\label{eq_Gamma_ss}
\Gamma(f)=\rho^n\sum_{w\in W_n}(F_w)_*\Gamma(f\circ F_w).
\end{equation}
\item\label{item_meas4} For any $f\in \mathcal{F}$ and any $\psi\in D_4$,
\begin{equation}\label{eq_Gamma_D4}
\Gamma(f\circ \psi)=(\psi^{-1})_*\Gamma(f).
\end{equation}
\item\label{item_meas21} For any $f\in \mathcal{F}\cap C(K)$ and any $a\in \mathbb{R}$,
\begin{equation}\label{eq_Gamma_AC}
\Gamma(f)\left(\left\{f=a\right\}\right)=0,
\end{equation}
\item\label{item_meas22} For any $f,g\in\mathcal{F}\cap C(K)$ and any $A\in \mathcal{B}(K)$,
\begin{equation}\label{eq_Gamma_local}
\text{if $(f-g)|_A$ is constant, then $\Gamma(f)(A)=\Gamma(g)(A)$},
\end{equation}
\end{enumerate}
\end{lemma}

Parts \ref{item_meas1}--\ref{item_meas4} are precisely \cite[Proposition 8.15(b)--(e)]{MS25}, respectively, with the correction that $\Phi_*$ in part (e) of the cited proposition should be replaced by $(\Phi^{-1})_*$. Part \ref{item_meas21} follows from \cite[Corollary 8.18]{MS25}, and part \ref{item_meas22} is precisely \cite[Corollary 8.19]{MS25}. Notice that the spectral gap condition \ref{item_SSEg} is not needed for these results.

As preparation, we need to establish several results.

\begin{lemma}\label{lem_ug}
There exist constants $C_1,C_2>1$ depending only on $p$ such that for any non-zero non-negative $u\in \mathcal{F}\cap C(K)$, there exists $g\in \mathcal{F}\cap C(K)$ satisfying that $0\le g\le 1$ on $K$, $\{g>0\}\subseteq\{u>0\}$, and
$$0<m \left(\left\{g>0\right\}\right)\le C_1 m \left(\left\{g=1\right\}\right),\qquad\frac{\mathcal{E}(g)}{m\left(\left\{g=1\right\}\right)}\le C_2 \frac{\mathcal{E}(u)}{\lVert u\rVert_{L^p(K;m)}^p}.$$
\end{lemma}

\begin{proof}
Since $m$ is a finite Radon measure on $K$, there exists $a>0$ such that
\begin{equation}\label{eq_ug1}
m\left(\bigcup_{n\in \mathbb{Z}}\left\{u=2^na\right\}\right)=0.
\end{equation}
By (\ref{eq_Gamma_AC}),
\begin{equation}\label{eq_ug2}
\Gamma(u)\left(\left\{u=0\right\}\right)=\Gamma(u)\left(\bigcup_{n\in \mathbb{Z}}\left\{u=2^na\right\}\right)=0.
\end{equation}
For any $n\in \mathbb{Z}$, let $g_n:=\frac{1}{2^na}((u-2^na)_+\wedge(2^na))$, then $g_n\in \mathcal{F}\cap C(K)$, $0\le g_n\le 1$ on $K$, and $\{g_n>0\}\subseteq\{u>0\}$. By (\ref{eq_ug2}), (\ref{eq_Gamma_AC}), and (\ref{eq_Gamma_local}),
\begin{equation}\label{eq_ug3}
\mathcal{E}(u)=\sum_{n\in \mathbb{Z}}\int_{2^na<u<2^{n+1}a}\mathrm{d}\Gamma(u)=a^p\sum_{n\in \mathbb{Z}}2^{pn}\mathcal{E}(g_n).
\end{equation}
For any $n\in \mathbb{Z}$, let
$$t_n:=m \left(\left\{u>2^na\right\}\right),$$
then by (\ref{eq_ug1}),
$$t_n=m \left(\left\{g_n>0\right\}\right),\qquad t_{n+1}=m \left(\left\{g_n=1\right\}\right).$$
By Fubini's theorem,
\begin{align*}
\lVert u\rVert_{L^p(K;m)}^p&\le p2^{p-1}a^p\sum_{n\in \mathbb{Z}}2^{pn}m \left(\left\{u>2^na\right\}\right)=p2^{p-1}a^p\sum_{n\in \mathbb{Z}}2^{pn}t_n=p2^{2p-1}a^p\sum_{n\in \mathbb{Z}}2^{pn}t_{n+1},\\
\lVert u\rVert_{L^p(K;m)}^p&\ge pa^p\sum_{n\in \mathbb{Z}}2^{pn}m \left(\left\{u>2^{n+1}a\right\}\right)=pa^p\sum_{n\in \mathbb{Z}}2^{pn}t_{n+1}=p2^{-p}a^p\sum_{n\in \mathbb{Z}}2^{pn}t_{n}.
\end{align*}
Let $C_1=2^{3p}$, then by the above estimates,
\begin{align*}
&\lVert u\rVert_{L^p(K;m)}^p\le p2^{2p-1}a^p\left(\sum_{n:t_n\le C_1 t_{n+1}}+\sum_{n:t_n> C_1 t_{n+1}}\right)2^{pn}t_{n+1}\\
&\le p2^{2p-1}a^p\sum_{n:t_n\le C_1 t_{n+1}}2^{pn}t_{n+1}+p2^{2p-1}a^p \frac{1}{C_1}\sum_{n:t_n> C_1 t_{n+1}}2^{pn}t_{n}\\
&\le p2^{2p-1}a^p\sum_{n:t_n\le C_1 t_{n+1}}2^{pn}t_{n+1}+\frac{1}{2}\lVert u\rVert_{L^p(K;m)}^p,
\end{align*}
which gives
$$0<\lVert u\rVert_{L^p(K;m)}^p\le p2^{2p}a^p\sum_{n:t_n\le C_1 t_{n+1}}2^{pn}t_{n+1}.$$
Combining this with (\ref{eq_ug3}),
$$\frac{\mathcal{E}(u)}{\lVert u\rVert_{L^p(K;m)}^p}\ge \frac{1}{p2^{2p}}\frac{\sum_{n:t_n\le C_1 t_{n+1}}2^{pn}\mathcal{E}(g_n)}{\sum_{n:t_n\le C_1 t_{n+1}}2^{pn}t_{n+1}}.$$
Recall the elementary fact that if $\{a_i\}_{i\in I}$ and $\{b_i\}_{i\in I}$ are two families of non-negative real numbers, where $I$ is a countable index set, satisfying that $\sum_{i\in I}a_i\in[0,+\infty)$ and $\sum_{i\in I}b_i\in(0,+\infty)$, then there exists $i\in I$ such that $\frac{a_i}{b_i}\le \frac{\sum_{i\in I}a_i}{\sum_{i\in I}b_i}$. Then there exists $n\in \mathbb{Z}$ with $t_n\le C_1 t_{n+1}$, that is,
$$m \left(\left\{g_n>0\right\}\right)\le C_1 m \left(\left\{g_n=1\right\}\right),$$
such that
$$\frac{\mathcal{E}(u)}{\lVert u\rVert_{L^p(K;m)}^p}\ge \frac{1}{p2^{2p}}\frac{\mathcal{E}(g_n)}{t_{n+1}}=\frac{1}{p2^{2p}}\frac{\mathcal{E}(g_n)}{m \left(\left\{g_n=1\right\}\right)}.$$
Letting $g=g_n$, we have the desired result with $C_2=p2^{2p}$.
\end{proof}

\begin{lemma}\label{lem_gh}
Let $g\in \mathcal{F}\cap C(K)$ satisfy $0\le g\le 1$ on $K$, and
$$m \left(\left\{g=1\right\}\right)\in(0,\delta),\qquad m \left(\left\{g>0\right\}\right)\le C m \left(\left\{g=1\right\}\right),$$
where $\delta\in(0,1)$ and $C\ge1$ are constants. Then there exists $w\in W_*$ such that $h:=g\circ F_w$ satisfies that
$$m \left(\left\{h=1\right\}\right)\ge \delta,\qquad m \left(\left\{h>0\right\}\right)\le 32C\delta,\qquad \mathcal{E}(h)\le 32\delta \frac{\mathcal{E}(g)}{m \left(\left\{g=1\right\}\right)}.$$
\end{lemma}

\begin{proof}
For any $n\ge0$, let $\mathcal{F}_n$ be the $\sigma$-algebra generated by $\{\Sigma_w:w\in W_n\}$. Let $\mathcal{F}_\infty$ be the $\sigma$-algebra generated by $\cup_{n\ge0}\mathcal{F}_n$, then $\{\mathcal{F}_n\}_{n\ge0}$ is a filtration on $(\Sigma,\mathcal{F}_\infty)$. Let $\mathbb{P}$ be the probability measure on $(\Sigma,\mathcal{F}_\infty)$ satisfying that $\mathbb{P}(\Sigma_w)=\frac{1}{8^{\lvert w\rvert}}$ for any $w\in W_*$, then obviously $\pi_* \mathbb{P}=m$ on $K$. Denote by $\mathbb{E}$ the expectation with respect to $\mathbb{P}$.

Let $A=\{g=1\}$, $B=\{g>0\}$, and $\widehat{A}=\pi^{-1}(A)$. Since $g\in C(K)$, $\widehat{A}\in \mathcal{F}_\infty$. For any $n\ge0$, let
$$X_n:=\mathbb{E}\left[1_{\widehat{A}}|\mathcal{F}_n\right].$$
Then $\{X_n,\mathcal{F}_n\}_{n\ge0}$ is a non-negative uniformly integrable martingale. By the martingale convergence theorem,
\begin{equation}\label{eq_gh1}
X_n\to \mathbb{E}\left[1_{\widehat{A}}|\mathcal{F}_\infty\right]=1_{\widehat{A}}\ \mathbb{P}\text{-a.s. and in }L^1.
\end{equation}
For any $n\ge0$, a direct calculation gives
$$X_n=\sum_{v\in W_n}\frac{\mathbb{P}(\widehat{A}\cap \Sigma_v)}{\mathbb{P}(\Sigma_v)}1_{\Sigma_v}=\sum_{v\in W_n}\frac{m({A}\cap K_v)}{m(K_v)}1_{\Sigma_v}=\sum_{v\in W_n}a_v1_{\Sigma_v},$$
where $a_v:=\frac{m({A}\cap K_v)}{m(K_v)}$; in particular, $X_0\equiv a_\emptyset=m(A)\in(0,\delta)$. For $w\in \Sigma$, let
\begin{align*}
T(w):=\inf \left\{n\ge0:X_n(w)\ge \delta\right\}=\inf \left\{n\ge1:a_{[w]_1}<\delta,\ldots,a_{[w]_{n-1}}<\delta,a_{[w]_n}\ge \delta\right\},
\end{align*}
then $T$ is a stopping time with respect to $\{\mathcal{F}_n\}_{n\ge0}$ and $T\ge1$ $\mathbb{P}$-a.s. By (\ref{eq_gh1}), for $\mathbb{P}$-a.s. $w\in \widehat{A}$, $X_n(w)\to 1_{\widehat{A}}(w)=1>\delta$, hence $T(w)<+\infty$. By the optional stopping theorem,
\begin{align*}
&m(A)=\mathbb{E}X_0=\mathbb{E}X_T=\sum_{n\ge1}\mathbb{E}[X_n1_{T=n}]+\mathbb{E}[1_{\widehat{A}\cap\{ T=\infty\}}]=\sum_{n\ge1}\mathbb{E}[X_n1_{T=n}]\\
&=\sum_{n\ge1}\sum_{v\in W_n}a_v \mathbb{P}(\Sigma_v\cap \{T=n\})=\sum_{n\ge1}\sum_{\substack{v\in W_n\\a_{[v]_1}<\delta,\ldots, a_{[v]_{n-1}}<\delta,a_v\ge \delta}}m(A\cap K_v)=\sum_{v\in V}m(A\cap K_v),
\end{align*}
where
\begin{align*}
&V:=\bigcup_{n\ge1}\left\{v\in W_n:a_{[v]_1}<\delta,\ldots, a_{[v]_{n-1}}<\delta,a_v\ge \delta\right\}\\
&=\left\{v\in \cup_{n\ge 1}W_n:a_{[v]_1}<\delta,\ldots, a_{[v]_{\lvert v\rvert-1}}<\delta,a_v\ge \delta\right\}.
\end{align*}
It is easy to see that $V$ is a countable set, and for any distinct $v^{(1)},v^{(2)}\in V$, exactly one of the following holds:
\begin{itemize}
\item $\lvert v^{(1)}\rvert=\lvert v^{(2)}\rvert$.
\item $\lvert v^{(1)}\rvert<\lvert v^{(2)}\rvert$ and $v^{(1)}\ne [v^{(2)}]_{\lvert v^{(1)}\rvert}$.
\item $\lvert v^{(2)}\rvert<\lvert v^{(1)}\rvert$ and $v^{(2)}\ne [v^{(1)}]_{\lvert v^{(2)}\rvert}$,
\end{itemize}
hence $\{\Sigma_v:v\in V\}$ is a pairwise disjoint family, which gives
\begin{equation}\label{eq_gh2}
\sum_{v\in V}m \left(B\cap K_v\right)\le m(B)\le Cm (A),\qquad\sum_{v\in V}\rho^{\lvert v\rvert}\mathcal{E}(g\circ F_v)\le \mathcal{E}(g).
\end{equation}

Let $\mu$ be the probability measure on $V$ given by $\mu(\{v\})=\frac{m(A\cap K_v)}{m(A)}>0$ for any $v\in V$, then (\ref{eq_gh2}) gives
$$\sum_{v\in V}\frac{m \left(B\cap K_v\right)}{m \left(A\cap K_v\right)}\mu(\{v\})\le  C,\qquad\sum_{v\in V}\frac{\rho^{\lvert v\rvert}\mathcal{E}(g\circ F_v)}{m \left(A\cap K_v\right)}\mu(\{v\})\le \frac{\mathcal{E}(g)}{m(A)}.$$
By Markov's inequality,
$$\mu \left(\left\{v\in V:\frac{m \left(B\cap K_v\right)}{m \left(A\cap K_v\right)}>4C\right\}\right)\le \frac14,\qquad\mu \left(\left\{v\in V:\frac{\rho^{\lvert v\rvert}\mathcal{E}(g\circ F_v)}{m \left(A\cap K_v\right)}>4\frac{\mathcal{E}(g)}{m(A)}\right\}\right)\le \frac14,$$
which gives
$$\mu \left(\left\{v\in V:\frac{m \left(B\cap K_v\right)}{m \left(A\cap K_v\right)}\le4C,\ \frac{\rho^{\lvert v\rvert}\mathcal{E}(g\circ F_v)}{m \left(A\cap K_v\right)}\le4\frac{\mathcal{E}(g)}{m(A)}\right\}\right)\ge \frac12>0,$$
hence there exists $v\in V\subseteq W_*$ such that
\begin{equation}\label{eq_gh3}
\frac{m \left(B\cap K_v\right)}{m \left(A\cap K_v\right)}\le4C,\qquad\frac{\rho^{\lvert v\rvert}\mathcal{E}(g\circ F_v)}{m \left(A\cap K_v\right)}\le4\frac{\mathcal{E}(g)}{m(A)}.
\end{equation}
By definition, for any $v\in V$, $a_{v}\ge \delta$ and $a_{[v]_{\lvert v\rvert-1}}<\delta$, that is,
$$\frac{m(A\cap K_v)}{m(K_v)}\ge\delta,\qquad \frac{m(A\cap K_{[v]_{\lvert v\rvert-1}})}{m(K_{[v]_{\lvert v\rvert-1}})}<\delta,$$
hence
\begin{equation}\label{eq_gh4}
\delta m(K_v)\le m(A\cap K_v)\le m(A\cap K_{[v]_{\lvert v\rvert-1}})<\delta m(K_{[v]_{\lvert v\rvert-1}})=8 \delta m(K_v).
\end{equation}

Let $h=g\circ F_v$, then combining (\ref{eq_gh3}) and (\ref{eq_gh4}),
$$m \left(\left\{h=1\right\}\right)=\frac{m \left(A\cap K_v\right)}{m(K_v)}\ge \delta,$$
\begin{align*}
&m \left(\left\{h>0\right\}\right)=\frac{m \left(B\cap K_v\right)}{m(K_v)}\le 4C\frac{m \left(A\cap K_v\right)}{m(K_v)}<32C\delta,
\end{align*}
$$\mathcal{E}(h)\le \frac{4}{\rho^{\lvert v\rvert}}m(A\cap K_v)\frac{\mathcal{E}(g)}{m(A)}\overset{(\star)}{\scalebox{2}[1]{$<$}} 32 \delta\cdot 2^{\lvert v\rvert}m(K_v)\frac{\mathcal{E}(g)}{m(A)}\le32\delta \frac{\mathcal{E}(g)}{m \left(\left\{g=1\right\}\right)},$$
where $(\star)$ follows from Lemma \ref{lem_rhochi}.
\end{proof}

\begin{proposition}\label{prop_gap_pre}
Let $\{f_n\}\subseteq \mathcal{F}\cap C(K)$ satisfy $0\le f_n\le 1$ on $K$ for any $n$, and $\lim_{n\to+\infty}\mathcal{E}(f_n)=0$. Then $\lim_{n\to+\infty}\lVert f_n-(f_n)_K\rVert_{L^p(K;m)}=0$.
\end{proposition}

\begin{proof}
Suppose that the result does not hold. Passing to a suitable subsequence, we may assume that there exist constants $a_1,b_1\in [0,1]$ with $a_1<b_1$ and $\delta>0$ such that for any $n$,
$$m \left(\{f_n\le a_1\}\right)\ge \delta,\qquad m \left(\{f_n\ge b_1\}\right)\ge \delta.$$
Take $a_2,b_2\in (a_1,b_1)$ with $a_2<b_2$, and for any $n$, let
$$u_n=(f_n-b_2)_+,\qquad v_n=(a_2-f_n)_+,$$
then $u_n,v_n\in \mathcal{F}\cap C(K)$ are non-negative, $\mathcal{E}(u_n),\mathcal{E}(v_n)\le \mathcal{E}(f_n)\to0$, and
\begin{equation}\label{eq_gap_pre1}
\lVert u_n\rVert_{L^p(K;m)}^p\ge (b_1-b_2)^p\delta>0,\qquad \lVert v_n\rVert_{L^p(K;m)}^p\ge (a_2-a_1)^p\delta>0.
\end{equation}

For any $n$, let
\begin{align*}
&\mathcal{U}_n:=\left\{U:U\text{ is a connected component of }\{u_n>0\}\right\},\\
&\mathcal{V}_n:=\left\{V:V\text{ is a connected component of }\{v_n>0\}\right\},
\end{align*}
and for any $i=0,2,4,6$, let
\begin{align*}
&\mathcal{U}_{n,i}:=\left\{U\in \mathcal{U}_n:U\cap K_i^{\mathrm{int}}=\emptyset\right\},\\
&\mathcal{V}_{n,i}:=\left\{V\in \mathcal{V}_n:V\cap K_i^{\mathrm{int}}=\emptyset\right\}.
\end{align*}

For fixed $n$, assume that $\cup_{i=0,2,4,6}\mathcal{U}_{n,i}\ne \mathcal{U}_n$ and $\cup_{i=0,2,4,6}\mathcal{V}_{n,i}\ne \mathcal{V}_n$, then there exist $U\in \mathcal{U}_n\backslash \cup_{i=0,2,4,6}\mathcal{U}_{n,i}$ and $V\in \mathcal{V}_n\backslash \cup_{i=0,2,4,6}\mathcal{V}_{n,i}$, that is, $U$ is a component of $\{f_n>b_2\}$ that intersects $K_i^{\mathrm{int}}$ for all $i=0,2,4,6$, and $V$ is a component of $\{f_n<a_2\}$ that intersects $K_i^{\mathrm{int}}$ for all $i=0,2,4,6$. By Lemma \ref{lem_inter4},
$$\mathcal{E}(f_n)\ge 2\rho(b_2-a_2)^p\chi.$$

Since $\lim_{n\to+\infty}\mathcal{E}(f_n)=0$, for any sufficiently large $n$, we have either $\cup_{i=0,2,4,6}\mathcal{U}_{n,i}= \mathcal{U}_n$ or $\cup_{i=0,2,4,6}\mathcal{V}_{n,i}=\mathcal{V}_n$. Passing to a subsequence, we may assume that for any sufficiently large $n$, $\cup_{i=0,2,4,6}\mathcal{U}_{n,i}= \mathcal{U}_n$, then for $i=0,2,4,6$, let $U_{n,i}=\cup_{U\in \mathcal{U}_{n,i}}U$ and $u_{n,i}=u_n1_{U_{n,i}}$. By Lemma \ref{lem_restrict}, $u_{n,i}\in \mathcal{F}\cap C(K)$ and $\mathcal{E}(u_{n,i})\le \mathcal{E}(u_n)\to0$. Obviously, $0\le u_{n,i}\le 1$ on $K$ and $u_{n,i}=0$ on $K_i^{\mathrm{int}}$, where $m(K_i^{\mathrm{int}})=\frac{1}{8}>0$, hence Lemma \ref{lem_bdd_Lp} yields $\lVert u_{n,i}\rVert_{L^p(K;m)}\to0$. Since $\cup_{i=0,2,4,6}\mathcal{U}_{n,i}= \mathcal{U}_n$, $\{u_n>0\}=\cup_{i=0,2,4,6}U_{n,i}$ and
$$\int_{K}u_n^p \mathrm{d}m=\int_{\{u_n>0\}}u_n^p \mathrm{d}m\le \sum_{i=0,2,4,6}\int_{U_{n,i}}u_n^p \mathrm{d}m=\sum_{i=0,2,4,6}\int_K u_{n,i}^p \mathrm{d}m\to0,$$
contradicting (\ref{eq_gap_pre1}).
\end{proof}

\begin{lemma}\label{lem_restrict}
Let $f\in \mathcal{F}\cap C(K)$ be non-negative. With respect to the relative topology on $(K,d)$, let $U$ be a union of connected components of the relatively open set $\{f>0\}$. Then $f1_U\in \mathcal{F}\cap C(K)$ and $\mathcal{E}(f1_U)\le \mathcal{E}(f)$.
\end{lemma}

\begin{proof}
In what follows, we may assume that the sets are non-empty; otherwise, the result holds trivially. Let $V=\{f>0\}\backslash U$, then $U$ and $V$ are relatively open subsets of $K$. For any $n\ge1$, by \ref{item_SSEf}, $(f-\frac{1}{n})_+\in \mathcal{F}\cap C(K)$ and $\mathcal{E}((f-\frac{1}{n})_+)\le \mathcal{E}(f)$. Moreover, $\{f\ge \frac{1}{n}\}\subseteq \{f>0\}=U\sqcup V$ is compact, hence $\{f\ge \frac{1}{n}\}\cap U=\{f\ge \frac{1}{n}\}\backslash V$ is compact. Since $K\backslash U$ is compact and $(\{f\ge \frac{1}{n}\}\cap U)\cap (K\backslash U)=\emptyset$, $\mathrm{dist}(\{f\ge \frac{1}{n}\}\cap U,K\backslash U)>0$. By \ref{item_SSEe}, $f_n:=(f-\frac{1}{n})_+1_{U}\in \mathcal{F}\cap C(K)$ and $\mathcal{E}(f_n)\le \mathcal{E}((f-\frac{1}{n})_+)\le \mathcal{E}(f)$, hence $\{f_n\}\subseteq \mathcal{F}\cap C(K)$ is $\mathcal{E}$-bounded. Since $\{f_n\}$ converges uniformly to $f1_U$, and hence also in $L^p(K;m)$, by Lemma \ref{lem_BAM}, $f1_U\in \mathcal{F}\cap C(K)$, there exists a subsequence, still denoted by $\{f_n\}$, that $\mathcal{E}_1$-weakly-converges to $f1_U$, and $\mathcal{E}(f1_U)\le \varliminf_{n\to+\infty}\mathcal{E}(f_n)\le \mathcal{E}(f)$.
\end{proof}

\begin{lemma}\label{lem_bdd_Lp}
Let $A\in \mathcal{B}(K)$ satisfy $m(A)>0$. Let $M\in(0,+\infty)$ be a constant. Let $\{f_n\}\subseteq \mathcal{F}$ satisfy $0\le f_n\le M$ on $K$ and $f_n=0$ $m$-a.e. on $A$, for any $n$, and $\lim_{n\to+\infty}\mathcal{E}(f_n)=0$. Then $\lim_{n\to+\infty}\lVert f_n\rVert_{L^p(K;m)}=0$.
\end{lemma}

\begin{proof}
Suppose that there exist a subsequence, still denoted by $\{f_n\}$, and a constant $\varepsilon>0$ such that $\lVert f_n\rVert_{L^p(K;m)}\ge \varepsilon>0$ for any $n$. By assumption, $\{f_n\}$ is $\mathcal{E}_1$-bounded, then Lemma \ref{lem_BAM} yields that there exists a subsequence, still denoted by $\{f_n\}$, that $\mathcal{E}_1$-weakly-converges to some $f\in \mathcal{F}$, and $\mathcal{E}(f)\le\varliminf_{n\to+\infty}\mathcal{E}(f_n)=0$. By \ref{item_SSEa}, there exists a constant $c \in \mathbb{R}$ such that $f=c$ $m$-a.e. on $K$. For any $B\in \mathcal{B}(K)$, H\"older's inequality gives that $u\mapsto \int_B u \mathrm{d}m$ is a bounded linear functional on $(\mathcal{F},\mathcal{E}_1^{\frac{1}{p}})$. Together with the $\mathcal{E}_1$-weak-convergence of $\{f_n\}$ to $f$, we have
$$\lim_{n\to+\infty}\int_Bf_n \mathrm{d}m=\int_Bf \mathrm{d}m\text{ for any }B\in \mathcal{B}(K).$$
In particular, $0=\int_{A}f_n \mathrm{d}m\to \int_Af \mathrm{d}m=cm(A)$. Since $m(A)>0$, $f=c=0$ $m$-a.e. on $K$. Moreover, $\lim_{n\to+\infty}\int_Kf_n \mathrm{d}m=\int_K f \mathrm{d}m=0$, hence $\{f_n\}$ is $L^1(K;m)$-convergent to $0$. Since $0\le f_n^p\le M^{p-1}f_n$ on $K$, $0\le \int_K f_n^p \mathrm{d}m\le M^{p-1}\int_K f_n \mathrm{d}m\to0$, a contradiction.
\end{proof}

A path in $K$ is a continuous map $\gamma:[0,1]\to K$. We say that $\gamma$ connects two subsets $A,B\subseteq K$ if $\gamma(0)\in A$ and $\gamma(1)\in B$. We also use $\gamma$ to denote its image $\gamma([0,1])$ if no ambiguity arises. Obviously, the Sierpi\'nski carpet $K$ is locally path connected, hence components of relatively open subsets are open, and connected relatively open subsets are path connected.

\begin{lemma}\label{lem_inter4}
Let $f\in \mathcal{F}\cap C(K)$ and let $a,b\in \mathbb{R}$ satisfy $a<b$. Assume that both of the following conditions hold:
\begin{itemize}
\item A connected component of $\{f>b\}$ intersects $K_i^{\mathrm{int}}$ for all $i=0,2,4,6$.
\item A connected component of $\{f<a\}$ intersects $K_i^{\mathrm{int}}$ for all $i=0,2,4,6$.
\end{itemize}
Then $\mathcal{E}(f)\ge 2\rho(b-a)^p\chi$, where $\chi\in(0,+\infty)$ is as in Lemma \ref{lem_rhochi}.
\end{lemma}

\begin{proof}
For simplicity, write $K_8=K_0$. For $i=1,3,5,7$, we say that $K_i\cap K_{i-1}$ and $K_i\cap K_{i+1}$, which are two closed line segments, are connection sides of $K_i$, see Figure \ref{fig_inter4}.

\begin{figure}[ht]
  \centering
  \begin{tikzpicture}[scale=3]
    \draw (0,0) rectangle (1,1);
    \foreach \t in {1,2} {
      \draw ({\t/3},0) -- ({\t/3},1);
      \draw (0,{\t/3}) -- (1,{\t/3});
    }

    \foreach \i/\x/\y in {
      0/1/1, 1/3/1, 2/5/1, 3/5/3,
      4/5/5, 5/3/5, 6/1/5, 7/1/3
    } {
      \node at ({\x/6},{\y/6}) {$K_{\i}$};
    }
  \end{tikzpicture}
  \caption{The eight level-$1$ cells}
  \label{fig_inter4}
\end{figure}

By assumption, we can choose connected components $U$ and $V$ of $\{f>b\}$ and $\{f<a\}$, respectively, such that $U\cap K_i^{\mathrm{int}}\ne \emptyset$ and $V\cap K_i^{\mathrm{int}}\ne\emptyset$ for all $i=0,2,4,6$. Let
\begin{align*}
&I_1:=\left\{i\in\{1,3,5,7\}:\text{$\exists$ a path $\gamma\subseteq K_i\cap U$ connecting the connection sides of $K_i$}\right\},\\
&I_2:=\left\{i\in\{1,3,5,7\}:\text{$\exists$ a path $\gamma\subseteq K_i\cap V$ connecting the connection sides of $K_i$}\right\},
\end{align*}
then $\# I_1\ge3$ and $\# I_2\ge 3$, hence $\#(I_1\cap I_2)\ge2$. For any $i\in I_1\cap I_2$, there exist paths $\gamma^{(1)}\subseteq K_i\cap U$ and $\gamma^{(2)}\subseteq K_i\cap V$ connecting the connection sides of $K_i$, then $f>b$ on $\gamma^{(1)}$ and $f<a$ on $\gamma^{(2)}$. By \ref{item_SSEd}, \ref{item_SSEf}, and an application of Lemma \ref{lem_path01} to the function $\frac{(f\circ F_i-a)_+}{b-a}\wedge1\in \mathcal{F}\cap C(K)$,
$$\frac{1}{(b-a)^p}\mathcal{E}(f\circ F_i)\ge \mathcal{E}\left(\frac{(f\circ F_i-a)_+}{b-a}\wedge1\right)\ge\chi.$$
Then by \ref{item_SSEe},
$$\mathcal{E}(f)\ge\rho\sum_{i\in I_1\cap I_2}\mathcal{E}(f\circ F_i)\ge \rho\cdot(b-a)^p \chi\cdot\#(I_1\cap I_2)\ge 2\rho(b-a)^p \chi.$$
\end{proof}

\begin{lemma}\label{lem_path01}
Let $f\in \mathcal{F}\cap C(K)$. If there exist paths $\gamma^{(1)}$ and $\gamma^{(2)}$ in $K$, both of which connect $\ell_L$ and $\ell_R$, such that $f=1$ on $\gamma^{(1)}$ and $f=0$ on $\gamma^{(2)}$, then $\mathcal{E}(f)\ge \chi$.
\end{lemma}

\begin{proof}
The proof relies heavily on planarity. By replacing $f$ by $(f\vee0)\wedge 1$ and using \ref{item_SSEf}, we may assume that $0\le f\le1$ on $K$. Obviously, $\gamma^{(1)}\cap \gamma^{(2)}=\emptyset$, see Figure \ref{fig_path01}.

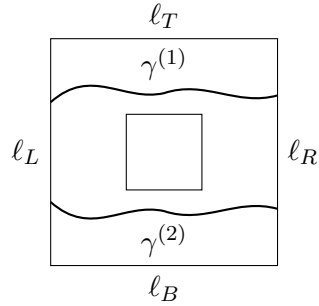
\begin{figure}[htbp]
  \centering
  \begin{tikzpicture}[scale=3]
    \draw (0,0) rectangle (1,1);

    \draw ({1/3},{1/3}) rectangle ({2/3},{2/3});

    \draw[thick]
      (0,0.72)
      .. controls (0.18,0.88) and (0.32,0.72) .. (0.50,0.76)
      .. controls (0.68,0.82) and (0.82,0.70) .. (1,0.75);

    \draw[thick]
      (0,0.28)
      .. controls (0.18,0.12) and (0.32,0.28) .. (0.50,0.24)
      .. controls (0.68,0.18) and (0.82,0.30) .. (1,0.25);

    \node at (0.50,0.89) {$\gamma^{(1)}$};
    \node at (0.50,0.11) {$\gamma^{(2)}$};
    \node[left] at (0,0.5) {$\ell_L$};
    \node[right] at (1,0.5) {$\ell_R$};
    \node[above] at (0.5,1) {$\ell_T$};
    \node[below] at (0.5,0) {$\ell_B$};
  \end{tikzpicture}
  \caption{Two paths $\gamma^{(1)}$ and $\gamma^{(2)}$}
  \label{fig_path01}
\end{figure}

Let $U$ be the connected component of $\{f>0\}$ that contains $\gamma^{(1)}$, then either $U\cap \ell_B=\emptyset$ or $U\cap \ell_T=\emptyset$; otherwise, there exists a path $\gamma\subseteq U$ connecting $\ell_B$ and $\ell_T$, by planarity, $\gamma\cap \gamma^{(2)}\ne\emptyset$, a contradiction. Without loss of generality, we may assume that $U\cap \ell_B=\emptyset$, let $g=f1_U$, then $g=1$ on $\gamma^{(1)}$ and $g=0$ on $\ell_B$. By Lemma \ref{lem_restrict}, $g\in \mathcal{F}\cap C(K)$ and $\mathcal{E}(g)\le \mathcal{E}(f)$.

Let $V$ be the connected component of $\{g<1\}=\{1-g>0\}$ that contains $\ell_B$, then $V\cap \ell_T=\emptyset$; otherwise, there exists a path $\gamma\subseteq V$ connecting $\ell_B$ and $\ell_T$, by planarity again, $\gamma\cap \gamma^{(1)}\ne \emptyset$, a contradiction. Let $h=(1-g)1_V$, then $h=1$ on $\ell_B$ and $h=0$ on $\ell_T$. Since $1-g\in \mathcal{F}\cap C(K)$, by Lemma \ref{lem_restrict} again, $h\in \mathcal{F}\cap C(K)$ and $\mathcal{E}(h)\le \mathcal{E}(g)$. By \ref{item_SSEd} and Lemma \ref{lem_rhochi},
$$\chi\le \mathcal{E}(h)\le \mathcal{E}(g)\le \mathcal{E}(f).$$
\end{proof}

\begin{lemma}\label{lem_BAM}
Let $\{u_n\}_n\subseteq \mathcal{F}$ be $\mathcal{E}_1$-bounded, then there exists a subsequence, still denoted by $\{u_n\}_n$, that $\mathcal{E}_1$-weakly-converges to some $v\in \mathcal{F}$. Moreover, for any $n$, there exist an integer $I_n\ge n$ and coefficients $\lambda^{(n)}_k\ge0$, $k=n,\ldots,I_n$, such that $\sum_{k=n}^{I_n}\lambda^{(n)}_k=1$, and $\{\sum_{k=n}^{I_n}\lambda^{(n)}_ku_k\}_n$ is $\mathcal{E}_1$-convergent to $v$, which gives
$$\mathcal{E}_1(v)\le\varliminf_{n\to+\infty}\mathcal{E}_1(u_n),\qquad \mathcal{E}(v)\le\varliminf_{n\to+\infty}\mathcal{E}(u_n).$$
In particular, if $\{u_n\}_n\subseteq \mathcal{F}$ is $\mathcal{E}$-bounded and is $L^p(K;m)$-convergent to some $u\in L^p(K;m)$, then $u=v\in \mathcal{F}$.
\end{lemma}

\begin{proof}
By \ref{item_SSEb} and the Milman--Pettis theorem (see \cite[Theorem 3.31]{Bre11}), $(\mathcal{F},\mathcal{E}_1^{\frac{1}{p}})$ is a reflexive Banach space. The results then follow directly from the Banach--Alaoglu theorem (see \cite[Theorems 3.16 and 3.18]{Bre11}) together with Mazur’s lemma, for which we refer to \cite[Theorem 2 in Section V.1]{Yos95}.
\end{proof}

We now give the proof of Proposition \ref{prop_gap} as follows.

\begin{proof}[Proof of Proposition \ref{prop_gap}]
By \ref{item_SSEc}, it suffices to prove the result for functions in $\mathcal{F}\cap C(K)$. Suppose that the result does not hold, then there exists $\{v_n\}\subseteq \mathcal{F}\cap C(K)$ such that $\lVert v_n-(v_n)_K\rVert_{L^p(K;m)}=1$ for any $n$, and $\lim_{n\to+\infty}\mathcal{E}(v_n)=0$.

For any $n$, there exists $c_n \in\mathbb{R}$ such that
\begin{equation}\label{eq_median}
m \left(\left\{v_n\ge c_n\right\}\right)\ge \frac{1}{2},\qquad m \left(\left\{v_n\le c_n\right\}\right)\ge \frac{1}{2}.
\end{equation}
By H\"older's inequality,
$$1=\int_{K}\lvert v_n-(v_n)_K\rvert^p \mathrm{d}m\le 2^p\int_K \lvert v_n-c_n\rvert^p \mathrm{d}m=2^p\int_K (v_n-c_n)_+^p \mathrm{d}m+2^p\int_K (v_n-c_n)_-^p \mathrm{d}m,$$
then either $\lVert (v_n-c_n)_+\rVert_{L^p(K;m)}^p\ge \frac{1}{2^{p+1}}$ or $\lVert (v_n-c_n)_-\rVert_{L^p(K;m)}^p\ge \frac{1}{2^{p+1}}$. Let
\begin{equation*}
u_n:=
\begin{cases}
(v_n-c_n)_+&\text{if }\lVert (v_n-c_n)_+\rVert_{L^p(K;m)}^p\ge \frac{1}{2^{p+1}},\\
(v_n-c_n)_-&\text{otherwise},
\end{cases}
\end{equation*}
then $u_n\in \mathcal{F}\cap C(K)$ is non-zero and non-negative, $\mathcal{E}(u_n)\le \mathcal{E}(v_n)\to0$, $\lVert u_n\rVert_{L^p(K;m)}^p\ge \frac{1}{2^{p+1}}$, and by (\ref{eq_median}),
\begin{equation}\label{eq_median_un}
m \left(\left\{u_n>0\right\}\right)\le \frac{1}{2},\qquad m \left(\left\{u_n=0\right\}\right)\ge \frac{1}{2}.
\end{equation}

Let $C_1,C_2>1$ be the constants in Lemma \ref{lem_ug}. Applying this lemma to $u_n$, we obtain $g_n\in \mathcal{F}\cap C(K)$ satisfying that $0\le g_n\le 1$ on $K$, $\{g_n>0\}\subseteq \{u_n>0\}$, and
$$0<m \left(\left\{g_n>0\right\}\right)\le C_1 m \left(\left\{g_n=1\right\}\right),$$
\begin{equation}\label{eq_Emgn}
\frac{\mathcal{E}(g_n)}{m \left(\left\{g_n=1\right\}\right)}\le C_2 \frac{\mathcal{E}(u_n)}{\lVert u_n\rVert^p_{L^p(K;m)}}\le 2^{p+1}C_2 \mathcal{E}(u_n)\to0.
\end{equation}
Fix $\delta\in(0,\frac{1}{64C_1})$. If $m (\{g_n=1\})<\delta$, then by Lemma \ref{lem_gh} with $C=C_1$, there exists $h_n\in \mathcal{F}\cap C(K)$ such that $0\le h_n\le1$ on $K$, and
$$m \left(\left\{h_n=1\right\}\right)\ge \delta,\qquad m \left(\left\{h_n>0\right\}\right)\le 32C_1\delta\le \frac{1}{2},\qquad \mathcal{E}(h_n)\le 32\delta \frac{\mathcal{E}(g_n)}{m \left(\left\{g_n=1\right\}\right)}\to0,$$
hence $m(\{h_n=0\})\ge \frac{1}{2}$. If $m (\{g_n=1\})\ge\delta$, then let $h_n=g_n$. Obviously, $m(\{h_n=1\})=m(\{g_n=1\})\ge\delta$. Since $\{g_n>0\}\subseteq \{u_n>0\}$, by (\ref{eq_median_un}),
$$m (\{h_n=0\})=m (\{g_n=0\})\ge m(\{u_n=0\})\ge\frac{1}{2}.$$
Moreover, by (\ref{eq_Emgn}),
$$\mathcal{E}(h_n)=\mathcal{E}(g_n)\le \frac{\mathcal{E}(g_n)}{m(\{g_n=1\})}\to0.$$
Therefore, in both cases, $h_n\in \mathcal{F}\cap C(K)$ satisfies that $0\le h_n\le 1$ on $K$, and
$$m \left(\left\{h_n=1\right\}\right)\ge \delta,\qquad m \left(\left\{h_n=0\right\}\right)\ge \frac{1}{2},\qquad \mathcal{E}(h_n)\to0.$$
On the one hand, by Proposition \ref{prop_gap_pre},
$$\int_{K}\lvert h_n-(h_n)_K\rvert^p \mathrm{d}m\to0.$$
On the other hand,
\begin{align*}
&\int_{K}\lvert h_n-(h_n)_K\rvert^p \mathrm{d}m\ge\int_{\{h_n=1\}}(1-(h_n)_K)^p \mathrm{d}m+\int_{\{h_n=0\}}((h_n)_K)^p \mathrm{d}m\\
&\ge \inf_{a\in[0,1]}\left(\delta(1-a)^p+\frac{1}{2}a^p\right)=\frac{\delta}{\left(1+(2\delta)^{\frac{1}{p-1}}\right)^{p-1}}>0,
\end{align*}
a contradiction.
\end{proof}

\section{Some results on harmonic functions}\label{sec_harmonic}

Let $(K,d,m)$ be the Sierpi\'nski carpet. The canonical $p$-energy $(\mathcal{E}^{(c)},\mathcal{F}^{(c)})$ satisfies the generalized $p$-contraction property introduced in \cite{KS24aV5} (see \cite[the Remark following Theorem 1.1]{MS25}, \cite[Remark 8.20 and Proposition 5.23(a)]{KS24aV5}), which in turn implies the following conditions.

\begin{itemize}
\item \textbf{Clarkson’s inequality} (see \cite[Theorem 6.22(i), the construction (8.14) in the proof of Theorem 1.1, and the Remark following Theorem 1.2]{MS25}, \cite[Proposition 2.3(e) and (f), and Proposition 5.10(e)]{KS24aV5}): For any $f,g\in\mathcal{F}^{(c)}$ and any $A\in \mathcal{B}(K)$, we have
\begin{equation*}\label{eq_Cla}\tag*{Cla}
\begin{cases}
\mathcal{E}^{(c)}(f+g)+\mathcal{E}^{(c)}(f-g)\ge2 \left(\mathcal{E}^{(c)}(f)^{\frac{1}{p-1}}+\mathcal{E}^{(c)}(g)^{\frac{1}{p-1}}\right)^{p-1}&\text{if }p\in(1,2],\\
\mathcal{E}^{(c)}(f+g)+\mathcal{E}^{(c)}(f-g)\le2 \left(\mathcal{E}^{(c)}(f)^{\frac{1}{p-1}}+\mathcal{E}^{(c)}(g)^{\frac{1}{p-1}}\right)^{p-1}&\text{if }p\in[2,+\infty),\\
\end{cases}
\end{equation*}
and
\begin{equation*}\label{eq_ClaGamma}\tag*{$\text{Cla}^{\Gamma^{(c)}}$}
\begin{cases}
\Gamma^{(c)}(f+g)(A)+\Gamma^{(c)}(f-g)(A)\\
\hspace{60pt}\ge2 \left(\Gamma^{(c)}(f)(A)^{\frac{1}{p-1}}+\Gamma^{(c)}(g)(A)^{\frac{1}{p-1}}\right)^{p-1}\text{if }p\in(1,2],\\
\Gamma^{(c)}(f+g)(A)+\Gamma^{(c)}(f-g)(A)\\
\hspace{60pt}\le2 \left(\Gamma^{(c)}(f)(A)^{\frac{1}{p-1}}+\Gamma^{(c)}(g)(A)^{\frac{1}{p-1}}\right)^{p-1}\text{if }p\in[2,+\infty).\\
\end{cases}
\end{equation*}
\item \textbf{Strong sub-additivity} (see \cite[Proposition 2.3(c)]{KS24aV5}): For any $f,g\in\mathcal{F}^{(c)}$, we have $f\vee g, f\wedge g\in\mathcal{F}^{(c)}$ and
\begin{equation*}\label{eq_SubAdd}\tag*{SubAdd}
\mathcal{E}^{(c)}(f\vee g)+\mathcal{E}^{(c)}(f\wedge g)\le\mathcal{E}^{(c)}(f)+\mathcal{E}^{(c)}(g).
\end{equation*}
\item \textbf{$\mathcal{F}^{(c)}\cap L^\infty(K;m)$ is an algebra} (see \cite[Proposition 2.3(d)]{KS24aV5}): For any $f,g\in\mathcal{F}^{(c)}\cap L^\infty(K;m)$, we have $fg\in\mathcal{F}^{(c)}$ and
\begin{equation*}\label{eq_Alg}\tag*{Alg}
\mathcal{E}^{(c)}(fg)^{{1}/{p}}\le \lVert {f}\rVert_{L^\infty(K;m)}\mathcal{E}^{(c)}(g)^{\frac{1}{p}}+\lVert {g}\rVert_{L^\infty(K;m)}\mathcal{E}^{(c)}(f)^{\frac{1}{p}}.
\end{equation*}
\item \textbf{Fr\'echet differentiability} (see \cite[Theorem 3.7, Proposition 4.3, Definition 4.4, Theorem 4.5]{KS24aV5}): For any $f,g\in\mathcal{F}^{(c)}$ and any $A\in \mathcal{B}(K)$, the derivatives
$$\mathcal{E}^{(c)}(f;g)=\frac{1}{p}\frac{\mathrm{d}}{\mathrm{d} t}\mathcal{E}^{(c)}(f+tg)|_{t=0}\in\mathbb{R},\qquad \Gamma^{(c)}(f;g)(A)=\frac{1}{p}\frac{\mathrm{d}}{\mathrm{d} t}\Gamma^{(c)}(f+tg)(A)|_{t=0}\in\mathbb{R}$$
exist, the maps $\mathcal{E}^{(c)}(f;\cdot),\Gamma^{(c)}(f;\cdot)(A):\mathcal{F}^{(c)}\to\mathbb{R}$ are linear, $\mathcal{E}^{(c)}(f;f)=\mathcal{E}^{(c)}(f)$, and $\Gamma^{(c)}(f;f)(A)=\Gamma^{(c)}(f)(A)$. Moreover, for any $f, g\in\mathcal{F}^{(c)}$, $A\in \mathcal{B}(K)$, and $a\in\mathbb{R}$,
$$\mathcal{E}^{(c)}(af;g)=\mathrm{sgn}(a)|a|^{p-1}\mathcal{E}^{(c)}(f;g),\qquad \Gamma^{(c)}(af;g)(A)=\mathrm{sgn}(a)|a|^{p-1}\Gamma^{(c)}(f;g)(A),$$
$$\lvert \mathcal{E}^{(c)}(f;g)\rvert\le\mathcal{E}^{(c)}(f)^{\frac{p-1}{p}}\mathcal{E}^{(c)}(g)^{\frac{1}{p}},\qquad \lvert \Gamma^{(c)}(f;g)(A)\rvert\le \Gamma^{(c)}(f)(A)^{\frac{p-1}{p}}\Gamma^{(c)}(g)(A)^{\frac{1}{p}}.$$
\item\textbf{H\"older-type inequality, Leibniz rule and chain rule} (see \cite[Theorem 4.6, Proposition 4.8, Theorem 4.15(b), Theorem 5.12]{KS24aV5}, \cite[Section 8]{Yan25a}): For any $f,g\in \mathcal{F}^{(c)}$, $\Gamma^{(c)}(f;g)$ is a signed measure on $K$, and for any non-negative Borel functions $\varphi,\psi$ and any $A\in \mathcal{B}(K)$,
\begin{equation}\label{eq_Holder}
\int_A {\varphi}{\psi} \mathrm{d}\lvert \Gamma^{(c)}(f;g)\rvert\le \left(\int_A {\varphi}^{\frac{p}{p-1}}\mathrm{d}\Gamma^{(c)}(f)\right)^{\frac{p-1}{p}}\left(\int_A{\psi}^p \mathrm{d}\Gamma^{(c)}(g)\right)^{\frac{1}{p}},
\end{equation}
where $\lvert \Gamma^{(c)}(f;g)\rvert$ is the total variation of $ \Gamma^{(c)}(f;g)$. For any $f\in \mathcal{F}^{(c)}$ and $g,h\in \mathcal{F}^{(c)}\cap L^\infty(K;m)$,
\begin{equation}\label{eq_Leibniz}
\mathrm{d}\Gamma^{(c)}(f;gh)=\widetilde{g}\mathrm{d}\Gamma^{(c)}(f;h)+\widetilde{h}\mathrm{d}\Gamma^{(c)}(f;g).
\end{equation}
For any $n\ge1$, $f,g_1,\ldots,g_n\in \mathcal{F}^{(c)}$, and $\Phi\in C^1(\mathbb{R}^n)$ with bounded first derivatives, let $\boldsymbol{g}=(g_1,\ldots,g_n)$, then $\Phi(\boldsymbol{g})\in \mathcal{F}^{(c)}$ and
\begin{equation}\label{eq_chain}
\mathrm{d}\Gamma^{(c)}(f;\Phi(\boldsymbol{g}))=\sum_{i=1}^n \frac{\partial \Phi}{\partial x_i}\left(\widetilde{g}_1,\ldots,\widetilde{g}_n\right)\mathrm{d}\Gamma^{(c)}(f;g_i).
\end{equation}
In \cite{KS24aV5}, (\ref{eq_Leibniz}) and (\ref{eq_chain}) are stated for functions in $\mathcal{F}^{(c)}\cap C(K)$. The quasi-continuity and capacity results in \cite[Section 8]{Yan25a}, together with standard approximation and weak compactness, extend these results to the classes of functions considered above, with quasi-continuous versions understood.

\end{itemize}

Moreover, we have the following properties.

\begin{itemize}
\item \textbf{Zero energy measure on boundaries} (see \cite[Remark 8.23]{MS25}): For any $f\in \mathcal{F}^{(c)}$ and any minimal energy-dominant measure $\mu$, we have $\Gamma^{(c)}(f)(K^\partial)=\mu(K^{\partial})=0$, $\Gamma^{(c)}(f)(K^\partial_n)=\mu(K^\partial_n)=0$ for any $n\ge0$, and hence
\begin{equation}\label{eq_bdy_zero}
\Gamma^{(c)}(f)(K^\partial_*)=\mu(K^\partial_*)=0.
\end{equation}

\item \textbf{Full support of minimal energy-dominant measures}: Let $\mu$ be a minimal energy-dominant measure, then $\mu$ has full support. Indeed, for any relatively open subset $U\subseteq K$, take $x\in U$ and $r>0$ such that $B(x,2r)\subseteq U$ and $K\backslash B(x,2r)\ne\emptyset$, then by \cite[Theorem 1.5]{MS25}, there exists $f\in \mathcal{F}^{(c)}\cap C(K)$ such that $f=1$ on $B(x,r)$ and $\mathrm{supp}(f)\subseteq B(x,2r)$, hence $f$ is non-constant, by the strongly local property, $\Gamma^{(c)}(f)(U)=\mathcal{E}^{(c)}(f)>0$, which gives $\mu(U)>0$.
\end{itemize}

For notational convenience, we omit the superscript $(c)$ when no ambiguity arises.

For any relatively open subset $U\subseteq K$, let
$$\mathcal{F}(U):=\left\{u\in \mathcal{F}:\widetilde{u}=0\text{ q.e. on }K\backslash U\right\},\qquad \mathcal{F}_D:=\mathcal{F}(K^{\mathrm{int}}),$$
then \cite[Lemma 6.2]{Yan25c} shows that $\mathcal{F}(U)$ is the $\mathcal{E}_1^{\frac{1}{p}}$-closure of $\{u\in \mathcal{F}\cap C(K):\mathrm{supp}(u)\subseteq U\}$.
We say that a function $h\in\mathcal{F}$ is harmonic in $U$ if $\mathcal{E}(h)\le \mathcal E(h+u)$ for any $u\in\mathcal{F}(U)$.

Let $\mathcal{H}$ be the set of all harmonic functions in $K^{\mathrm{int}}$ and for any $n\ge0$, let
$$\mathcal{H}_n:=\left\{f\in \mathcal{F}:f\circ F_w\in \mathcal{H}\text{ for any }w\in W_n\right\}.$$
Let $\mathcal{H}_*:=\cup_{n\ge0}\mathcal{H}_n$. Roughly speaking, $\mathcal{H}_*$ is the set of piecewise harmonic functions.

The first main result of this section is the following density result.

\begin{lemma}[{$p$-version of \cite[Lemma 3.10]{Hin05}}]\label{lem_Hstar}
We have that $\mathcal{H}_*$ is $\mathcal{E}_1$-dense in $\mathcal{F}$.
\end{lemma}

\begin{proof}
For notational convenience, we may assume that all functions in $\mathcal{F}$ are quasi-continuous. By regularity \ref{item_SSEc}, we may assume that $f\in \mathcal{F}\cap C(K)$. For any $n\ge0$, since $K^{\mathrm{int}}_n=\cup_{w\in W_n}K^{\mathrm{int}}_w$ is a bounded open subset of $K$, by \ref{item_SSEg} or \ref{eq_PI}, together with the same minimizing argument as in \cite[Proposition 4.2]{Yan25c}, there exists $h_n\in \mathcal{F}$, such that $h_n$ is harmonic in $K^{\mathrm{int}}_n$ and ${h}_n={f}$ q.e. on $K\backslash K^{\mathrm{int}}_n=K^\partial_n$, hence $h_n\in \mathcal{H}_n\subseteq \mathcal{H}_*$. For any $w\in W_n$, since $h_n=f$ q.e. on $K^\partial_w$, by the comparison principle (see \cite[Proposition 3.5]{Yan25d}), $\inf_{K^\partial_w}f\le h_n\le \sup_{K^\partial_w}f$. Since $f$ is uniformly continuous on $K$, $\{h_n\}$ converges uniformly to $f$, which also implies $L^p(K;m)$-convergence.

We claim that $\{h_n\}$ is an $\mathcal{E}$-Cauchy sequence. Indeed, for any $n\ge0$, since $\Gamma(h_n)(K^\partial_*)=\Gamma(f)(K^\partial_*)=0$,
$$\mathcal{E}(h_n)=\sum_{w\in W_n}\Gamma(h_n)(K^{\mathrm{int}}_w)\le\sum_{w\in W_n}\Gamma(f)(K^{\mathrm{int}}_w)=\mathcal{E}(f),$$
and
$$\mathcal{E}(h_n)=\sum_{w\in W_n}\Gamma(h_n)(K^{\mathrm{int}}_w)\le\sum_{w\in W_n}\Gamma(h_{n+1})(K^{\mathrm{int}}_w)=\mathcal{E}(h_{n+1}),$$
which implies $\{\mathcal{E}(h_n)\}_n$ is an increasing sequence bounded above by $\mathcal{E}(f)$. Hence $A:=\lim_{n\to+\infty}\mathcal{E}(h_n)\le \mathcal{E}(f)$. For any $n\le m$, since $\Gamma\left(\frac{h_n+h_m}{2}\right)(K^\partial_*)=0$,
$$\mathcal{E}(h_n)=\sum_{w\in W_n}\Gamma(h_n)(K^{\mathrm{int}}_w)\le\sum_{w\in W_n}\Gamma\left(\frac{h_n+h_m}{2}\right)(K^{\mathrm{int}}_w)=\mathcal{E}\left(\frac{h_n+h_m}{2}\right).$$
If $p\in(1,2]$, then by \ref{eq_Cla},
$$2 \left(\mathcal{E}\left(\frac{h_n+h_m}{2}\right)^{\frac{1}{p-1}}+\mathcal{E}\left(\frac{h_n-h_m}{2}\right)^{\frac{1}{p-1}}\right)^{p-1}\le \mathcal{E}(h_n)+\mathcal{E}(h_m)\to 2A,$$
where
\begin{align*}
&\text{LHS}\ge 2 \left(\mathcal{E}\left(h_n\right)^{\frac{1}{p-1}}+\mathcal{E}\left(\frac{h_n-h_m}{2}\right)^{\frac{1}{p-1}}\right)^{p-1}\\
&\to 2 \left(A^{\frac{1}{p-1}}+\varlimsup_{n\le m\to+\infty}\mathcal{E}\left(\frac{h_n-h_m}{2}\right)^{\frac{1}{p-1}}\right)^{p-1},
\end{align*}
hence $\lim_{n\le m\to+\infty}\mathcal{E}\left(\frac{h_n-h_m}{2}\right)=0$. If $p\in[2,+\infty)$, then the same result follows by applying \ref{eq_Cla} again.

Therefore, $\{h_n\}\subseteq \mathcal{H}_*$ is an $\mathcal{E}$-Cauchy sequence, and is $L^p(K;m)$-convergent to $f$, hence $\mathcal{E}_1$-convergent to $f$.
\end{proof}

For any $n\ge0$ and any $A\subseteq W_n$, let $K_A:=\cup_{w\in A}K_w$. By (\ref{eq_bdy_zero}), for any $f\in \mathcal{F}$, any $n\ge0$, and any $A_1,A_2\subseteq W_n$ with $A_1\cap A_2=\emptyset$,
$$\Gamma(f)(K_{A_1\cup A_2})=\Gamma(f)(K_{A_1})+\Gamma(f)(K_{A_2}).$$
For any $w\in W_n$, let $\mathcal{N}_0(w):=\{w\}$, and for any $k\ge0$, let
\begin{equation}\label{eq_k_nbhd}
\mathcal{N}_{k+1}(w):=\left\{v\in W_n:K_v\cap K_{\mathcal{N}_{k}(w)}\ne\emptyset\right\}.
\end{equation}
Roughly speaking, $\mathcal{N}_k(w)$ is the $k$-th neighborhood of $w$ in $W_n$. Obviously, for any $w\in W_*$ and any $k\ge0$, $\# \mathcal{N}_k(w)\le (2k+1)^2$. Moreover, for any $k\ge0$,
\begin{equation}\label{eq_sup_nbhd}
\sup_{n\ge0}\sup_{v\in W_n} \# \left\{w\in W_n:v\in \mathcal{N}_k(w)\right\}\le (2k+1)^2.
\end{equation}
For any $n\ge0$ and any subset $A\subseteq W_n$, let $K_A:=\cup_{w\in A}K_w$, and
\begin{align*}
\mathcal{F}^{0}_A&:=\left\{f\in \mathcal{F}:f=0\ m\text{-a.e. on }K_{W_n\backslash A}\right\},\\
\mathcal{H}(A)&:=\left\{h\in \mathcal{F}:\mathcal{E}(h)\le \mathcal{E}(h+g)\text{ for any }g\in \mathcal{F}^0_A\right\}.
\end{align*}

The second main result of this section is the following compactness result.

\begin{proposition}[{$p$-version of \cite[Proposition 4.18]{Hin13}}]\label{prop_Hin418}
Let $\{h_n\}_{n\ge1}\subseteq \mathcal{F}$ and $\{w^{(n)}\}_{n\ge1}\subseteq W_*$ satisfy $K_{\mathcal{N}_3(w^{(n)})}\cap K^\partial=\emptyset$, $h_n\in \mathcal{H}(\mathcal{N}_3(w^{(n)}))$, and $\int_{K_{w^{(n)}}}h_n \mathrm{d}m=0$, for any $n\ge1$, and
\begin{equation}\label{eq_Hin425}
\sup_{n\ge1}\frac{1}{\rho^{\lvert w^{(n)}\rvert}}\Gamma(h_n) \left(K_{\mathcal{N}_3(w^{(n)})}\right)<+\infty.
\end{equation}
Then there exist a subsequence of $\{h_n\circ F_{w^{(n)}}\}_{n\ge1}$ and a function $h\in \mathcal{H}\cap C(K)$, such that the subsequence is $\mathcal{E}_1$-convergent and uniformly convergent to $h$ on $K$.
\end{proposition}

The remainder of this section is devoted to its proof. The idea is to establish elliptic Harnack inequalities that give suitable oscillation estimates, allowing us to apply the Arzel\`a--Ascoli theorem. Unlike the argument in the linear setting in \cite{Hin13}, our argument will use recent results on non-linear functional inequalities.

For any $f\in \mathcal{F}$, $n\ge0$, and $A\subseteq W_n$, let
$$\mathcal{E}^{A}(f):=\rho^{n}\sum_{w\in A}\mathcal{E}(f\circ F_w).$$
For any $w,v\in W_n$, write $w\sim_n v$ if $K_w\cap K_v$ is a line segment. By \cite[(8.22) and (8.25)]{MS25}, there exists a constant $C>0$ such that for any $n\ge0$ and any $w,v\in W_n$ with $w\sim_n v$,
\begin{equation}\label{eq_P9a}
\int_{K_w}|f-f_{K_w}|^p\mathrm{d} m\le \frac{C}{8^{\lvert w\rvert}}\mathcal{E}(f\circ F_w),
\end{equation}
and
\begin{equation}\label{eq_P9b}
|f_{K_w}-f_{K_v}|^p\le C\left(\mathcal{E}(f\circ F_w)+\mathcal{E}(f\circ F_v)\right).
\end{equation}

\begin{lemma}[{$p$-version of \cite[Lemma 5.20]{Hin13}}]\label{lem_poinc}
There exists a constant $C>0$ such that for any $f\in \mathcal{F}$, $w\in W_*$ with $K_{\mathcal{N}_3(w)}\cap K^\partial=\emptyset$, and $v\in \mathcal{N}_3(w)$,
\begin{gather}
\lvert f_{K_w}-f_{K_v}\rvert^p \le \frac{C}{\rho^{\lvert w\rvert}}\mathcal{E}^{\mathcal{N}_3(w)}(f),\label{eq_Hin520a}\\
\int_{K_{\mathcal{N}_3(w)}}|f-f_{K_w}|^p\mathrm{d} m\le \frac{C}{(8\rho)^{|w|}}\mathcal{E}^{\mathcal{N}_3(w)}(f).\label{eq_Hin520b}
\end{gather}
\end{lemma}

\begin{proof}
Let $n\ge0$ and $w\in W_n$. By (\ref{eq_sup_nbhd}), $\# \mathcal{N}_3(w)\le 49$. For any $v\in \mathcal{N}_3(w)\backslash\{w\}$, by the geometry of the Sierpi\'nski carpet, there exist $v_0,\ldots,v_k\in \mathcal{N}_3(w)$ with $v_0=w$, $v_k=v$, and $v_{i-1}\sim_n v_{i}$ for any $i=1,\ldots,k$. By (\ref{eq_P9b}),
\begin{align*}
&\lvert f_{K_w}-f_{K_v}\rvert^p \le k^{p-1}\sum_{i=1}^k \lvert f_{K_{v_{i-1}}}-f_{K_{v_{i}}}\rvert^p\\
&\lesssim \sum_{i=1}^k\left(\mathcal{E}(f\circ F_{v_{i-1}})+\mathcal{E}(f\circ F_{v_i})\right)\le 2\sum_{v\in \mathcal{N}_3(w)}\mathcal{E}(f\circ F_v),
\end{align*}
hence we obtain (\ref{eq_Hin520a}). Then by (\ref{eq_P9a}),
$$\int_{K_v}\lvert f-f_{K_w}\rvert^p \mathrm{d}m\le 2^{p-1}\left(\int_{K_v}\lvert f-f_{K_v}\rvert^p \mathrm{d}m+\frac{1}{8^n}\lvert f_{K_w}-f_{K_v}\rvert^p \right)\lesssim \frac{1}{8^n}\sum_{v\in \mathcal{N}_3(w)}\mathcal{E}(f\circ F_v),$$
hence
\begin{align*}
&\int_{K_{\mathcal{N}_3(w)}}|f-f_{K_w}|^p\mathrm{d} m=\int_{K_{w}}|f-f_{K_w}|^p\mathrm{d} m+\sum_{v\in \mathcal{N}_3(w)\backslash\{w\}}\int_{K_{v}}|f-f_{K_w}|^p\mathrm{d} m\\
&\lesssim \frac{{2^{p+2}48^{p}C}}{8^n}\sum_{v\in \mathcal{N}_3(w)}\mathcal{E}(f\circ F_v)=\frac{{2^{p+2}48^{p}C}}{(8\rho)^n}\mathcal{E}^{\mathcal{N}_3(w)}(f).
\end{align*}
\end{proof}

\begin{lemma}[{$p$-version of {\cite[Lemmas 5.12 and 5.18]{Hin13}}}]\label{lem_poinc0}
There exists a constant $C>0$ such that for any $n\ge0$, $w\in W_n$ with $K_{\mathcal{N}_3(w)}\cap K^\partial=\emptyset$, and $f\in\mathcal{F}^{0}_{\mathcal{N}_3(w)}$,
$$\int_K|f|^p\mathrm{d} m=\int_{K_{\mathcal{N}_3(w)}}|f|^p\mathrm{d} m\le  \frac{C}{(8\rho)^n}\mathcal{E}^{\mathcal{N}_3(w)}(f)\le \frac{C}{(8\rho)^n}\mathcal{E}(f).$$
\end{lemma}

\begin{proof}
It suffices to prove the first inequality. Notice that $(W_n,\sim_n)$ is a connected finite graph. Since $\mathcal{N}_3(w)\ne W_n$, there exist $v\in \mathcal{N}_3(w)$ and $u\in W_n\backslash\mathcal{N}_3(w)$ such that $v\sim_nu$. Since $f\in \mathcal{F}^0_{\mathcal{N}_3(w)}$, $f=0$ on $K_u$, that is, $f\circ F_u=0$, which gives $f_{K_u}=0$ and $\mathcal{E}(f\circ F_u)=0$. By (\ref{eq_P9b}),
$$\lvert f_{K_v}\rvert^p\lesssim \mathcal{E}(f\circ F_v)\le \sum_{v\in \mathcal{N}_3(w)}\mathcal{E}(f\circ F_v)=\frac{1}{\rho^n}\mathcal{E}^{\mathcal{N}_3(w)}(f).$$
Then by (\ref{eq_Hin520a}),
$$\lvert f_{K_w}\rvert^p \le 2^{p-1}\left(\lvert f_{K_v}\rvert^p+\lvert f_{K_w}-f_{K_v}\rvert^p\right)\lesssim\frac{1}{\rho^n}\mathcal{E}^{\mathcal{N}_3(w)}(f).$$
Finally, (\ref{eq_Hin520b}) gives
$$\int_{K_{\mathcal{N}_3(w)}}|f|^p\mathrm{d} m\le 2^{p-1}\left(\int_{K_{\mathcal{N}_3(w)}}|f-f_{K_w}|^p\mathrm{d} m+\frac{\# \mathcal{N}_3(w)}{8^n}\lvert f_{K_w}\rvert^p\right)\lesssim \frac{1}{(8\rho)^n}\mathcal{E}^{\mathcal{N}_3(w)}(f).$$
\end{proof}

\begin{lemma}[{$p$-version of {\cite[Lemmas 5.13 and 5.14]{Hin13}}}]\label{lem_ext}
Let $n\ge0$ and let $w\in W_n$ satisfy $K_{\mathcal{N}_3(w)}\cap K^\partial=\emptyset$. Denote $A=\mathcal{N}_3(w)$.
\begin{enumerate}[label=(\alph*),ref=(\alph*)]
\item\label{item_exta} For any $f\in \mathcal{F}$, there exists a unique function in $\mathcal{H}(A)$, denoted by $H_{A}f$, such that $H_{A}f=f$ on $K_{W_n\backslash A}$. Moreover, $H_{A}f-f$ is the unique function in $\mathcal{F}^0_{A}$ that minimizes $\mathcal{F}^0_A\to \mathbb{R}$, $\varphi\mapsto\mathcal{E}^A(f+\varphi)$, $H_Af$ is the unique minimizer of $\mathcal{E}$ over $\{f+\varphi:\varphi\in \mathcal{F}^0_A\}$, and $f\in \mathcal{H}(A)$ if and only if $H_{A}f=f$.
\item\label{item_extb} Let $f_1,f_2\in \mathcal{F}$. If $f_1=f_2$ on $K_A$, then $H_Af_1=H_Af_2$ on $K_A$, hence $H_Af_1-H_Af_2=f_1-f_2$.
\item\label{item_extc} Let $f\in \mathcal{F}$. If $f\ge0$ on $K$, then $H_Af\ge0$ on $K$.
\item\label{item_extd} For any $f\in \mathcal{F}$ and any $a,b\in \mathbb{R}$, we have $H_A(af+b)=aH_Af+b$.
\end{enumerate}
\end{lemma}

\begin{proof}
\ref{item_exta} The proof is identical to that of \cite[Proposition 4.2]{Yan25c}, with the Poincar\'e inequality used there replaced by Lemma \ref{lem_poinc0}.

\ref{item_extb} Since $f_1=f_2$ on $K_A$, the following two functionals defined on $\mathcal{F}^0_A$ coincide: $\varphi\mapsto \mathcal{E}^{A}(f_1+\varphi)$ and $\varphi\mapsto \mathcal{E}^{A}(f_2+\varphi)$. By \ref{item_exta}, they have the same minimizer $\varphi=H_Af_1-f_1=H_Af_2-f_2\in \mathcal{F}^0_A$, hence $H_Af_1-H_Af_2=f_1-f_2$, in particular, $H_Af_1=H_Af_2$ on $K_A$.

\ref{item_extc} Let $h:=H_Af$, then $h_+\in \mathcal{F}$ and $\mathcal{E}(h_+)\le \mathcal{E}(h)$. Since $f\ge0$ on $K$, on $K_{W_n\backslash A}$, $h_+-f=f_+-f=f-f=0$, which gives $h_+\in\{f+\varphi:\varphi\in \mathcal{F}^0_A\}$. Since $h=H_Af$ is the unique minimizer of $\mathcal{E}$ over $\{f+\varphi:\varphi\in \mathcal{F}^0_A\}$ by \ref{item_exta}, $h_+=h$, which implies $H_Af\ge0$ on $K$.

\ref{item_extd} It is obvious by the uniqueness in \ref{item_exta} and the definition of $\mathcal{H}(A)$.
\end{proof}

\begin{lemma}[Comparison principle]\label{lem_comparison}
Let $A$ be as in Lemma \ref{lem_ext} and let $u,v\in\mathcal{H}(A)$ satisfy $u\le v$ on $K_{W_n\backslash A}$. Then $u\le v$ on $K$. In particular, for any $h\in\mathcal{H}(A)$,
$$-H_Ah_-\le h\le H_Ah_+,\qquad |h|\le H_Ah_++H_Ah_-.$$
\end{lemma}

\begin{proof}
Since $u\le v$ on $K_{W_n\backslash A}$, $\pm(u-v)_+\in \mathcal{F}^0_A$. Notice that $u\vee v=v+(u-v)_+$ and $u\wedge v=u-(u-v)_+$. Since $u,v\in \mathcal{H}(A)$, by Lemma \ref{lem_ext}\ref{item_exta}, $u=H_Au$, $v=H_Av$, and $\mathcal{E}(v)\le \mathcal{E}(u\vee v)$, $\mathcal{E}(u)\le \mathcal{E}(u\wedge v)$. Then by \ref{eq_SubAdd},
$$\mathcal{E}(u)+\mathcal{E}(v)\le \mathcal{E}(u\vee v)+\mathcal{E}(u\wedge v)\le \mathcal{E}(u)+\mathcal{E}(v).$$
Hence $\mathcal{E}(v)= \mathcal{E}(u\vee v)$ and $\mathcal{E}(u)= \mathcal{E}(u\wedge v)$. By Lemma \ref{lem_ext}\ref{item_exta} again, $v=u\vee v$ and $u=u\wedge v$, that is, $u\le v$ on $K$.
\end{proof}

\begin{lemma}[Elliptic Harnack inequality on cell neighbourhoods; $p$-version of {\cite[Lemma 5.21]{Hin13}}]\label{lem_ehi_cell}
There exists a constant $C_H\ge1$ such that for any $w\in W_*$ with $K_{\mathcal{N}_3(w)}\cap K^\partial=\emptyset$, and any $u\in\mathcal{H}(\mathcal{N}_3(w))$ with $u\ge0$ on $K_{\mathcal{N}_3(w)}$,
$$\esup_{K_{\mathcal{N}_2(w)}}u\le C_H\einf_{K_{\mathcal{N}_2(w)}}u.$$
\end{lemma}

\begin{proof}
It follows directly from the elliptic Harnack inequality for balls obtained in \cite[Corollary 2.8]{Yan25d}, the comparability of balls and cells, and a suitable chaining argument. Although the underlying space in \cite{Yan25d} is assumed to be unbounded, the strongly local property ensures that the result still holds for bounded spaces.
\end{proof}

\begin{proposition}[{$p$-version of {\cite[Proposition 5.22]{Hin13}}}]\label{prop_sup}
There exists a constant $C>0$ such that for any $w\in W_*$ with $K_{\mathcal{N}_3(w)}\cap K^\partial=\emptyset$ and any $h\in\mathcal{H}(\mathcal{N}_3(w))$,
\begin{equation}\label{eq_Hin522}
\esup_{K_{\mathcal{N}_2(w)}}|h-h_{K_w}|\le C\left(\frac{1}{\rho^{|w|}}\mathcal{E}^{\mathcal{N}_3(w)}(h)\right)^{\frac{1}{p}}.
\end{equation}
\end{proposition}

\begin{proof}
Replacing $h$ by $h-h_{K_w}$, we may assume that $h_{K_w}=0$. Denote $A=\mathcal{N}_3(w)$. By Lemma \ref{lem_comparison}, $\lvert h\rvert\le H_{A}h_++H_Ah_-$, where $H_Ah_\pm\ge0$ by Lemma \ref{lem_ext}\ref{item_extc}. Lemma \ref{lem_ehi_cell} gives
\begin{equation}\label{eq_Hin522a}
\esup_{K_{\mathcal{N}_2(w)}}\lvert h\rvert\le \esup_{K_{\mathcal{N}_2(w)}}H_Ah_++\esup_{K_{\mathcal{N}_2(w)}}H_Ah_-\lesssim \einf_{K_{\mathcal{N}_2(w)}}H_Ah_++\einf_{K_{\mathcal{N}_2(w)}}H_Ah_-.
\end{equation}
Since $m(K_{\mathcal{N}_2(w)})=\frac{\# \mathcal{N}_2(w)}{8^{\lvert w\rvert}}\ge \frac{1}{8^{\lvert w\rvert}}$,
\begin{align}
&\einf_{K_{\mathcal{N}_2(w)}}H_Ah_\pm\le\left( \dashint_{K_{\mathcal{N}_2(w)}}( H_Ah_\pm)^p \mathrm{d}m\right)^{\frac{1}{p}}\nonumber\\
&\le 8^{\frac{\lvert w\rvert}{p}}\left(\int_{K_A}\lvert H_Ah_\pm-h_{\pm}\rvert^p \mathrm{d}m\right)^{\frac{1}{p}}+8^{\frac{\lvert w\rvert}{p}}\left(\int_{K_A}h_{\pm}^p \mathrm{d}m\right)^{\frac{1}{p}}.\label{eq_Hin522b}
\end{align}
By Lemma \ref{lem_ext}\ref{item_exta}, $H_Ah_{\pm}-h_{\pm}\in \mathcal{F}^0_A$, then Lemma \ref{lem_poinc0} gives
\begin{align}
&\left(\int_{K_A}\lvert H_Ah_\pm-h_{\pm}\rvert^p \mathrm{d}m\right)^{\frac{1}{p}}\lesssim \frac{1}{(8\rho)^{\frac{\lvert w\rvert}{p}}}\mathcal{E}^{A}(H_Ah_{\pm}-h_{\pm})^{\frac{1}{p}}\nonumber\\
&\le \frac{1}{(8\rho)^{\frac{\lvert w\rvert}{p}}}\left(\mathcal{E}^{A}(H_Ah_{\pm})^{\frac{1}{p}}+\mathcal{E}^{A}(h_{\pm})^{\frac{1}{p}}\right)\le \frac{2}{(8\rho)^{\frac{\lvert w\rvert}{p}}}\mathcal{E}^{A}(h_{\pm})^{\frac{1}{p}}\le \frac{2}{(8\rho)^{\frac{\lvert w\rvert}{p}}}\mathcal{E}^{A}(h)^{\frac{1}{p}}.\label{eq_Hin522c}
\end{align}
Since $h_{K_w}=0$, by (\ref{eq_Hin520b}),
$$\int_{K_A} h_{\pm}^p \mathrm{d}m\le \int_{K_A}\lvert h\rvert^p \mathrm{d}m\lesssim \frac{1}{(8\rho)^{\lvert w\rvert}}\mathcal{E}^A(h).$$
Combining this with (\ref{eq_Hin522a})--(\ref{eq_Hin522c}), we obtain (\ref{eq_Hin522}).
\end{proof}

We have the following oscillation estimates.

\begin{proposition}\label{prop_holder}
Let $C_H\ge1$ be the constant in Lemma \ref{lem_ehi_cell} and let $\delta=\max\{\frac{C_H-1}{C_H+1},\frac12\}\in[\frac12,1)$. There exists a constant $C>0$ such that, for any $n\ge0$, any $w\in W_n$ with $K_{\mathcal{N}_3(w)}\cap K^\partial=\emptyset$, and any $h\in\mathcal{H}(\mathcal{N}_3(w))$, writing $\lambda:=\frac{1}{\rho^{\lvert w\rvert}}\mathcal{E}^{\mathcal{N}_3(w)}(h)$, the function $h$ has a version $\overline{h}$, which is continuous on $K_{\mathcal{N}_2(w)}$ and satisfies
\begin{enumerate}[label=(\arabic*),ref=(\arabic*)]
\item\label{item_holder1}
$$\sup_{K_{\mathcal{N}_2(w)}}|\overline{h}-h_{K_w}|\le C\lambda^{\frac{1}{p}};$$
\item\label{item_holder2} for any $k\ge1$ and any $x,y\in K_{\mathcal{N}_2(w)}$ with $|x-y|<\frac{1}{3^{n+k}}$,
$$|\overline{h}(x)-\overline{h}(y)|\le C\delta^k\lambda^{\frac{1}{p}}.$$
\end{enumerate}
\end{proposition}

\begin{proof}
It follows directly by combining a standard iteration argument based on the elliptic Harnack inequality (Lemma \ref{lem_ehi_cell}) with the estimate (\ref{eq_Hin522}) in Proposition \ref{prop_sup}.
\end{proof}

We now give the proof of Proposition \ref{prop_Hin418} as follows.

\begin{proof}[Proof of Proposition \ref{prop_Hin418}]
Let
\begin{equation}\label{eq_Lambda}
\Lambda:=\sup_{n\ge1}\frac{1}{\rho^{\lvert w^{(n)}\rvert}}\Gamma(h_n)(K_{\mathcal{N}_3(w^{(n)})})=\sup_{n\ge1}\frac{1}{\rho^{\lvert w^{(n)}\rvert}}\mathcal{E}^{\mathcal{N}_3(w^{(n)})}(h_n)<+\infty.
\end{equation}
By Proposition \ref{prop_holder}, each $h_n$ has a version, still denoted by $h_n$, which is continuous on $K_{\mathcal{N}_2(w^{(n)})}$, such that $\{h_n\}_{n\ge1}$ satisfies
\begin{enumerate}[label=(\alph*),ref=(\alph*)]
\item\label{item_unibdd}
$$\sup_{n\ge1}\sup_{K_{\mathcal{N}_2(w^{(n)})}}\lvert h_n\rvert\lesssim \Lambda^{\frac{1}{p}}.$$
\item\label{item_equicont} for any $k\ge1$ and any $x,y\in K_{\mathcal{N}_2(w^{(n)})}$ with $|x-y|<\frac{1}{3^{\lvert w^{(n)}\rvert+k}}$,
$$|{h}_n(x)-{h}_n(y)|\lesssim \delta^k\Lambda^{\frac{1}{p}},$$
where $\delta\in[\frac{1}{2},1)$ is the constant in Proposition \ref{prop_holder}.
\end{enumerate}

For any $w\in W_*$, let $c_w\in \mathbb{R}^2$ denote the center of the square $F_w(\square)$.
For any $w\in W_*$ with $K_{\mathcal{N}_3(w)}\cap K^\partial=\emptyset$, we define its ``shape" $P(w)$ as
$$P(w):=\left\{3^{\lvert w\rvert}(c_v-c_w)\in \mathbb{R}^2:v\in \mathcal{N}_3(w)\right\}\subseteq \left\{(i,j)\in \mathbb{Z}^2:-3\le i,j\le 3\right\},$$
here we use the convention that if $(x_1,y_1),(x_2,y_2)\in \mathbb{R}^2$, then $(x_1,y_1)-(x_2,y_2):=(x_1-x_2,y_1-y_2)$. Notice that, as $w$ ranges over $W_*$, the ``shapes" $P(w)$, which are subsets of $\{(i,j)\in \mathbb{Z}^2:-3\le i,j\le 3\}$, can take at most $2^{49}$ distinct values. By passing to a subsequence, we may assume that $\{w^{(n)}\}_n$ all have the same ``shape", that is, $P(w^{(n)})=P(w^{(1)})$ for all $n$. Let $u=w^{(1)}$ and for any $n$, let $\xi_n:\mathbb{R}^2\to \mathbb{R}^2$ be given by $\xi_n(x)=3^{\lvert w^{(1)}\rvert-\lvert w^{(n)}\rvert}(x-c_{w^{(1)}})+c_{w^{(n)}}$, then $\xi_n(K_{\mathcal{N}_3(u)})=K_{\mathcal{N}_3(w^{(n)})}$; see also \cite[Definition 5.17]{Hin13} for the definition of ``same shape" and the construction of such a function $\xi_n$.

By regularity, there exists $g\in \mathcal{F}^0_{\mathcal{N}_2(u)}\cap C(K)$ satisfying $0\le g\le 1$ on $K$ and $g=1$ on $K_{\mathcal{N}_1(u)}$. For any $n$, let
\begin{equation}\label{eq_fnghn}
{f}_n(x):=
\begin{cases}
g(x)h_n(\xi_n(x))&\text{if }x\in K_{\mathcal{N}_2(u)},\\
0&\text{if }x\in K\backslash K_{\mathcal{N}_2(u)},
\end{cases}
\end{equation}
then $f_n\in \mathcal{F}\cap C(K)$ is well-defined and $f_n\in \mathcal{H}(\mathcal{N}_1(u))$. By \ref{item_unibdd}, \ref{item_equicont}, and the uniform continuity of $g\in C(K)$, $\{f_n\}_{n\ge1}$ is uniformly bounded and equicontinuous on $K_{\mathcal{N}_2(u)}$. Since $K_{\mathcal{N}_2(u)}$ is compact, by the Arzel\`a--Ascoli theorem, there exist a subsequence, still denoted by $\{f_n\}_{n\ge1}$, and a function $f_\infty\in C(K_{\mathcal{N}_2(u)})$, such that $\{f_n\}_{n\ge1}$ converges uniformly to $f_\infty$ on $K_{\mathcal{N}_2(u)}$. Since $\mathrm{supp}(f_n)\subseteq K_{\mathcal{N}_2(u)}$ for any $n$, $\mathrm{supp}(f_\infty)\subseteq K_{\mathcal{N}_2(u)}$ and $\{f_n\}_{n\ge1}$ converges uniformly to $f_\infty$ on $K$.

It suffices to show that $\{f_n\circ F_u\}_{n\ge1}$ is an $\mathcal{E}$-Cauchy sequence, which would imply that $h:=f_\infty\circ F_u\in \mathcal{H}\cap C(K)$, and that $\{f_n\circ F_u\}_{n\ge1}$ is $\mathcal{E}_1$-convergent and uniformly convergent to $h=f_\infty\circ F_u$ on $K$. By \ref{item_unibdd} and (\ref{eq_fnghn}),
$$\sup_{n\ge1}\sup_K\lvert f_n\rvert\lesssim \Lambda^{\frac{1}{p}},$$
moreover, \ref{eq_Alg} and (\ref{eq_Lambda}) imply
\begin{align*}
&\mathcal{E}(f_n)^{\frac{1}{p}}=\mathcal{E}^{\mathcal{N}_2(u)}\left(g(h_n\circ \xi_n)\right)^{\frac{1}{p}}\\
&\le \left(\sup_{K_{\mathcal{N}_2(u)}}g\right) \mathcal{E}^{\mathcal{N}_2(u)}(h_n\circ \xi_n)^{\frac{1}{p}}+\left(\sup_{K_{\mathcal{N}_2(u)}}\lvert h_n\circ \xi_n\rvert\right)\mathcal{E}^{\mathcal{N}_2(u)}(g)^{\frac{1}{p}}\lesssim \Lambda^{\frac{1}{p}},
\end{align*}
hence $\sup_{n\ge1}\mathcal{E}(f_n)^{\frac{1}{p}}\lesssim \Lambda^{\frac{1}{p}}$.

For any $k,l\ge1$, let $\varepsilon_{k,l}=\sup_{K_{\mathcal{N}_1(u)}}\lvert f_k-f_l\rvert\lesssim \Lambda^{\frac{1}{p}}$, then $\lim_{k,l\to+\infty}\varepsilon_{k,l}=0$. By regularity again, there exists $\widehat{g}\in \mathcal{F}^0_{\mathcal{N}_1(u)}\cap C(K)$ satisfying $0\le \widehat{g}\le 1$ on $K$ and $\widehat{g}=1$ on $K_{u}$. Let
\begin{align*}
A_0&:=\left\{x\in K_{\mathcal{N}_1(u)}:\lvert f_k(x)-f_l(x)\rvert\le \varepsilon_{k,l}\widehat{g}\right\}\supseteq K_u,\\
A_+&:=\left\{x\in K_{\mathcal{N}_1(u)}:f_k(x)-f_l(x)> \varepsilon_{k,l}\widehat{g}\right\},\\
A_-&:=\left\{x\in K_{\mathcal{N}_1(u)}:f_k(x)-f_l(x)<- \varepsilon_{k,l}\widehat{g}\right\},
\end{align*}
then $K_{\mathcal{N}_1(u)}=A_0\sqcup A_+\sqcup A_-$. Since
$$\rho^{\lvert u\rvert}\mathcal{E}((f_k-f_l)\circ F_u)=\Gamma(f_k-f_l)(K_u)\le \Gamma(f_k-f_l)(A_0),$$
we only need to show that
\begin{equation}\label{eq_Hin418goal}
\lim_{k,l\to+\infty}\Gamma(f_k-f_l)(A_0)=0.
\end{equation}

Let
$$\phi:=\frac{1}{2}\left(((f_k-f_l)\vee (-\varepsilon_{k,l}\widehat{g}))\wedge (\varepsilon_{k,l}\widehat{g})\right),$$
then $\phi\in \mathcal{F}^0_{\mathcal{N}_1(u)}\cap C(K)$. Since $f_l$ is harmonic on $\mathcal{N}_1(u)$, that is, $f_l\in \mathcal{H}(\mathcal{N}_1(u))$,
\begin{equation}\label{eq_Hin418a}
\Gamma(f_l)(K_{\mathcal{N}_1(u)})\le \Gamma(f_l+\phi)(K_{\mathcal{N}_1(u)}).
\end{equation}
Notice that $f_l+\phi=f_l\pm\frac{1}{2}\varepsilon_{k,l}\widehat{g}\text{ on }A_{\pm}$. By the strongly local property,
$$\Gamma(f_l+\phi)(A_{\pm})^{\frac{1}{p}}=\Gamma \left(f_l\pm\frac{1}{2}\varepsilon_{k,l}\widehat{g}\right)(A_{\pm})^{\frac{1}{p}}\le \Gamma(f_l)(A_{\pm})^{\frac{1}{p}}+\frac{1}{2}\varepsilon_{k,l}\mathcal{E}(\widehat{g})^{\frac{1}{p}}.$$
By the elementary result that for any $a,b\ge0$, $(a+b)^p\le a^p+pb(a+b)^{p-1}$,
$$\Gamma(f_l+\phi)(A_{\pm})\le \Gamma(f_l)(A_{\pm})+\varepsilon_{k,l}D_1,$$
where
$$D_1:=\frac{1}{2}p\mathcal{E}(\widehat{g})^{\frac{1}{p}}\left(\mathcal{E}(f_l)^{\frac{1}{p}}+\frac{1}{2}\varepsilon_{k,l}\mathcal{E}(\widehat{g})^{\frac{1}{p}}\right)^{p-1}\lesssim \Lambda^{\frac{p-1}{p}}.$$
Since $f_l+\phi=\frac{f_k+f_l}{2}$ on $A_0$, by (\ref{eq_Hin418a}) and the strongly local property,
\begin{align*}
&\Gamma(f_l)(K_{\mathcal{N}_1(u)})=\Gamma(f_l)(A_0)+\Gamma(f_l)(A_+)+\Gamma(f_l)(A_-)\\
&\le\Gamma(f_l+\phi)(K_{\mathcal{N}_1(u)})=\Gamma(f_l+\phi)(A_0)+\Gamma(f_l+\phi)(A_+)+\Gamma(f_l+\phi)(A_-)\\
&\le \Gamma \left(\frac{f_k+f_l}{2}\right)(A_0)+\Gamma(f_l)(A_+)+\Gamma(f_l)(A_-)+2\varepsilon_{k,l}D_1,
\end{align*}
which implies
\begin{equation}\label{eq_Hin418b}
\Gamma(f_l)(A_0)\le \Gamma \left(\frac{f_k+f_l}{2}\right)(A_0)+2\varepsilon_{k,l}D_1.
\end{equation}
Similarly, by changing the roles of $k$ and $l$,
\begin{equation}\label{eq_Hin418c}
\Gamma(f_k)(A_0)\le \Gamma \left(\frac{f_k+f_l}{2}\right)(A_0)+2\varepsilon_{k,l}D_2,
\end{equation}
where
$$D_2:=\frac{1}{2}p\mathcal{E}(\widehat{g})^{\frac{1}{p}}\left(\mathcal{E}(f_k)^{\frac{1}{p}}+\frac{1}{2}\varepsilon_{k,l}\mathcal{E}(\widehat{g})^{\frac{1}{p}}\right)^{p-1}\lesssim \Lambda^{\frac{p-1}{p}}.$$

If $p\in[2,+\infty)$, then by \hyperref[eq_ClaGamma]{$\text{Cla}^{\Gamma}$},
\begin{align*}
&\Gamma \left(\frac{f_k-f_l}{2}\right)(A_0)\le 2 \left(\Gamma\left(\frac{f_k}{2}\right)(A_0)^{\frac{1}{p-1}}+\Gamma\left(\frac{f_l}{2}\right)(A_0)^{\frac{1}{p-1}}\right)^{p-1}-\Gamma \left(\frac{f_k+f_l}{2}\right)(A_0)\\
&=\frac{1}{2^{p-1}}\left(\Gamma\left({f_k}\right)(A_0)^{\frac{1}{p-1}}+\Gamma\left({f_l}\right)(A_0)^{\frac{1}{p-1}}\right)^{p-1}-\Gamma \left(\frac{f_k+f_l}{2}\right)(A_0).
\end{align*}
Combining this with (\ref{eq_Hin418b}), (\ref{eq_Hin418c}), and the elementary result that for any $p>1$ and any $a,b\ge0$, $(a^{\frac{1}{p-1}}+b^{\frac{1}{p-1}})^{p-1}\le \max\{1,2^{p-2}\}(a+b)$,
$$\Gamma \left(\frac{f_k-f_l}{2}\right)(A_0)\le \varepsilon_{k,l}(D_1+D_2)\to0\text{ as }k,l\to+\infty.$$
If $p\in(1,2]$, then by \hyperref[eq_ClaGamma]{$\text{Cla}^{\Gamma}$}, (\ref{eq_Hin418b}), and (\ref{eq_Hin418c}),
\begin{align*}
&2 \left(\Gamma \left(\frac{f_k+f_l}{2}\right)(A_0)^{\frac{1}{p-1}}+\Gamma \left(\frac{f_k-f_l}{2}\right)(A_0)^{\frac{1}{p-1}}\right)^{p-1}\\
&\le \Gamma(f_k)(A_0)+\Gamma(f_l)(A_0)\le 2\Gamma \left(\frac{f_k+f_l}{2}\right)(A_0)+2\varepsilon_{k,l}(D_1+D_2).
\end{align*}
By the elementary result that for any $p\in(1,2]$ and any $a,b\ge0$, $(a+b)^{\frac{1}{p-1}}\le a^{\frac{1}{p-1}}+ C_p \left(a^{\frac{2-p}{p-1}} b+ b^{\frac{1}{p-1}}\right)$, where $C_p = \frac{1}{p-1}2^{\max\left\{\frac{1}{p-1}-2,\,0\right\}}$,
\begin{align*}
&\Gamma \left(\frac{f_k+f_l}{2}\right)(A_0)^{\frac{1}{p-1}}+\Gamma \left(\frac{f_k-f_l}{2}\right)(A_0)^{\frac{1}{p-1}}\le \left(\Gamma \left(\frac{f_k+f_l}{2}\right)(A_0)+\varepsilon_{k,l}(D_1+D_2)\right)^{\frac{1}{p-1}}\\
&\le \Gamma \left(\frac{f_k+f_l}{2}\right)(A_0)^{\frac{1}{p-1}}+C_p \left(\Gamma \left(\frac{f_k+f_l}{2}\right)(A_0)^{\frac{2-p}{p-1}}\varepsilon_{k,l}(D_1+D_2)+\left(\varepsilon_{k,l}(D_1+D_2)\right)^{\frac{1}{p-1}}\right),
\end{align*}
hence
\begin{align*}
&\Gamma \left(\frac{f_k-f_l}{2}\right)(A_0)^{\frac{1}{p-1}}\lesssim \left(\sup_{n\ge1}\mathcal{E}(f_n)\right)^{\frac{2-p}{p-1}}\varepsilon_{k,l}(D_1+D_2)+\left(\varepsilon_{k,l}(D_1+D_2)\right)^{\frac{1}{p-1}}\\
&\lesssim \Lambda^{\frac{1}{p(p-1)}}\varepsilon_{k,l}+\Lambda^{\frac{1}{p}}\varepsilon_{k,l}^{\frac{1}{p-1}}\to0\text{ as }k,l\to+\infty.
\end{align*}
In summary, we obtain (\ref{eq_Hin418goal}).
\end{proof}

\section{Index one for the \texorpdfstring{$p$}{p}-energy}\label{sec_R1}

Let us recall the definitions of a $p$-Dirichlet space and its regularity from \cite[Definitions 2.2 and 2.12]{EM25} as follows.

We say that $(K,d,m,\mathcal{E},\mathcal{F},\Gamma)$ is a $p$-Dirichlet space if the following hold:
\begin{enumerate}[label=($p$-DS\arabic*),ref=($p$-DS\arabic*)]
\item\label{item_pDS1} The triple $(K,d,m)$ is a locally compact metric space equipped with a positive Radon measure $m$.
\item\label{item_pDS2} The space $\mathcal{F}$ is a subspace of $L^p(K;m)$ and $\mathcal{E}:\mathcal{F}\to [0,+\infty)$ is a non-negative function such that $(\mathcal{F},(\mathcal{E}(\cdot)+\lVert \cdot\rVert_{L^p(K;m)}^p)^{\frac{1}{p}})$ is a Banach space.
\item\label{item_pDS3} For any $f\in \mathcal{F}$, there exists a finite non-negative Borel measure $\Gamma(f)$ on $K$ such that $\Gamma(f)(K)=\mathcal{E}(f)$ and for any $\lambda\in \mathbb{R}$, $\Gamma(\lambda f)=\lvert \lambda\rvert^p \Gamma(f)$.
\item\label{item_pDS4} For any $f,g\in \mathcal{F}$ and any $A\in \mathcal{B}(K)$,
$$\Gamma(f+g)(A)^{\frac{1}{p}}\le \Gamma(f)(A)^{\frac{1}{p}}+\Gamma(g)(A)^{\frac{1}{p}}.$$
\item\label{item_pDS5} For any $f\in \mathcal{F}$ and any function $\phi:\mathbb{R}\to\mathbb{R}$ with $\phi(0)=0$ and $\lvert \phi(t)-\phi(s)\rvert\le L \lvert t-s\rvert$ for any $t,s\in \mathbb{R}$, where $L\ge0$ is a constant, we have $\phi\circ f\in \mathcal{F}$ and
$$\mathrm{d}\Gamma(\phi\circ f)\le L^p \mathrm{d}\Gamma(f).$$
\item\label{item_pDS6} For any $f\in \mathcal{F}$ and any open subset $A\subseteq K$ satisfying that $f=c$ $m$-a.e. on $A$, where $c\in \mathbb{R}$ is a constant, then $\Gamma(f)(A)=0$.
\item\label{item_pDS7} Let $\{f_n\}\subseteq \mathcal{F}$ and $f\in L^p(K;m)$. If $\sup_{n}\mathcal{E}(f_n)<+\infty$ and $\{f_n\}_n$ is $L^p(K;m)$-convergent to $f$, then $f\in \mathcal{F}$ and for any $A\in \mathcal{B}(K)$,
$$\Gamma(f)(A)\le\varliminf_{n\to+\infty}\Gamma(f_n)(A).$$
\end{enumerate}

A $p$-Dirichlet space $(K,d,m,\mathcal{E},\mathcal{F},\Gamma)$ is called regular if $\mathcal{F}\cap C_c(K)$ is uniformly dense in $C_c(K)$ and $(\mathcal{E}(\cdot)+\lVert \cdot\rVert_{L^p(K;m)}^p)^{\frac{1}{p}}$-dense in $\mathcal{F}$.

As in the theory of Dirichlet forms, one can develop a theory of capacity and quasi-continuity for regular $p$-Dirichlet spaces. In particular, for a regular $p$-Dirichlet space $(K,d,m,\mathcal{E},\mathcal{F},\Gamma)$, any function $f\in \mathcal{F}$ admits a quasi-continuous modification $\widetilde{f}\in \mathcal{F}$, unique up to a capacity zero set, such that $f=\widetilde{f}$ $m$-a.e. on $K$; see \cite[Proposition 2.34]{EM25}.

Let $n\ge1$ be an integer. Let $\mathbb{S}^{n-1}:=\{\boldsymbol{\lambda}\in \mathbb{R}^n:\lvert \boldsymbol{\lambda}\rvert=1\}$ denote the unit sphere in $\mathbb{R}^n$. Let $\mathcal{F}^n$ denote the $n$-fold Cartesian product of $\mathcal{F}$.

Let $\Lambda$ be a $\sigma$-finite Borel measure on $(K,d)$, and let $\{f_i\}_{i\in I}$ be a family of measurable functions, where $I$ is the index set. Then the lattice infimum
$$\bigwedge_{i\in I}^\Lambda f_i$$
is the unique measurable function $f$, up to $\Lambda$-a.e. equality, such that the following hold.
\begin{enumerate}[label=(\arabic*),ref=(\arabic*)]
\item\label{item_lattice1} For any $i\in I$, we have $f\le f_i$ $\Lambda$-a.e. on $K$.
\item\label{item_lattice2} If $g$ also satisfies \ref{item_lattice1}, then $g\le f$ $\Lambda$-a.e. on $K$.
\end{enumerate}

We say that $\boldsymbol{\phi}=(\phi_1,\ldots,\phi_n)\in \mathcal{F}^n$ is $p$-independent on $A\in \mathcal{B}(K)$ if
$$\bigwedge^{\Lambda_{\boldsymbol{\phi}}}_{\boldsymbol{\lambda}=(\lambda_1,\ldots,\lambda_n)\in \mathbb{S}^{n-1}}\frac{\mathrm{d}\Gamma \left(\sum_{i=1}^n\lambda_i\phi_i\right)}{\mathrm{d}\Lambda_{\boldsymbol{\phi}}}(x)>0\text{ for $\Lambda_{\boldsymbol{\phi}}$-a.e. }x\in A,$$
where
$$\Lambda_{\boldsymbol{\phi}}:=\sum_{i=1}^n\Gamma(\phi_i).$$

By sub-linearity \ref{item_pDS4}, for any $\boldsymbol{\lambda}=(\lambda_1,\ldots,\lambda_n)$, $\boldsymbol{\lambda}'=(\lambda_1',\ldots,\lambda_n')\in \mathbb{R}^{n}$,
\begin{align}
&\lvert \left(\frac{\mathrm{d}\Gamma \left(\sum_{i=1}^n\lambda_i\phi_i\right)}{\mathrm{d}\Lambda_{\boldsymbol{\phi}}}(x)\right)^{\frac{1}{p}}-\left(\frac{\mathrm{d}\Gamma \left(\sum_{i=1}^n\lambda_i'\phi_i\right)}{\mathrm{d}\Lambda_{\boldsymbol{\phi}}}(x)\right)^{\frac{1}{p}}\rvert\nonumber\\
&\le \sum_{i=1}^n \lvert \lambda_i-\lambda_i'\rvert \left(\frac{\mathrm{d}\Gamma(\phi_i)}{\mathrm{d}\Lambda_{\boldsymbol{\phi}}}(x)\right)^{\frac{1}{p}}\text{ for $\Lambda_{\boldsymbol{\phi}}$-a.e. }x\in K.\label{eq_Gamma_holder}
\end{align}

From now on, let $S$ be any fixed countable dense subset of $\mathbb{S}^{n-1}$ containing $\boldsymbol{e}_1=(1,0,\ldots,0)$, \ldots, $\boldsymbol{e}_n=(0,0,\ldots,1)$.

For any $\boldsymbol{\phi}=(\phi_1,\ldots,\phi_n)\in \mathcal{F}^n$, there exists $N_{\boldsymbol{\phi}}\subseteq K$ with $\Lambda_{\boldsymbol{\phi}}(N_{\boldsymbol{\phi}})=0$ such that, for any $x\in K\backslash N_{\boldsymbol{\phi}}$, $(\frac{\mathrm{d}\Gamma \left(\sum_{i=1}^n\lambda_i\phi_i\right)}{\mathrm{d}\Lambda_{\boldsymbol{\phi}}}(x))^{\frac{1}{p}}$ is well-defined for any $\boldsymbol{\lambda}=(\lambda_1,\ldots,\lambda_n)\in S$ and is Lipschitz continuous as a function of $\boldsymbol{\lambda}$ on $S$. It therefore has a unique Lipschitz continuous extension to $\mathbb{S}^{n-1}$, which we denote by the same expression. For any $x\in K\backslash N_{\boldsymbol{\phi}}$, let
\begin{equation}\label{eq_Psiphi}
\Psi_{\boldsymbol{\phi}}(\boldsymbol{\lambda})(x)=\Psi_{\boldsymbol{\phi}}(\lambda_1,\ldots,\lambda_n)(x):=\frac{\mathrm{d}\Gamma \left(\sum_{i=1}^n\lambda_i\phi_i\right)}{\mathrm{d}\Lambda_{\boldsymbol{\phi}}}(x),
\end{equation}
then $\Psi_{\boldsymbol{\phi}}(\cdot)(x)\in C(\mathbb{S}^{n-1})$. By sub-linearity \ref{item_pDS4}, after enlarging $N_{\boldsymbol{\phi}}$ by a $\Lambda_{\boldsymbol{\phi}}$-null set, for any $x\in K\backslash N_{\boldsymbol{\phi}}$ and any $\boldsymbol{\lambda}\in S$,
\begin{equation}\label{eq_Psi_upper}
\Psi_{\boldsymbol{\phi}}(\boldsymbol{\lambda})(x)^{\frac{1}{p}}\le\sum_{i=1}^n\lvert\lambda_i\rvert \left(\frac{\mathrm{d}\Gamma(\phi_i)}{\mathrm{d}\Lambda_{\boldsymbol{\phi}}}(x)\right)^{\frac{1}{p}}\le\sum_{i=1}^n \lvert\lambda_i\rvert\le \sqrt{n}\left(\sum_{i=1}^n \lambda_i^2\right)^{\frac{1}{2}}=\sqrt{n}.
\end{equation}
By continuity in $\boldsymbol{\lambda}$, this estimate also holds for any $\boldsymbol{\lambda}\in \mathbb{S}^{n-1}$. Moreover, by definition, $\boldsymbol{\phi}\in \mathcal{F}^n$ is $p$-independent on $A\in \mathcal{B}(K)$ if and only if
\begin{equation}\label{eq_pind}
\Lambda_{\boldsymbol{\phi}}\left(\left\{x\in A\backslash N_{\boldsymbol{\phi}}:\inf_{\boldsymbol{\lambda}\in S}\Psi_{\boldsymbol{\phi}}(\boldsymbol{\lambda})(x)=\inf_{\boldsymbol{\lambda}\in {\mathbb{S}}^{n-1}}\Psi_{\boldsymbol{\phi}}(\boldsymbol{\lambda})(x)=0\right\}\right)=0.
\end{equation}

We say that a $\sigma$-finite Borel measure $\mu$ on $K$ is a minimal energy-dominant measure of $(\mathcal{E},\mathcal{F})$ if the following two conditions hold.
\begin{enumerate}[label=(\roman*)]
\item (Domination) For any $f\in \mathcal{F}$, we have $\Gamma(f)\ll \mu$.
\item (Minimality) If another $\sigma$-finite Borel measure $\nu$ on $K$ also satisfies the above domination condition, then $\mu\ll\nu$.
\end{enumerate}
If $(K,d)$ is separable, then \cite[Proposition 3.3.55]{HKST15} implies that $L^p(K;m)$ is separable, and \cite[Proposition 4.1]{AHM23} and \ref{item_SSEb} imply that $(\mathcal{F},(\mathcal{E}(\cdot)+\lVert \cdot\rVert_{L^p(K;m)}^p)^{\frac{1}{p}})$ is also separable. Then the existence of a minimal energy-dominant measure $\mu$ follows from \cite[Lemma 9.20]{MS25}; see also \cite[Lemma 2.2]{Nak85}, \cite[LEMMAS 2.2, 2.3 and 2.4]{Hin10} for the existence in the Dirichlet form setting. Moreover, for $A\in \mathcal{B}(K)$, $\mu(A)=0$ if and only if $\Gamma(f)(A)=0$ for any $f\in \mathcal{F}$.

Let $\mu$ be a minimal energy-dominant measure. Then in (\ref{eq_Psiphi}),
\begin{equation}\label{eq_Psiphimu}
\Psi_{\boldsymbol{\phi}}(\boldsymbol{\lambda})(x)=
\begin{cases}
\frac{\frac{\mathrm{d}\Gamma \left(\sum_{i=1}^n\lambda_i\phi_i\right)}{\mathrm{d}\mu}(x)}{\frac{\mathrm{d}\Lambda_{\boldsymbol{\phi}}}{\mathrm{d}\mu}(x)}&\text{if }\frac{\mathrm{d}\Lambda_{\boldsymbol{\phi}}}{\mathrm{d}\mu}(x)>0,\\
0&\text{if }\frac{\mathrm{d}\Lambda_{\boldsymbol{\phi}}}{\mathrm{d}\mu}(x)=0.
\end{cases}
\end{equation}

\begin{lemma}[{$p$-version of \cite[LEMMA 2.5]{Hin10}}]\label{lem_Hin25}
Let $\mu$ be a minimal energy-dominant measure and let $f,g\in \mathcal{F}$.
\begin{enumerate}[label=(\roman*),ref=(\roman*)]
\item\label{item_enedom1} It holds that
\begin{equation}\label{eq_enedom1}
\lvert \left(\frac{\mathrm{d}\Gamma(f)}{\mathrm{d}\mu}\right)^{\frac{1}{p}}-\left(\frac{\mathrm{d}\Gamma(g)}{\mathrm{d}\mu}\right)^{\frac{1}{p}}\rvert\le \left(\frac{\mathrm{d}\Gamma(f-g)}{\mathrm{d}\mu}\right)^{\frac{1}{p}}\mu\text{-a.e. on }K.
\end{equation}
\item\label{item_enedom2} Let $\{f_n\}\subseteq \mathcal{F}$ and $f\in \mathcal{F}$ be such that $\{f_n\}$ is $\mathcal{E}$-convergent to $f$. Then $\{\frac{\mathrm{d}\Gamma(f_n)}{\mathrm{d}\mu}\}$ is $L^1(K;\mu)$-convergent to $\frac{\mathrm{d}\Gamma(f)}{\mathrm{d}\mu}$.
\end{enumerate}
\end{lemma}

\begin{proof}
\ref{item_enedom1} It follows directly from sub-linearity \ref{item_pDS4}. \ref{item_enedom2} Recall the following elementary result: There exists a constant $C_p>0$, depending only on $p$, such that for any $a,b\in[0,+\infty)$,
$$\lvert a^p-b^p\rvert\le C_p \lvert a-b\rvert \left(a^{p-1}+b^{p-1}\right).$$
Then by (\ref{eq_enedom1}) and H\"older's inequality,
\begin{align*}
&\int_K \lvert \frac{\mathrm{d}\Gamma(f_n)}{\mathrm{d}\mu}-\frac{\mathrm{d}\Gamma(f)}{\mathrm{d}\mu}\rvert \mathrm{d}\mu \\
&\le C_p\int_K \lvert \left(\frac{\mathrm{d}\Gamma(f_n)}{\mathrm{d}\mu}\right)^{\frac{1}{p}}-\left(\frac{\mathrm{d}\Gamma(f)}{\mathrm{d}\mu}\right)^{\frac{1}{p}}\rvert \left(\left(\frac{\mathrm{d}\Gamma(f_n)}{\mathrm{d}\mu}\right)^{\frac{p-1}{p}}+\left(\frac{\mathrm{d}\Gamma(f)}{\mathrm{d}\mu}\right)^{\frac{p-1}{p}}  \right)\mathrm{d}\mu\\
&\le C_p\int_K  \left(\frac{\mathrm{d}\Gamma(f_n-f)}{\mathrm{d}\mu}\right)^{\frac{1}{p}} \left(\left(\frac{\mathrm{d}\Gamma(f_n)}{\mathrm{d}\mu}\right)^{\frac{p-1}{p}}+\left(\frac{\mathrm{d}\Gamma(f)}{\mathrm{d}\mu}\right)^{\frac{p-1}{p}}  \right)\mathrm{d}\mu\\
&\le C_p \left(\mathcal{E}(f_n)^{\frac{p-1}{p}}+\mathcal{E}(f)^{\frac{p-1}{p}}\right)\mathcal{E}(f_n-f)^{\frac{1}{p}}\to0.
\end{align*}
\end{proof}

In our setting, the Sierpi\'nski carpet $(K,d,m)$ equipped with the canonical $p$-energy $(\mathcal{E}^{(c)},\mathcal{F}^{(c)})$ and its $p$-energy measure $\Gamma^{(c)}$ is a regular $p$-Dirichlet space; see \cite[Remark 2.6(2)]{EM25}. For notational convenience, we omit the superscript $(c)$ when no ambiguity arises.

The main result of this section is the following index one result.

\begin{theorem}\label{thm_R1}
For any $n\ge2$ and any $\boldsymbol{f}=(f_1,\ldots,f_n)\in \mathcal{F}^n$,
$$\Lambda_{\boldsymbol{f}}\left(\left\{\text{$\boldsymbol{f}$ is $p$-independent}\right\}\right)=0.$$
Equivalently, if $N_{\boldsymbol{f}}\subseteq K$ is a $\Lambda_{\boldsymbol{f}}$-null set such that $\Psi_{\boldsymbol{f}}(\cdot)(x)\in C(\mathbb{S}^{n-1})$ for any $x\in K\backslash N_{\boldsymbol{f}}$, then
$$\inf_{\boldsymbol{\lambda}\in \mathbb{S}^{n-1}}\Psi_{\boldsymbol{f}}(\boldsymbol{\lambda})(x)=0\text{ for $\Lambda_{\boldsymbol{f}}$-a.e. }x\in K\backslash N_{\boldsymbol{f}}.$$
\end{theorem}

\begin{remark}\label{rmk_R1}
Indeed, it suffices to prove the result for $n=2$. Suppose that there exist $\boldsymbol{f}=(f_1,f_2)\in \mathcal{F}^2$ and $A\in \mathcal{B}(K)$ satisfying $\Lambda_{\boldsymbol{f}}(A)>0$, $\Psi_{{\boldsymbol{f}}}(\cdot)(x)\in C(\mathbb{S}^1)$ for any $x\in A$, and
\begin{equation}\label{eq_R1_f}
\inf_{\boldsymbol{\lambda}\in \mathbb{S}^1}\Psi_{\boldsymbol{f}}(\boldsymbol{\lambda})(x)>0\text{ for any }x\in A.
\end{equation}
We derive a contradiction.
\end{remark}

Our strategy follows the essential ideas of \cite[Strategy 3.6]{Hin13}, with substantial modifications for the non-linear setting.

\begin{enumerate}[label=Step \arabic*:, start=0,ref=Step \arabic*]
\item\label{item_step0} Use the density result (Lemma \ref{lem_Hstar}) to reduce $\boldsymbol{f}\in \mathcal{F}^2$ in Remark \ref{rmk_R1} to $\boldsymbol{g}\in \mathcal{H}^2$; see Proposition \ref{prop_Hin49}.
\item\label{item_step1} Reduce the function $\Psi_{\boldsymbol{g}}(\cdot)(x)$ from \ref{item_step0} to a fixed function $\Theta\in C(\mathbb{S}^1)$. More precisely, find $\boldsymbol{h}\in \mathcal{H}^2$ with $\Lambda_{\boldsymbol{h}}(K)>0$, and $\Theta\in C(\mathbb{S}^1)$, bounded above and below by positive constants, such that $\Psi_{\boldsymbol{h}}(\cdot)(x)=\Theta$ for $\Lambda_{\boldsymbol{h}}$-a.e. $x\in K$; see Theorem \ref{thm_step1}.
\item\label{item_step2} We show that, after multiplying the measure $\Lambda_{\boldsymbol{h}}$ from \ref{item_step1} by a suitable cutoff function $f^p$, the pushforward measure $\widetilde{\boldsymbol{h}}_*(\widetilde{f}^p\Lambda_{\boldsymbol{h}})$ on $\mathbb{R}^2$ is absolutely continuous with respect to Lebesgue measure $\mathcal{L}^2$, with a controlled Radon--Nikodym derivative; see Proposition \ref{prop_Hin34}. This gives estimates for $\widetilde{\boldsymbol{h}}_*(\widetilde{f}^p \Lambda_{\boldsymbol{h}})$ in the results from Lemma \ref{lem_Hin39} through Proposition \ref{prop_Hin313}.
\item\label{item_step3} Finally, we use the estimates from \ref{item_step2} to obtain a contradiction in either case: to the non-triviality of $\Lambda_{\boldsymbol{h}}$ when $\beta_p>d_h$, and to a Sobolev inequality when $\beta_p\le d_h$. In both cases, the condition $\frac{d_h}{2}<\frac{\beta_p}{p}$ is intrinsically needed.
\end{enumerate}

We begin with \ref{item_step0}.

\begin{proposition}[{$p$-version of \cite[Proposition 4.9]{Hin13}}]\label{prop_Hin49}
Suppose that (\ref{eq_R1_f}) holds for some $\boldsymbol{f}=(f_1,f_2)\in \mathcal{F}^2$ and some Borel subset $A\in \mathcal{B}(K)$ with $\Lambda_{\boldsymbol{f}}(A)>0$. Then there exist $\boldsymbol{g}=(g_1,g_2)\in \mathcal{H}^2$, a Borel subset, still denoted by $A\in \mathcal{B}(K)$, and a constant $a>0$, such that $\Lambda_{\boldsymbol{g}}(A)>0$, $\Psi_{\boldsymbol{g}}(\cdot)(x)\in C(\mathbb{S}^1)$ for any $x\in A$, and
\begin{equation}\label{eq_R1_h}
\inf_{\boldsymbol{\lambda}\in \mathbb{S}^1}\Psi_{\boldsymbol{g}}(\boldsymbol{\lambda})(x)\ge a>0\text{ for any }x\in A.
\end{equation}
\end{proposition}

\begin{proof}
First, we prove that there exist $\boldsymbol{h}=(h_1,h_2)\in \mathcal{H}_*^2$, $B\in \mathcal{B}(K)$, and a constant $a>0$, such that $\Lambda_{\boldsymbol{h}}(B)>0$, $\Psi_{\boldsymbol{h}}(\cdot)(x)\in C(\mathbb{S}^1)$ for any $x\in B$, and
\begin{equation}\label{eq_R1_h1}
\inf_{\boldsymbol{\lambda}\in \mathbb{S}^1}\Psi_{\boldsymbol{h}}(\boldsymbol{\lambda})(x)\ge a>0\text{ for any }x\in B.
\end{equation}
By (\ref{eq_R1_f}), there exist a Borel subset of $A$, still denoted by $A$, and a constant $a\in(0,\frac{1}{2^p}]$ sufficiently small, such that $\Lambda_{\boldsymbol{f}}(A)>0$, and
\begin{equation}\label{eq_f2h00}
\inf_{\boldsymbol{\lambda}\in \mathbb{S}^1}\Psi_{\boldsymbol{f}}(\boldsymbol{\lambda})(x)\ge 2^pa>0\text{ for any }x\in A.
\end{equation}
By Lemma \ref{lem_Hstar}, there exists a sequence $\{\boldsymbol{h}^{(n)}=(h^{(n)}_1,h^{(n)}_2)\}_n\subseteq \mathcal{H}_*^2$ such that $\{h^{(n)}_i\}_n$ is $\mathcal{E}_1$-convergent to $f_i$ for $i=1,2$. By sub-linearity \ref{item_pDS4}, we may choose a common $\mu$-null set $N\in \mathcal{B}(K)$ and $\mu$-versions of the following Radon--Nikodym derivatives such that, for any $x\in K\backslash N$, any $\boldsymbol{\lambda}=(\lambda_1,\lambda_2),\boldsymbol{\delta}=(\delta_1,\delta_2)\in S$, and any $\boldsymbol{u}=(u_1,u_2),\boldsymbol{v}=(v_1,v_2)\in \{\boldsymbol{f}\}\cup\{\boldsymbol{h}^{(n)},\boldsymbol{h}^{(n)}-\boldsymbol{f}\}_n$,
$$\lvert \left(\frac{\mathrm{d}\Gamma(\lambda_1u_1+\lambda_2 u_2)}{\mathrm{d}\mu}(x)\right)^{\frac{1}{p}}-\left(\frac{\mathrm{d}\Gamma(\lambda_1v_1+\lambda_2v_2)}{\mathrm{d}\mu}(x)\right)^{\frac{1}{p}}\rvert\le\sum_{i=1}^2 \lvert \lambda_i\rvert\left(\frac{\mathrm{d}\Gamma(u_i-v_i)}{\mathrm{d}\mu}(x)\right)^{\frac{1}{p}},$$
and
$$\lvert \left(\frac{\mathrm{d}\Gamma(\lambda_1u_1+\lambda_2 u_2)}{\mathrm{d}\mu}(x)\right)^{\frac{1}{p}}-\left(\frac{\mathrm{d}\Gamma(\delta_1u_1+\delta_2u_2)}{\mathrm{d}\mu}(x)\right)^{\frac{1}{p}}\rvert\le\sum_{i=1}^2 \lvert \lambda_i-\delta_i\rvert\left(\frac{\mathrm{d}\Gamma(u_i)}{\mathrm{d}\mu}(x)\right)^{\frac{1}{p}}.$$
Then for any $\boldsymbol{u}=(u_1,u_2)\in \{\boldsymbol{f}\}\cup\{\boldsymbol{h}^{(n)},\boldsymbol{h}^{(n)}-\boldsymbol{f}\}_n$ and any $x\in K\backslash N$, the function
$$\boldsymbol{\lambda}=(\lambda_1,\lambda_2)\mapsto\left(\frac{\mathrm{d}\Gamma(\lambda_1u_1+\lambda_2 u_2)}{\mathrm{d}\mu}(x)\right)^{\frac{1}{p}}$$
is well-defined and Lipschitz continuous on $S$. It therefore has a unique Lipschitz continuous extension to $\mathbb{S}^1$, which we denote by $R_{\boldsymbol{u}}(\cdot)(x)\in C(\mathbb{S}^1)$, satisfying that for any $\boldsymbol{u}\in \{\boldsymbol{f}\}\cup\{\boldsymbol{h}^{(n)}\}_n$, any $\boldsymbol{\lambda},\boldsymbol{\delta}\in \mathbb{S}^1$, and any $x\in K\backslash N$,
\begin{equation}\label{eq_Rld}
\lvert R_{\boldsymbol{u}}(\boldsymbol{\lambda})(x)-R_{\boldsymbol{u}}(\boldsymbol{\delta})(x)\rvert \le\sum_{i=1}^2 \lvert \lambda_i-\delta_i\rvert R_{\boldsymbol{u}}(\boldsymbol{e}_i)(x),
\end{equation}
and for any $n$,
\begin{equation}\label{eq_Ruv}
\lvert R_{\boldsymbol{h}^{(n)}}(\boldsymbol{\lambda})(x)-R_{\boldsymbol{f}}(\boldsymbol{\lambda})(x)\rvert\le \sum_{i=1}^2 \lvert \lambda_i\rvert R_{\boldsymbol{h}^{(n)}-\boldsymbol{f}}(\boldsymbol{e}_i)(x).
\end{equation}
Moreover, Lemma \ref{lem_Hin25}\ref{item_enedom2} yields that, for any $\boldsymbol{\lambda}=(\lambda_1,\lambda_2)\in \mathbb{S}^1$,
\begin{equation}\label{eq_GammaR}
R_{\boldsymbol{u}}(\boldsymbol{\lambda})(\cdot)^p=\frac{\mathrm{d}\Gamma(\lambda_1u_1+\lambda_2u_2)}{\mathrm{d}\mu}\ \mu\text{-a.e.},
\end{equation}
hence
\begin{equation}\label{eq_LambdaR}
R_{\boldsymbol{u}}(\boldsymbol{e}_1)(\cdot)^p+R_{\boldsymbol{u}}(\boldsymbol{e}_2)(\cdot)^p=\frac{\mathrm{d}\Lambda_{\boldsymbol{u}}}{\mathrm{d}\mu}\ \mu\text{-a.e.}
\end{equation}
By (\ref{eq_Psiphimu}), for any $\boldsymbol{u}\in \{\boldsymbol{f}\}\cup\{\boldsymbol{h}^{(n)}\}_n$, any $x\in K\backslash N$, and any $\boldsymbol{\lambda}\in \mathbb{S}^1$,
\begin{equation}\label{eq_PsiR}
\Psi_{\boldsymbol{u}}(\boldsymbol{\lambda})(x)=
\begin{cases}
\frac{R_{\boldsymbol{u}}(\boldsymbol{\lambda})(x)^p}{R_{\boldsymbol{u}}(\boldsymbol{e}_1)(x)^p+R_{\boldsymbol{u}}(\boldsymbol{e}_2)(x)^p}&\text{if }R_{\boldsymbol{u}}(\boldsymbol{e}_1)(x)^p+R_{\boldsymbol{u}}(\boldsymbol{e}_2)(x)^p>0,\\
0&\text{if }R_{\boldsymbol{u}}(\boldsymbol{e}_1)(x)^p+R_{\boldsymbol{u}}(\boldsymbol{e}_2)(x)^p=0.
\end{cases}
\end{equation}
By Minkowski's inequality and (\ref{eq_Ruv}), for any $x\in K\backslash N$,
\begin{align}
&\lvert \left(R_{\boldsymbol{h}^{(n)}}(\boldsymbol{e}_1)(x)^p+R_{\boldsymbol{h}^{(n)}}(\boldsymbol{e}_2)(x)^p\right)^{\frac{1}{p}}-\left(R_{\boldsymbol{f}}(\boldsymbol{e}_1)(x)^p+R_{\boldsymbol{f}}(\boldsymbol{e}_2)(x)^p\right)^{\frac{1}{p}}\rvert\nonumber\\
&\le \left(\lvert R_{\boldsymbol{h}^{(n)}}(\boldsymbol{e}_1)(x)-R_{\boldsymbol{f}}(\boldsymbol{e}_1)(x)\rvert^p+\lvert R_{\boldsymbol{h}^{(n)}}(\boldsymbol{e}_2)(x)-R_{\boldsymbol{f}}(\boldsymbol{e}_2)(x)\rvert^p\right)^{\frac{1}{p}}\nonumber\\
&\le \left(R_{\boldsymbol{h}^{(n)}-\boldsymbol{f}}(\boldsymbol{e}_1)^p+R_{\boldsymbol{h}^{(n)}-\boldsymbol{f}}(\boldsymbol{e}_2)^p\right)^{\frac{1}{p}}.\label{eq_Rhnf}
\end{align}
Since $\mu(N)=0$, $\Lambda_{\boldsymbol{f}}(N)=0$, which implies $\Lambda_{\boldsymbol{f}}(A\backslash N)=\Lambda_{\boldsymbol{f}}(A)>0$. Combining the above inequality with (\ref{eq_Ruv}), (\ref{eq_PsiR}), and (\ref{eq_f2h00}), for any $x\in A\backslash N$,
\begin{align}
&R_{\boldsymbol{h}^{(n)}}(\boldsymbol{\lambda})(x)\ge R_{\boldsymbol{f}}(\boldsymbol{\lambda})(x)-\sum_{i=1}^2 \lvert \lambda_i\rvert R_{\boldsymbol{h}^{(n)}-\boldsymbol{f}}(\boldsymbol{e}_i)(x)\nonumber\\
&\ge 2 a^{\frac{1}{p}} \left(R_{\boldsymbol{f}}(\boldsymbol{e}_1)(x)^p+R_{\boldsymbol{f}}(\boldsymbol{e}_2)(x)^p\right)^{\frac{1}{p}}-\sum_{i=1}^2  R_{\boldsymbol{h}^{(n)}-\boldsymbol{f}}(\boldsymbol{e}_i)(x)\nonumber\\
&\ge 2 a^{\frac{1}{p}} \left(R_{\boldsymbol{h}^{(n)}}(\boldsymbol{e}_1)(x)^p+R_{\boldsymbol{h}^{(n)}}(\boldsymbol{e}_2)(x)^p\right)^{\frac{1}{p}}-2\sum_{i=1}^2  R_{\boldsymbol{h}^{(n)}-\boldsymbol{f}}(\boldsymbol{e}_i)(x).\label{eq_f2h03}
\end{align}

Let
$$N_{i,n}:=\left\{x\in K\backslash N:R_{\boldsymbol{h}^{(n)}-\boldsymbol{f}}(\boldsymbol{e}_i)(x)^p> \frac{a}{2^{3p}}\left(R_{\boldsymbol{f}}(\boldsymbol{e}_1)^p+R_{\boldsymbol{f}}(\boldsymbol{e}_2)^p\right)\right\},$$
then for any $x\in K\backslash (N\cup N_{1,n}\cup N_{2,n})$, (\ref{eq_Rhnf}) gives
\begin{align*}
&R_{\boldsymbol{h}^{(n)}-\boldsymbol{f}}(\boldsymbol{e}_1)(x)+R_{\boldsymbol{h}^{(n)}-\boldsymbol{f}}(\boldsymbol{e}_2)(x)\le \frac{a^{\frac{1}{p}}}{2^2} \left(R_{\boldsymbol{f}}(\boldsymbol{e}_1)^p+R_{\boldsymbol{f}}(\boldsymbol{e}_2)^p\right)^{\frac{1}{p}}\\
&\le \frac{a^{\frac{1}{p}}}{2^2} \left(\left(R_{\boldsymbol{h}^{(n)}}(\boldsymbol{e}_1)^p+R_{\boldsymbol{h}^{(n)}}(\boldsymbol{e}_2)^p\right)^{\frac{1}{p}}+\left(R_{\boldsymbol{h}^{(n)}-\boldsymbol{f}}(\boldsymbol{e}_1)(x)^p+R_{\boldsymbol{h}^{(n)}-\boldsymbol{f}}(\boldsymbol{e}_2)(x)^p\right)^{\frac{1}{p}}\right),
\end{align*}
hence
\begin{align}
&R_{\boldsymbol{h}^{(n)}-\boldsymbol{f}}(\boldsymbol{e}_1)(x)+R_{\boldsymbol{h}^{(n)}-\boldsymbol{f}}(\boldsymbol{e}_2)(x)\nonumber\\
&\le \frac{\frac{a^{\frac{1}{p}}}{2^2}}{1-\frac{a^{\frac{1}{p}}}{2^2}}\left(R_{\boldsymbol{h}^{(n)}}(\boldsymbol{e}_1)^p+R_{\boldsymbol{h}^{(n)}}(\boldsymbol{e}_2)^p\right)^{\frac{1}{p}}\le \frac{a^{\frac{1}{p}}}{2}\left(R_{\boldsymbol{h}^{(n)}}(\boldsymbol{e}_1)^p+R_{\boldsymbol{h}^{(n)}}(\boldsymbol{e}_2)^p\right)^{\frac{1}{p}}.\label{eq_f2h04}
\end{align}

Let $A_n:=A\backslash (N\cup N_{1,n}\cup N_{2,n})$, then by (\ref{eq_f2h03}) and (\ref{eq_f2h04}), for any $x\in A_n$,
\begin{equation}\label{eq_f2h05}
R_{\boldsymbol{h}^{(n)}}(\boldsymbol{\lambda})(x)\ge a^{\frac{1}{p}}\left(R_{\boldsymbol{h}^{(n)}}(\boldsymbol{e}_1)^p+R_{\boldsymbol{h}^{(n)}}(\boldsymbol{e}_2)^p\right)^{\frac{1}{p}}.
\end{equation}
Moreover, by (\ref{eq_LambdaR}) and (\ref{eq_GammaR}),
$$\Lambda_{\boldsymbol{f}}(N_{i,n})=\int_{N_{i,n}}\frac{\mathrm{d}\Lambda_{\boldsymbol{f}}}{\mathrm{d}\mu}(x) \mu(\mathrm{d}x)\le \frac{2^{3p}}{a}\int_{N_{i,n}}\frac{\mathrm{d}\Gamma(h^{(n)}_i-f_i)}{\mathrm{d}\mu}(x)\mu(\mathrm{d}x)\le \frac{2^{3p}}{a}\mathcal{E}(h^{(n)}_i-f_i)\to0,$$
hence there exists $M_1\ge1$ such that for any $n>M_1$, $\Lambda_{\boldsymbol{f}}(N_{i,n})<\frac{1}{4}\Lambda_{\boldsymbol{f}}(A)$, then $\Lambda_{\boldsymbol{f}}(A_{n})\ge \frac{1}{2}\Lambda_{\boldsymbol{f}}(A)$. By Lemma \ref{lem_Hin25}\ref{item_enedom2}, $\frac{\mathrm{d}\Lambda_{\boldsymbol{h}^{(n)}}}{\mathrm{d}\mu}$ is $L^1(K;\mu)$-convergent to $\frac{\mathrm{d}\Lambda_{\boldsymbol{f}}}{\mathrm{d}\mu}$, then for any $B\in \mathcal{B}(K)$,
$$\lvert \Lambda_{\boldsymbol{h}^{(n)}}(B)-\Lambda_{\boldsymbol{f}}(B)\rvert\le \int_B \lvert \frac{\mathrm{d}\Lambda_{\boldsymbol{h}^{(n)}}}{\mathrm{d}\mu}-\frac{\mathrm{d}\Lambda_{\boldsymbol{f}}}{\mathrm{d}\mu}\rvert \mathrm{d}\mu\le \int_K \lvert \frac{\mathrm{d}\Lambda_{\boldsymbol{h}^{(n)}}}{\mathrm{d}\mu}-\frac{\mathrm{d}\Lambda_{\boldsymbol{f}}}{\mathrm{d}\mu}\rvert \mathrm{d}\mu,$$
hence there exists $M_2\ge1$ such that for any $n>M_2$, the RHS $\le \frac{1}{4}\Lambda_{\boldsymbol{f}}(A)$. Take any $n>\max\{M_1,M_2\}$, then $\Lambda_{\boldsymbol{h}^{(n)}}(A_n)\ge \frac{1}{4}\Lambda_{\boldsymbol{f}}(A)>0$, and (\ref{eq_f2h05}), together with (\ref{eq_PsiR}), gives $\inf_{\boldsymbol{\lambda}\in\mathbb{S}^1}\Psi_{\boldsymbol{h}^{(n)}}(\boldsymbol{\lambda})(x)\ge a>0$ for any $x\in A_n$. Therefore, we obtain (\ref{eq_R1_h1}).

Second, we prove (\ref{eq_R1_h}). Let $\boldsymbol{h}=(h_1,h_2)\in \mathcal{H}^2_*$, $B\in \mathcal{B}(K)$ and $a>0$ be as in (\ref{eq_R1_h1}), then there exists $n\ge0$ such that $h_1,h_2\in \mathcal{H}_n$. Since $\Lambda_{\boldsymbol{h}}(B)>0$, there exists $w\in W_n$ such that $\Lambda_{\boldsymbol{h}}(B\cap K_w)>0$. Let $g_1=h_1\circ F_w$, $g_2=h_2\circ F_w$, and $A=F_w^{-1}(B\cap K_w)$, then $\boldsymbol{g}=(g_1,g_2)\in \mathcal{H}^2$, $\Lambda_{\boldsymbol{g}}(A)=\frac{1}{\rho^n}\Lambda_{\boldsymbol{h}}(B\cap K_w)>0$, and for any $x\in A$, we have $F_w(x)\in B$, hence $\Psi_{\boldsymbol{g}}(\cdot)(x)=\Psi_{\boldsymbol{h}}(\cdot)(F_w(x))\in C(\mathbb{S}^1)$ and
$$\inf_{\boldsymbol{\lambda}\in\mathbb{S}^1}\Psi_{\boldsymbol{g}}(\boldsymbol{\lambda})(x)=\inf_{\boldsymbol{\lambda}\in\mathbb{S}^1}\Psi_{\boldsymbol{h}}(\boldsymbol{\lambda})(F_w(x))\ge a>0.$$
\end{proof}

The following result gives \ref{item_step1}.

\begin{theorem}\label{thm_step1}
Suppose that (\ref{eq_R1_h}) holds for some $\boldsymbol{g}=(g_1,g_2)\in \mathcal{H}^2$, some Borel subset $A\in \mathcal{B}(K)$ with $\Lambda_{\boldsymbol{g}}(A)>0$, and some constant $a>0$. Then there exist $\boldsymbol{h}=(h_1,h_2)\in \mathcal{H}^2$ and $\Theta\in C(\mathbb{S}^1)$ such that $\Lambda_{\boldsymbol{h}}(K)=1$, $a\le \Theta\le 2^{\frac{p}{2}}$ on $\mathbb{S}^1$, and
$$\Psi_{\boldsymbol{h}}(\cdot)(x)=\Theta\text{ for $\Lambda_{\boldsymbol{h}}$-a.e. }x\in K.$$
\end{theorem}

\begin{proof}
The proof follows a blowup argument similar to the one in \cite[pp. 771--773]{Hin13}. By normalization, we may assume that $\Lambda_{\boldsymbol{g}}(K)=1$. Let $\delta:=\Lambda_{\boldsymbol{g}}(A)>0$. For any $n\ge0$, let
$$A_n:=\{w\in W_n:K_{\mathcal{N}_3(w)}\cap K^\partial\ne\emptyset\},$$
then $K_{A_n}\downarrow K^\partial$. By (\ref{eq_bdy_zero}), $\Lambda_{\boldsymbol{g}}(K^\partial)=0$, then there exists $n_0\ge0$ such that for any $n\ge n_0$, $\Lambda_{\boldsymbol{g}}(K_{A_n})<\frac{\delta}{3}$. For any $n\ge0$, let
$$G_n:=\left\{w\in W_n:\Lambda_{\boldsymbol{g}}(K_w)<\frac{\delta}{147}\Lambda_{\boldsymbol{g}}\left(K_{\mathcal{N}_3(w)}\right)\right\}.$$
Roughly speaking, the desired harmonic function should satisfy a mean value property across neighboring cells, and $G_n$ consists of the cells where this property fails. By (\ref{eq_bdy_zero}) and (\ref{eq_sup_nbhd}),
\begin{align*}
&\Lambda_{\boldsymbol{g}}\left(K_{G_n}\right)=\sum_{w\in G_n}\Lambda_{\boldsymbol{g}}(K_w)\le \frac{\delta}{147}\sum_{w\in G_n}\Lambda_{\boldsymbol{g}} \left(K_{\mathcal{N}_3(w)}\right)=\frac{\delta}{147}\sum_{w\in G_n}\sum_{v\in \mathcal{N}_3(w)}\Lambda_{\boldsymbol{g}} \left(K_v\right)\\
&\le \frac{\delta}{147}\sum_{v\in W_n} \Lambda_{\boldsymbol{g}}(K_v)\cdot \#\{w\in W_n:v\in \mathcal{N}_3(w)\}\le \frac{\delta}{147}\cdot (2\cdot3+1)^2\sum_{v\in W_n}\Lambda_{\boldsymbol{g}}(K_v)=\frac{\delta}{3}.
\end{align*}
Let $K_\infty:=\varliminf_{n\to+\infty}K_{G_n}=\cup_{N\ge0}\cap_{n\ge N}K_{G_n}$, then Fatou's lemma gives
$$\Lambda_{\boldsymbol{g}}(K_\infty)\le \varliminf_{n\to+\infty}\Lambda_{\boldsymbol{g}}(K_{G_n})\le \frac{\delta}{3}.$$
Let $B:=A\backslash (K_{A_{n_0}}\cup K_\infty\cup K^\partial_*)$, then $\Lambda_{\boldsymbol{g}}(B)\ge \delta-\frac{\delta}{3}-\frac{\delta}{3}=\frac{\delta}{3}$.

Recall that $\Psi_{\boldsymbol{g}}(\cdot)(x)\in C(\mathbb{S}^1)$ for any $x\in A$. Hence $\Psi_{\boldsymbol{g}}|_B:B\to (C(\mathbb{S}^1),\lVert \cdot\rVert_{C(\mathbb{S}^1)})$, $x\mapsto \Psi_{\boldsymbol{g}}(\cdot)(x)$ is measurable, where $\lVert \Theta\rVert_{C(\mathbb{S}^1)}=\max_{\boldsymbol{\lambda}\in \mathbb{S}^1}\lvert \Theta(\boldsymbol{\lambda})\rvert$ for $\Theta\in C(\mathbb{S}^1)$. Indeed, by the paragraph preceding (\ref{eq_Psiphi}), for any fixed $\Theta_0\in C(\mathbb{S}^1)$, $\lVert \Psi_{\boldsymbol{g}}(\cdot)(x)-\Theta_0\rVert_{C(\mathbb{S}^1)}=\sup_{\boldsymbol{\lambda}\in S}\lvert \Psi_{\boldsymbol{g}}(\boldsymbol{\lambda})(x)-\Theta_0(\boldsymbol{\lambda})\rvert$, where $S$ is a countable dense subset of $\mathbb{S}^1$. Then the pushforward $(\Psi_{\boldsymbol{g}}|_B)_*(\Lambda_{\boldsymbol{g}}|_B)$ of $\Lambda_{\boldsymbol{g}}|_B$ under $\Psi_{\boldsymbol{g}}|_B$ is a positive measure on $C(\mathbb{S}^1)$ satisfying $(\Psi_{\boldsymbol{g}}|_B)_*(\Lambda_{\boldsymbol{g}}|_B)(C(\mathbb{S}^1))=\Lambda_{\boldsymbol{g}}(B)>0$. Take
$$\Theta\in \mathrm{supp}\left((\Psi_{\boldsymbol{g}}|_B)_*(\Lambda_{\boldsymbol{g}}|_B)\right),$$
then $\Theta\in C(\mathbb{S}^1)$ satisfies $a\le \Theta\le 2^{\frac{p}{2}}$ on $\mathbb{S}^1$, where the lower and upper bounds follow from (\ref{eq_R1_h}) and (\ref{eq_Psi_upper}), respectively. For any $k\ge1$, let
\begin{equation}\label{eq_Bk}
B_k:=\left\{x\in B:\lVert \Psi_{\boldsymbol{g}}(\cdot)(x)-\Theta\rVert_{C(\mathbb{S}^1)}<\frac{1}{k}\right\},
\end{equation}
then $\Lambda_{\boldsymbol{g}}(B_k)>0$. Since $\Lambda_{\boldsymbol{g}}$ is a probability measure on $K$, let $\mathbb{E}$ denote the corresponding expectation. For any $n\ge0$, let $\mathcal{G}_n$ be the $\sigma$-algebra generated by $\{K_w:w\in W_n\}$, then $\{\mathcal{G}_n\}_{n\ge0}$ is a filtration, and $\mathcal{B}(K)=\mathcal{G}_\infty:=\sigma (\cup_{n\ge0}\mathcal{G}_n)$. For any $k\ge1$ and any $n\ge0$, let
$$Y^{(k)}_n:=\mathbb{E}[1_{B_k}|\mathcal{G}_n],$$
then $\{Y^{(k)}_n\}_{n\ge0}$ is a non-negative uniformly integrable martingale. By the martingale convergence theorem,
\begin{equation}\label{eq_Ykn_conv}
Y^{(k)}_n\to \mathbb{E}[1_{B_k}|\mathcal{G}_\infty]=1_{B_k}\ \Lambda_{\boldsymbol{g}}\text{-a.e. on $K$ and in }L^1(K;\Lambda_{\boldsymbol{g}}).
\end{equation}
Recall that for any $x\in K\backslash K^\partial_*$ and any $n\ge0$, $[x]_n$ is the unique word in $W_n$ such that $x\in K_{[x]_n}$. By (\ref{eq_bdy_zero}), $\Lambda_{\boldsymbol{g}}(K^\partial_*)=0$, hence for $\Lambda_{\boldsymbol{g}}$-a.e. $x\in K\backslash K^\partial_*$,
\begin{equation}\label{eq_Ykn_x}
Y^{(k)}_n(x)=
\begin{cases}
\frac{\Lambda_{\boldsymbol{g}}(B_k\cap K_{[x]_n})}{\Lambda_{\boldsymbol{g}}(K_{[x]_n})}&\text{if }\Lambda_{\boldsymbol{g}}(K_{[x]_n})>0,\\
0&\text{if }\Lambda_{\boldsymbol{g}}(K_{[x]_n})=0.
\end{cases}
\end{equation}
For any $k\ge1$, since $\Lambda_{\boldsymbol{g}}(B_k)>0$, by (\ref{eq_Ykn_conv}), there exist $x_k\in B_k$ and $N_k\ge0$ such that $1_{B_k}(x_k)=1$, and for any $n\ge N_k$, $Y^{(k)}_{n}(x_k)>1-\frac{1}{2^k}$, which implies $\Lambda_{\boldsymbol{g}}(B_k\cap K_{[x_k]_n})>0$. Since $x_k\in K\backslash K_\infty =\cap_{N\ge0}\cup_{n\ge N}(K\backslash K_{G_n})$, there exist infinitely many $n\ge0$ such that $\Lambda_{\boldsymbol{g}}(K_{[x_k]_n})\ge \frac{\delta}{147}\Lambda_{\boldsymbol{g}}(K_{\mathcal{N}_3([x_k]_n)})$. In particular, there exists $n_k\ge \max\{N_k,n_0\}$ such that $\Lambda_{\boldsymbol{g}}(B_k\cap K_{[x_k]_{n_k}})>0$, and
$$Y^{(k)}_{n_k}(x_k)>1-\frac{1}{2^k},\qquad \Lambda_{\boldsymbol{g}}(K_{[x_k]_{n_k}})\ge \frac{\delta}{147}\Lambda_{\boldsymbol{g}}(K_{\mathcal{N}_3([x_k]_{n_k})}).$$
For any $k\ge1$, let $\boldsymbol{g}^{(k)}=(g^{(k)}_1,g^{(k)}_2)\in \mathcal{H}^2$ be given by
$$g^{(k)}_i=\left(\frac{\rho^{n_k}}{\Lambda_{\boldsymbol{g}}(K_{[x_k]_{n_k}})}\right)^{\frac{1}{p}}\left(g_i-\dashint_{K_{[x_k]_{n_k}}}g_i \mathrm{d}m\right),\ i=1,2,$$
then $\int_{K_{[x_k]_{n_k}}}g^{(k)}_i \mathrm{d}m=0$,
\begin{equation}\label{eq_gk_1}
\Lambda_{\boldsymbol{g}^{(k)}}(K_{[x_k]_{n_k}})=\frac{\rho^{n_k}}{\Lambda_{\boldsymbol{g}}(K_{[x_k]_{n_k}})}\Lambda_{\boldsymbol{g}}(K_{[x_k]_{n_k}})=\rho^{n_k},
\end{equation}
and
\begin{equation}\label{eq_gk_bd}
\Lambda_{\boldsymbol{g}^{(k)}}(K_{\mathcal{N}_3([x_k]_{n_k})})=\frac{\rho^{n_k}}{\Lambda_{\boldsymbol{g}}(K_{[x_k]_{n_k}})}\Lambda_{\boldsymbol{g}}(K_{\mathcal{N}_3([x_k]_{n_k})})\le \frac{147}{\delta}\rho^{n_k}.
\end{equation}
Let $h^{(k)}_i=g^{(k)}_i\circ F_{[x_k]_{n_k}}$, then $\boldsymbol{h}^{(k)}=(h^{(k)}_1,h^{(k)}_2)\in \mathcal{H}^2$ satisfies
$$\Lambda_{\boldsymbol{h}^{(k)}}(K)=\sum_{i=1}^2 \Gamma(h^{(k)}_i)(K)=\frac{1}{\rho^{n_k}}\sum_{i=1}^2 \Gamma(g^{(k)}_i)(K_{[x_k]_{n_k}})=\frac{1}{\rho^{n_k}}\Lambda_{\boldsymbol{g}^{(k)}}(K_{[x_k]_{n_k}})=1.$$
Let $\widehat{B}_k=F_{[x_k]_{n_k}}^{-1}(B_k)$, then
\begin{align*}
&\Lambda_{\boldsymbol{h}^{(k)}}(\widehat{B}_k)=\sum_{i=1}^2 \Gamma(h^{(k)}_i)(\widehat{B}_k)=\frac{1}{\rho^{n_k}}\sum_{i=1}^2 \Gamma(g^{(k)}_i)(B_k\cap K_{[x_k]_{n_k}})\\
&=\frac{\Lambda_{\boldsymbol{g}}(B_k\cap K_{[x_k]_{n_k}})}{\Lambda_{\boldsymbol{g}}(K_{[x_k]_{n_k}})}=Y^{(k)}_{n_k}(x_{k})>1-\frac{1}{2^k},
\end{align*}
hence
\begin{equation}\label{eq_hk_bd}
\Lambda_{\boldsymbol{h}^{(k)}}(K\backslash\widehat{B}_k)<\frac{1}{2^k}.
\end{equation}
Since $x_k\not\in K_{A_{n_0}}$, Proposition \ref{prop_Hin418} applies to $\{g^{(k)}_i\}_k$ and $\{[x_k]_{n_k}\}$. By (\ref{eq_gk_bd}), there exist a subsequence $\{\boldsymbol{h}^{k(j)}\}_j$ of $\{\boldsymbol{h}^{(k)}\}_k$, and some $\boldsymbol{h}=(h_1,h_2)\in \mathcal{H}^2$, such that $\{h^{k(j)}_i\}_j$ is $\mathcal{E}_1$-convergent to $h_i$ for any $i=1,2$. By Lemma \ref{lem_Hin25}\ref{item_enedom2}, after passing to a further subsequence, we may assume that $\mathcal{E}(h^{k(j)}_i-h_i)<\frac{1}{2^{j}}$ for any $j$ and any $i=1,2$, and that
\begin{equation}\label{eq_lim_hkj}
\lim_{j\to+\infty}\Psi_{\boldsymbol{h}^{(k(j))}}(\boldsymbol{\lambda})(x)=\Psi_{\boldsymbol{h}}(\boldsymbol{\lambda})(x)\text{ for any }\boldsymbol{\lambda}\in S,\text{ for $\Lambda_{\boldsymbol{h}}$-a.e. }x\in K.
\end{equation}
Then by (\ref{eq_gk_1}), $\Lambda_{\boldsymbol{h}}(K)=1$, and by sub-linearity \ref{item_pDS4} and (\ref{eq_hk_bd}),
\begin{align*}
&\Lambda_{\boldsymbol{h}}(K\backslash \widehat{B}_{k(j)})^{\frac{1}{p}}\le \lvert \Lambda_{\boldsymbol{h}^{(k(j))}}(K\backslash \widehat{B}_{k(j)})^{\frac{1}{p}}-\Lambda_{\boldsymbol{h}}(K\backslash \widehat{B}_{k(j)})^{\frac{1}{p}}\rvert+\Lambda_{\boldsymbol{h}^{(k(j))}}(K\backslash \widehat{B}_{k(j)})^{\frac{1}{p}}\\
&\overset{(\star)}{\scalebox{2}[1]{$\le$}}\sum_{i=1}^2 \lvert \left(\Gamma(h_i^{(k(j))})(K\backslash \widehat{B}_{k(j)})\right)^{\frac{1}{p}}-\left(\Gamma(h_i)(K\backslash \widehat{B}_{k(j)})\right)^{\frac{1}{p}}\rvert+\Lambda_{\boldsymbol{h}^{(k(j))}}(K\backslash \widehat{B}_{k(j)})^{\frac{1}{p}}\\
&\le\sum_{i=1}^2 \Gamma(h_i^{(k(j))}-h_i)(K\backslash \widehat{B}_{k(j)})^{\frac{1}{p}}+\Lambda_{\boldsymbol{h}^{(k(j))}}(K\backslash \widehat{B}_{k(j)})^{\frac{1}{p}}\\
&\le\sum_{i=1}^2\mathcal{E}(h^{(k(j))}_i-h_i)^{\frac{1}{p}}+\Lambda_{\boldsymbol{h}^{(k(j))}}(K\backslash \widehat{B}_{k(j)})^{\frac{1}{p}}\le \frac{3}{2^{\frac{j}{p}}},
\end{align*}
where $(\star)$ follows from the following elementary result: For any $a_1,a_2,b_1,b_2\in [0,+\infty)$,
$$\lvert (a_1+a_2)^{\frac{1}{p}}-(b_1+b_2)^{\frac{1}{p}}\rvert\le \lvert a_1^{\frac{1}{p}}-b_1^{\frac{1}{p}}\rvert+\lvert a_2^{\frac{1}{p}}-b_2^{\frac{1}{p}}\rvert.$$
By the Borel--Cantelli lemma, $\Lambda_{\boldsymbol{h}}(\varliminf_{j\to+\infty}\widehat{B}_{k(j)})=1$, that is, for $\Lambda_{\boldsymbol{h}}$-a.e. $x\in K$, there exist infinitely many $j$ such that $x\in \widehat{B}_{k(j)}=F^{-1}_{[x_{k(j)}]_{n_{k(j)}}}(B_{k(j)})$, then $F_{[x_{k(j)}]_{n_{k(j)}}}(x)\in B_{k(j)}$. By (\ref{eq_Bk}),
$$\lVert \Psi_{\boldsymbol{h}^{(k(j))}}(\cdot)(x)-\Theta\rVert_{C(\mathbb{S}^1)}=\lVert \Psi_{\boldsymbol{g}^{(k(j))}}(\cdot)(F_{[x_{k(j)}]_{n_{k(j)}}}(x))-\Theta\rVert_{C(\mathbb{S}^1)}<\frac{1}{k(j)}$$
for infinitely many $j$. Combining this with (\ref{eq_lim_hkj}), we obtain $\Psi_{\boldsymbol{h}}(\cdot)(x)=\Theta$ for $\Lambda_{\boldsymbol{h}}$-a.e. $x\in K$.
\end{proof}

In \ref{item_step2}, we consider the pushforwards of energy measures under $\boldsymbol{h}=(h_1,h_2):K\to \mathbb{R}^2$, which are measures on $\mathbb{R}^2$.

\begin{proposition}[{$p$-version of \cite[Proposition 3.4]{Hin13}}]\label{prop_Hin34}
Let $\boldsymbol{h}=(h_1,h_2)\in \mathcal{H}^2$ satisfy $\Lambda_{\boldsymbol{h}}(K)>0$. Assume that there exists a function $\Theta\in C(\mathbb{S}^1)$, bounded above and below by positive constants, such that $\Psi_{\boldsymbol{h}}(\cdot)(x)=\Theta$ for $\Lambda_{\boldsymbol{h}}$-a.e. $x\in K$. Let $f\in \mathcal{F}_D$ be non-negative and bounded. Then the pushforward $\widetilde{\boldsymbol{h}}_*(\widetilde{f}^p\Lambda_{\boldsymbol{h}})$ of $\widetilde{f}^p\Lambda_{\boldsymbol{h}}$ under $\widetilde{\boldsymbol{h}}=(\widetilde{h}_1,\widetilde{h}_2):K\to \mathbb{R}^2$ is absolutely continuous with respect to $\mathcal{L}^2$, and its Radon--Nikodym derivative $\xi:=\frac{\mathrm{d}\widetilde{\boldsymbol{h}}_*(\widetilde{f}^p\Lambda_{\boldsymbol{h}})}{\mathrm{d}\mathcal{L}^2}$ satisfies $\xi^{\frac{1}{p}}\in W^{1,p}(\mathbb{R}^2)$.
\end{proposition}

\begin{proof}
For notational convenience, we may assume that all functions in $\mathcal{F}$ are quasi-continuous. Obviously, $\boldsymbol{h}$ is non-zero. For any $\varphi\in C^\infty_c(\mathbb{R}^2)$, since $\boldsymbol{h}=(h_1,h_2)\in \mathcal{H}^2$ and $f^p\varphi(\boldsymbol{h})\in \mathcal{F}_D$, by (\ref{eq_Leibniz}) and (\ref{eq_chain}), for any $i=1,2$,
\begin{equation}\label{eq_Hin34a}
0=\mathcal{E}(h_i;f^p\varphi(\boldsymbol{h}))=\int_K \varphi(\boldsymbol{h})\mathrm{d}\Gamma(h_i;f^p)+\sum_{j=1}^2\int_K f^p\frac{\partial\varphi}{\partial x_j}(\boldsymbol{h}) \mathrm{d}\Gamma(h_i;h_j).
\end{equation}
We extend $\Theta$ from $\mathbb{S}^1$ to $\mathbb{R}^2$, still denoted by $\Theta$, by
$$\Theta(\boldsymbol{\lambda}):=
\begin{cases}
\lvert \boldsymbol{\lambda}\rvert^p\Theta \left(\frac{\boldsymbol{\lambda}}{\lvert\boldsymbol{\lambda}\rvert}\right)&\text{if }\boldsymbol{\lambda}\ne0,\\
0&\text{if }\boldsymbol{\lambda}=0.
\end{cases}
$$
By assumption, there exist constants $\delta_1,\delta_2>0$ such that
\begin{equation}\label{eq_Hin34b}
\delta_1 \lvert \boldsymbol{\lambda}\rvert^p\le \Theta(\boldsymbol{\lambda})\le \delta_2 \lvert \boldsymbol{\lambda}\rvert^p\text{ for any }\boldsymbol{\lambda}\in \mathbb{R}^2.
\end{equation}
Then by the same argument as in the paragraph preceding (\ref{eq_Psiphi}),
\begin{equation*}
\frac{\mathrm{d}\Gamma(\lambda_1h_1+\lambda_2h_2)}{\mathrm{d}\Lambda_{\boldsymbol{h}}}(x)=\Theta(\boldsymbol{\lambda})\text{ for any $\boldsymbol{\lambda}=(\lambda_1,\lambda_2)\in \mathbb{R}^2$, for }\Lambda_{\boldsymbol{h}}\text{-a.e. $x\in K$},
\end{equation*}
hence
\begin{equation}\label{eq_Hin34c}
\mathcal{E}(\lambda_1h_1+\lambda_2h_2)=\Theta(\boldsymbol{\lambda})\Lambda_{\boldsymbol{h}}(K)\text{ for any }\boldsymbol{\lambda}=(\lambda_1,\lambda_2)\in \mathbb{R}^2,
\end{equation}
which gives for any $i,j=1,2$,
\begin{equation}\label{eq_Hin34d}
\mathcal{E}(h_i;h_j)=\frac{1}{p}\lim_{\lambda_j\to0}\frac{\mathcal{E}(h_i+\lambda_jh_j)-\mathcal{E}(h_i)}{\lambda_j}=\frac{1}{p}\frac{\partial\Theta}{\partial\lambda_j}(\boldsymbol{e}_i)\Lambda_{\boldsymbol{h}}(K),\qquad \mathcal{E}(h_i)=\Theta(\boldsymbol{e}_i)\Lambda_{\boldsymbol{h}}(K).
\end{equation}
Moreover, for any $B\in \mathcal{B}(K)$,
$$\Gamma(h_i;h_j)(B)=\frac{1}{p}\lim_{\lambda_j\to0}\frac{\Gamma(h_i+\lambda_jh_j)(B)-\Gamma(h_i)(B)}{\lambda_j}=\frac{1}{p}\frac{\partial\Theta}{\partial\lambda_j}(\boldsymbol{e}_i)\Lambda_{\boldsymbol{h}}(B),$$
hence
$$\frac{\mathrm{d}\Gamma(h_i;h_j)}{\mathrm{d}\Lambda_{\boldsymbol{h}}}=\frac{1}{p}\frac{\partial\Theta}{\partial\lambda_j}(\boldsymbol{e}_i),\qquad\frac{\mathrm{d}\Gamma(h_i;h_i)}{\mathrm{d}\Lambda_{\boldsymbol{h}}}=\frac{\mathrm{d}\Gamma(h_i)}{\mathrm{d}\Lambda_{\boldsymbol{h}}}=\Theta(\boldsymbol{e}_i),$$
where the right-hand sides are fixed numbers, then (\ref{eq_Hin34a}) gives
$$\int_{\mathbb{R}^2}\varphi \mathrm{d}\left(\boldsymbol{h}_*\Gamma(h_i;f^p)\right)=-\frac{1}{p}\sum_{j=1}^2 \frac{\partial\Theta}{\partial \lambda_j}(\boldsymbol{e}_i)\int_{\mathbb{R}^2} \frac{\partial\varphi}{\partial x_j}\mathrm{d}\left(\boldsymbol{h}_*(f^p\Lambda_{\boldsymbol{h}})\right),$$
Let $\kappa=\boldsymbol{h}_*(f^p\Lambda_{\boldsymbol{h}})$ and $\kappa_i=\boldsymbol{h}_*\Gamma(h_i;f^p)$ for $i=1,2$, then
\begin{equation}\label{eq_Hin34e}
\left(
\begin{array}{c}
\int_{\mathbb{R}^2}\varphi \mathrm{d}\kappa_1\\
\int_{\mathbb{R}^2}\varphi \mathrm{d}\kappa_2
\end{array}
\right)
=-
\left(
\begin{array}{cc}
\frac{1}{p}\frac{\partial\Theta}{\partial\lambda_1}(\boldsymbol{e}_1)&\frac{1}{p}\frac{\partial\Theta}{\partial\lambda_2}(\boldsymbol{e}_1)\\
\frac{1}{p}\frac{\partial\Theta}{\partial\lambda_1}(\boldsymbol{e}_2)&\frac{1}{p}\frac{\partial\Theta}{\partial\lambda_2}(\boldsymbol{e}_2)
\end{array}
\right)
\left(
\begin{array}{c}
\int_{\mathbb{R}^2}\frac{\partial\varphi}{\partial x_1}\mathrm{d}\kappa\\
\int_{\mathbb{R}^2}\frac{\partial\varphi}{\partial x_2}\mathrm{d}\kappa
\end{array}
\right).
\end{equation}

We claim that the $(2\times2)$ matrix on the right-hand side is invertible. It suffices to show that it has a non-zero determinant. Indeed, suppose that its determinant is zero, then by (\ref{eq_Hin34b}) and (\ref{eq_Hin34d}),
$$\mathcal{E}(h_1;h_2)\mathcal{E}(h_2;h_1)=\mathcal{E}(h_1)\mathcal{E}(h_2)>0.$$
Since
$$\lvert\mathcal{E}(h_1;h_2)\rvert\le \mathcal{E}(h_1)^{\frac{p-1}{p}}\mathcal{E}(h_2)^{\frac{1}{p}},\qquad \lvert\mathcal{E}(h_2;h_1)\rvert\le \mathcal{E}(h_2)^{\frac{p-1}{p}}\mathcal{E}(h_1)^{\frac{1}{p}},$$
equality holds in these two inequalities. Let
$$f_1:=\frac{h_1}{\mathcal{E}(h_1)^{\frac{1}{p}}},\qquad f_2:=\frac{h_2}{\mathcal{E}(h_2)^{\frac{1}{p}}}\mathrm{sgn}\left(\mathcal{E}(h_1;h_2)\right),$$
then $\mathcal{E}(f_1)=\mathcal{E}(f_2)=\mathcal{E}(f_1;f_2)=1$. By linearity and the triangle inequality,
\begin{align*}
&1=\mathcal{E}\left(f_1;\frac{f_1+f_2}{2}\right)\le \mathcal{E}(f_1)^{\frac{p-1}{p}}\mathcal{E}\left(\frac{f_1+f_2}{2}\right)^{\frac{1}{p}}\\
&=\mathcal{E}\left(\frac{f_1+f_2}{2}\right)^{\frac{1}{p}}\le \frac{1}{2}\left(\mathcal{E}(f_1)^{\frac{1}{p}}+\mathcal{E}(f_2)^{\frac{1}{p}}\right)=1,
\end{align*}
hence $\mathcal{E}\left(\frac{f_1+f_2}{2}\right)=1$. By \ref{eq_Cla}, $\mathcal{E}(f_1-f_2)=0$, hence there exists a constant $c\in \mathbb{R}$ such that $\mathcal{E}(h_1-ch_2)=0$. By (\ref{eq_Hin34c}) and (\ref{eq_Hin34b}),
$$0=\mathcal{E}(h_1-ch_2)=\Theta((1,-c))\Lambda_{\boldsymbol{h}}(K)\ge \delta_1 (1+c^2)^\frac{p}{2}\Lambda_{\boldsymbol{h}}(K)>0,$$
a contradiction. Therefore, (\ref{eq_Hin34e}) gives
$$
\left(
\begin{array}{c}
\int_{\mathbb{R}^2}\frac{\partial\varphi}{\partial x_1}\mathrm{d}\kappa\\
\int_{\mathbb{R}^2}\frac{\partial\varphi}{\partial x_2}\mathrm{d}\kappa
\end{array}
\right)
=-
\left(
\begin{array}{cc}
a_{11}&a_{12}\\
a_{21}&a_{22}
\end{array}
\right)
\left(
\begin{array}{c}
\int_{\mathbb{R}^2}\varphi \mathrm{d}\kappa_1\\
\int_{\mathbb{R}^2}\varphi \mathrm{d}\kappa_2
\end{array}
\right),
$$
where $a_{11},a_{12},a_{21},a_{22}\in \mathbb{R}$ are constants. Hence, for any $i=1,2$ and any $\varphi\in C^\infty_c(\mathbb{R}^2)$,
\begin{equation}\label{eq_Hin34f}
\int_{\mathbb{R}^2}\frac{\partial\varphi}{\partial x_i}\mathrm{d}\kappa=-\sum_{j=1}^2 a_{ij}\int_{\mathbb{R}^2}\varphi \mathrm{d}\kappa_j,
\end{equation}
which implies, in the sense of distributions, $\frac{\partial}{\partial x_i}\kappa=\sum_{j=1}^2a_{ij}\kappa_j$ for $i=1,2$. By \cite[Lemma I.7.2.2.1]{BH91}, $\kappa\ll \mathcal{L}^2$. Let $\xi:=\frac{\mathrm{d}\kappa}{\mathrm{d}\mathcal{L}^2}$ denote the Radon--Nikodym derivative.

For any $i=1,2$ and any $B\in \mathcal{B}(\mathbb{R}^2)$, by (\ref{eq_Holder}),
\begin{align}
&\lvert \kappa_i(B)\rvert=p\lvert \int_{\boldsymbol{h}^{-1}(B)}f^{p-1}\mathrm{d}\Gamma(h_i;f)\rvert\le p \left(\int_{\boldsymbol{h}^{-1}(B)}f^p \mathrm{d}\Gamma(h_i)\right)^{\frac{p-1}{p}}\left(\int_{\boldsymbol{h}^{-1}(B)}\mathrm{d}\Gamma(f)\right)^{\frac{1}{p}}\nonumber\\
&\le p \left(\int_{\boldsymbol{h}^{-1}(B)}f^p \mathrm{d}\Lambda_{\boldsymbol{h}}\right)^{\frac{p-1}{p}}\left(\int_{\boldsymbol{h}^{-1}(B)}\mathrm{d}\Gamma(f)\right)^{\frac{1}{p}}=p \kappa(B)^{\frac{p-1}{p}}(\boldsymbol{h}_*\Gamma(f))(B)^{\frac{1}{p}},\label{eq_Hin34g}
\end{align}
hence $\kappa_i\ll \kappa$, which implies $\kappa_i\ll \mathcal{L}^2$. Let $\xi_i:=\frac{\mathrm{d}\kappa_i}{\mathrm{d}\mathcal{L}^2}$ denote the Radon--Nikodym derivative, then for any $i=1,2$,
$$\frac{\partial \xi}{\partial x_i}=\sum_{j=1}^2 a_{ij}\xi_j,$$
hence $\xi\in W^{1,1}(\mathbb{R}^2)$. By (\ref{eq_Hin34g}), applying \cite[Theorem 3.22]{Fol99}, we have
$$\lvert \xi_i\rvert\le p \xi^{\frac{p-1}{p}} \left(\frac{\mathrm{d}\left(\boldsymbol{h}_*\Gamma(f)\right)_{\mathrm{ac}}}{\mathrm{d}\mathcal{L}^2}\right)^{\frac{1}{p}}\ \mathcal{L}^2\text{-a.e. on }\mathbb{R}^2,$$
where $\left(\boldsymbol{h}_*\Gamma(f)\right)_{\mathrm{ac}}$ is the absolutely continuous part in the Lebesgue decomposition of $\boldsymbol{h}_*\Gamma(f)$ with respect to $\mathcal{L}^2$. For any $\varepsilon>0$, let $\gamma_\varepsilon(t)=(t+\varepsilon)^{\frac{1}{p}}-\varepsilon^{\frac{1}{p}}$, then
$$\frac{\partial \gamma_\varepsilon(\xi)}{\partial x_i}=\frac{1}{p}\frac{1}{(\xi+\varepsilon)^{1-\frac{1}{p}}}\frac{\partial\xi}{\partial x_i}=\sum_{j=1}^2 \frac{a_{ij}}{p}\frac{\xi_j}{(\xi+\varepsilon)^{1-\frac{1}{p}}},$$
where
$$\int_{\mathbb{R}^2}\frac{\lvert \xi_j\rvert^p}{(\xi+\varepsilon)^{p-1}}\mathrm{d}\mathcal{L}^2\le p^p \left(\boldsymbol{h}_*\Gamma(f)\right)_{\mathrm{ac}}(\mathbb{R}^2)\le p^p \mathcal{E}(f)<+\infty,$$
hence $\sup_{\varepsilon>0}\lVert \frac{\partial \gamma_\varepsilon(\xi)}{\partial x_i}\rVert_{L^p(\mathbb{R}^2)}<+\infty$ for $i=1,2$. Since $\gamma_\varepsilon(t)\uparrow t^{\frac{1}{p}}$ as $\varepsilon\downarrow0$ and $\xi\in L^1(\mathbb{R}^2)$, $\gamma_\varepsilon(\xi)\to \xi^{\frac{1}{p}}$ in $L^p(\mathbb{R}^2)$ as $\varepsilon\downarrow0$. By \cite[Section 9.1, Remark 4(i), p. 264]{Bre11}, $\xi^{\frac{1}{p}}\in W^{1,p}(\mathbb{R}^2)$.
\end{proof}

To complete \ref{item_step2}, we need some further preparation. Let $\rho$ and $\xi$ be measurable functions on $\mathbb{R}^2$ satisfying that $0\le \rho\le \xi$ $\mathcal{L}^2$-a.e. on $\mathbb{R}^2$ and $\xi^{\frac{1}{p}}\in W^{1,p}(\mathbb{R}^2)$. For any $u\in C_c^1(\mathbb{R}^2)$, let
$$Q^\rho(u):=\int_{\mathbb{R}^2}\lvert \nabla u\rvert^p (\rho+1)\mathrm{d}\mathcal{L}^2,\qquad Q^\rho_1(u):=\int_{\mathbb{R}^2}(\lvert \nabla u\rvert^p+\lvert u\rvert^p) (\rho+1)\mathrm{d}\mathcal{L}^2.$$
Let $\mathrm{Dom}(Q^\rho)$ be the $Q^\rho_1$-closure of $C_c^1(\mathbb{R}^2)$ in $L^p(\mathbb{R}^2;(\rho+1)\mathcal{L}^2)$, then $(Q^\rho,\mathrm{Dom}(Q^\rho))$ is a local regular $p$-energy on $L^p(\mathbb{R}^2;(\rho+1)\mathcal{L}^2)$. Let $(Q,W^{1,p}(\mathbb{R}^2))$ be the standard regular $p$-energy on $L^p(\mathbb{R}^2)$ given by
$$Q(u):=\int_{\mathbb{R}^2}\lvert \nabla u\rvert^p \mathrm{d}\mathcal{L}^2,\qquad Q_1(u):=\int_{\mathbb{R}^2}(\lvert \nabla u \rvert^p+\lvert u\rvert^p) \mathrm{d}\mathcal{L}^2.$$
Let $\mathrm{Cap}^{\rho}$ and $\mathrm{Cap}^{1,p}$ denote the capacities associated with ${Q}^{\rho}$ and $Q$, respectively; see \cite[Section 8]{Yan25a} for the relevant definitions and results on capacity and quasi-continuity for general local regular $p$-energies. For any $x\in \mathbb{R}^2$ and any $r>0$, let
$$B(x,r)=\{y\in \mathbb{R}^2:\lvert x-y\rvert_{\mathbb{R}^2}<r\},\qquad \overline{B}(x,r)=\{y\in \mathbb{R}^2:\lvert x-y\rvert_{\mathbb{R}^2}\le r\},$$
where $\lvert \cdot\rvert_{\mathbb{R}^2}$ denotes the Euclidean metric on $\mathbb{R}^2$.

The following result parallels {\cite[Lemma 3.9]{Hin13}}. Its proof carries over verbatim, with squaring and taking square roots replaced by taking $p$th powers and $p$th roots, respectively, and $\mathrm{Cap}^{1,2}$ throughout replaced by $\mathrm{Cap}^{1,p}$. We therefore omit the proof.

\begin{lemma}\label{lem_Hin39}
For $\mathrm{Cap}^{1,p}$-q.e. $x\in \mathbb{R}^2$, we have
$$\sup_{r>0}\dashint_{B(x,r)}\rho(y)\mathrm{d}y<+\infty.$$
\end{lemma}

The following result parallels \cite[Proposition 3.10]{Hin13}. The proof in \cite{Hin13} relies on the characterization of $\mathrm{Dom}(Q^{\xi})$ in the Dirichlet form setting given by R\"ockner--Zhang or Eberle. Since we have not found a $p$-analogue of this characterization in the existing literature, for completeness, we give a proof in Section \ref{sec_Hin310}.

\begin{proposition}\label{prop_Hin310}
Let $B\subseteq \mathbb{R}^2$. Then $\mathrm{Cap}^{\rho}(B)=0$ if and only if $\mathrm{Cap}^{1,p}(B)=0$.
\end{proposition}

\begin{lemma}[{\cite[Lemma 3.11]{Hin13} with $d=2$}]\label{lem_Hin311}
Let $\kappa$ be a positive Radon measure on $\mathbb{R}^2$. Then
$$\kappa\left(\left\{x\in \mathbb{R}^2:\varliminf_{r\downarrow0}\frac{\kappa(\overline{B}(x,r))}{r^2}=0\right\}\right)=0.$$
\end{lemma}

Let us recall the definition of Hausdorff measure on $\mathbb{R}^2$ as follows. Let $A\subseteq \mathbb{R}^2$ and $s>0$. For any $\delta>0$, let
$$\mathcal{H}^s_\delta(A):=\inf\left\{\sum_{i=1}^\infty \left(\mathrm{diam}(E_i)\right)^s:A\subseteq \cup_{i=1}^\infty E_i,\mathrm{diam}(E_i)\le \delta\right\}.$$
Then the $s$-dimensional Hausdorff measure $\mathcal{H}^s(A)$ of $A$ is defined as $\mathcal{H}^{s}(A)=\lim_{\delta\downarrow0}\mathcal{H}^s_\delta(A)$.

\begin{lemma}\label{lem_Hin312}
Let $p\in(1,2)$ and let $\kappa$ be a positive Radon measure on $\mathbb{R}^2$. Then
\begin{equation}\label{eq_Hin312}
\lim_{r\downarrow0}\frac{\kappa \left(\overline{B}(x,r)\right)}{r^{2-p}}=0\text{ for $\mathrm{Cap}^{1,p}$-q.e. }x\in \mathbb{R}^2.
\end{equation}
\end{lemma}

\begin{proof}
The proof follows the argument of {\cite[Lemma 3.12]{Hin13}}. For any $n\ge1$, let
$$A_n:=\left\{x\in B(0,n):\varlimsup_{r\downarrow0}\frac{\kappa \left(\overline{B}(x,r)\right)}{r^{2-p}}>\frac{1}{n}\right\}.$$
For any $\delta\in(0,1)$, let
$$\mathcal{G}_{n,\delta}:=\left\{\overline{B}(x,r):x\in A_n,r\in(0,\delta), \frac{\kappa \left(\overline{B}(x,r)\right)}{r^{2-p}}>\frac{1}{n}\right\}.$$
Then for any $x\in A_n$, there exists $r\in(0,\delta)$ such that $\overline{B}(x,r)\in \mathcal{G}_{n,\delta}$, hence $A_n\subseteq \cup_{B\in \mathcal{G}_{n,\delta}}B$. By the $5B$-covering lemma (see \cite[Theorem 1.2]{Hei01}), there exists a countable family of disjoint balls $\{\overline{B}(x_i,r_i)\}_i\subseteq \mathcal{G}_{n,\delta}$ such that $A_n\subseteq\cup_{B\in \mathcal{G}_{n,\delta}}B\subseteq \cup_i\overline{B}(x_i,5r_i)$, hence
$$\mathcal{H}^{2-p}_{10\delta}(A_n)\le \sum_i (10r_i)^{2-p}\le 10^{2-p}n\sum_i \kappa \left(\overline{B}(x_i,r_i)\right)\le 10^{2-p}n \kappa \left(\overline{B}(0,n+1)\right).$$
Letting $\delta\downarrow0$, we have $\mathcal{H}^{2-p}(A_n)\le 10^{2-p}n \kappa \left(\overline{B}(0,n+1)\right)<+\infty$. By \cite[Theorem 5.1.9]{AH96}, $\mathrm{Cap}^{1,p}(A_n)=0$ for any $n\ge1$, hence $\mathrm{Cap}^{1,p}(\cup_{n\ge1}A_n)=0$, which gives (\ref{eq_Hin312}).
\end{proof}

\begin{proposition}\label{prop_Hin313}
Let $\boldsymbol{h}=(h_1,h_2)\in \mathcal{F}^2$ satisfy $\Lambda_{\boldsymbol{h}}(K)>0$ and $\boldsymbol{h}_*\Lambda_{\boldsymbol{h}}\ll \mathcal{L}^2$. Assume that its Radon--Nikodym derivative $\rho:=\frac{\mathrm{d}(\widetilde{\boldsymbol{h}}_*\Lambda_{\boldsymbol{h}})}{\mathrm{d}\mathcal{L}^2}$ is dominated by a non-negative function $\xi$ with $\xi^{\frac{1}{p}}\in W^{1,p}(\mathbb{R}^2)$. Then the measure $\boldsymbol{h}_*m$ does not concentrate on any $\mathrm{Cap}^{1,p}$-null set. More precisely, if $B\in \mathcal{B}(\mathbb{R}^2)$ satisfies $\mathrm{Cap}^{1,p}(B)=0$, then $(\boldsymbol{h}_*m)(\mathbb{R}^2\backslash B)>0$.
\end{proposition}

\begin{proof}
We adapt the argument of \cite[Proposition 3.13]{Hin13} to the non-linear setting. Suppose that there exists $B\in \mathcal{B}(\mathbb{R}^2)$ satisfying that $\mathrm{Cap}^{1,p}(B)=(\boldsymbol{h}_*m)(\mathbb{R}^2\backslash B)=0$. By regularity, there exists a sequence of compact sets $\{B_n\}_{n\ge1}$ such that $B_n\subseteq B_{n+1}\subseteq B$ for any $n$, and
\begin{equation}\label{eq_Hin313a}
(\boldsymbol{h}_*m)(\mathbb{R}^2\backslash B_n)=(\boldsymbol{h}_*m)(B\backslash B_n)\downarrow0.
\end{equation}
By Proposition \ref{prop_Hin310}, $\mathrm{Cap}^{\rho}(B)=\mathrm{Cap}^{\rho}(B_n)=0$ for any $n\ge1$. Then there exists $f_n\in C^1_c(\mathbb{R}^2)$ satisfying that $0\le f_n\le 1+\frac{1}{n}$ on $\mathbb{R}^2$ and $1\le f_n\le 1+\frac{1}{n}$ on $B_n$ for any $n\ge1$, such that
\begin{equation}\label{eq_Hin313b}
\lim_{n\to+\infty}\int_{\mathbb{R}^2}\left(\lvert \nabla f_n\rvert^p +\lvert f_n\rvert^p\right)(\rho+1)\mathrm{d}\mathcal{L}^2=0.
\end{equation}

Recall the following basic property for $p$-energy measures. Let $\boldsymbol{h}=(h_1,\ldots,h_n)\in \mathcal{F}^n$ and let $\phi\in C_c^1(\mathbb{R}^n)$ satisfy $\phi(0)=0$, then $\phi(\boldsymbol{h})\in \mathcal{F}$ and
\begin{equation}\label{eq_Hin313c}
\Gamma(\phi(\boldsymbol{h}))\le n^{p-1} \left(\sum_{i=1}^n \lvert \frac{\partial\phi}{\partial x_i}(\boldsymbol{h})\rvert^p\right)\sum_{i=1}^n \Gamma(h_i).
\end{equation}
Indeed, by the chain rule,
$$\Gamma(\phi(\boldsymbol{h}))=\Gamma(\phi(\boldsymbol{h});\phi(\boldsymbol{h}))=\sum_{i=1}^n  \frac{\partial\phi}{\partial x_i}(\boldsymbol{h})\Gamma(\phi(\boldsymbol{h});h_i).$$
Let $\mu=\Gamma(\phi(\boldsymbol{h}))+\sum_{i=1}^n \Gamma(h_i)$, then by the preceding identity,
$$\frac{\mathrm{d}\Gamma(\phi(\boldsymbol{h}))}{\mathrm{d}\mu}\le\sum_{i=1}^n \left(\frac{\mathrm{d}\Gamma(\phi(\boldsymbol{h}))}{\mathrm{d}\mu}\right)^{\frac{p-1}{p}}\left(\lvert \frac{\partial\phi}{\partial x_i}(\boldsymbol{h})\rvert^p \frac{\mathrm{d}\Gamma(h_i)}{\mathrm{d}\mu}\right)^{\frac{1}{p}},$$
hence
\begin{align*}
&\frac{\mathrm{d}\Gamma(\phi(\boldsymbol{h}))}{\mathrm{d}\mu}\le \left(\sum_{i=1}^n\lvert \frac{\partial\phi}{\partial x_i}(\boldsymbol{h})\rvert \left(\frac{\mathrm{d}\Gamma(h_i)}{\mathrm{d}\mu}\right)^{\frac{1}{p}}\right)^p\le \left(\sum_{i=1}^n\lvert \frac{\partial\phi}{\partial x_i}(\boldsymbol{h})\rvert^{\frac{p}{p-1}}\right)^{p-1}\left(\sum_{i=1}^n \frac{\mathrm{d}\Gamma(h_i)}{\mathrm{d}\mu}\right)\\
&\le \left(n\max_{i=1,\ldots,n}\lvert \frac{\partial\phi}{\partial x_i}(\boldsymbol{h})\rvert^{\frac{p}{p-1}}\right)^{p-1}\left(\sum_{i=1}^n \frac{\mathrm{d}\Gamma(h_i)}{\mathrm{d}\mu}\right)\le n^{p-1}\left(\sum_{i=1}^n \lvert\frac{\partial\phi}{\partial x_i}(\boldsymbol{h})\rvert^p\right) \sum_{i=1}^n \frac{\mathrm{d}\Gamma(h_i)}{\mathrm{d}\mu}.
\end{align*}

Take any fixed $u\in C_c^1(\mathbb{R}^2)$ with $u(0)=0$ such that $\mathcal{E}(u(\boldsymbol{h}))>0$. Indeed, since $\Lambda_{\boldsymbol{h}}(K)>0$, there exist $R>0$ and $i=1,2$ such that $\Gamma(h_i)(\{\lvert \boldsymbol{h}\rvert<R\})>0$. Let $u\in C_c^1(\mathbb{R}^2)$ satisfy $u(x_1,x_2)=x_i$ for any $(x_1,x_2)\in B(0,R)$, then by the strongly local property, $\mathcal{E}(u(\boldsymbol{h}))\ge \Gamma(h_i)(\{\lvert \boldsymbol{h}\rvert<R\})>0$. Noting that $(uf_n)(\boldsymbol{h})\in \mathcal{F}$, by (\ref{eq_Hin313c}) and (\ref{eq_Hin313b}),
\begin{align*}
&\mathcal{E}((uf_n)(\boldsymbol{h}))\le 2^{p-1}\sum_{i=1}^2\int_{K} \lvert\frac{\partial (uf_n)}{\partial x_i}(\boldsymbol{h})\rvert^p \mathrm{d}\Lambda_{\boldsymbol{h}}=2^{p-1}\sum_{i=1}^2\int_{\mathbb{R}^2} \lvert\frac{\partial (uf_n)}{\partial x_i}\rvert^p \rho\mathrm{d}\mathcal{L}^2\\
&\le2^{2p-2}\sum_{i=1}^2\int_{\mathbb{R}^2}\left( \lvert u\rvert^p  \lvert \frac{\partial f_n}{\partial x_i}\rvert^p+\lvert f_n\rvert^p \lvert \frac{\partial u}{\partial x_i}\rvert^p\right)\rho \mathrm{d}\mathcal{L}^2\\
&\le 2^{2p-1} \left(\max_{\mathbb{R}^2}\lvert u\rvert^p \int_{\mathbb{R}^2}\lvert \nabla f_n\rvert^p \rho \mathrm{d}\mathcal{L}^2+\max_{\mathbb{R}^2}\lvert \nabla u\rvert^p \int_{\mathbb{R}^2}\lvert  f_n\rvert^p \rho \mathrm{d}\mathcal{L}^2\right)\to0,
\end{align*}
hence $\{(uf_n)(\boldsymbol{h})\}_{n\ge1}$ is $\mathcal{E}$-convergent to $0$. By (\ref{eq_Hin313a}),
\begin{align*}
&\int_K \lvert (uf_n)(\boldsymbol{h})-u(\boldsymbol{h})\rvert^p \mathrm{d}m= \left(\int_{B_n}+\int_{\mathbb{R}^2\backslash B_n}\right) \lvert u(1-f_n)\rvert^p \mathrm{d}(\boldsymbol{h}_*m)\\
&\le \max_{\mathbb{R}^2}\lvert u\rvert^p\left( \frac{1}{n^p}m(K)+2^p (\boldsymbol{h}_*m)(\mathbb{R}^2\backslash B_n)\right)\to0,
\end{align*}
hence $\{(uf_n)(\boldsymbol{h})\}_{n\ge1}$ is $L^p(K;m)$-convergent to $u(\boldsymbol{h})$. Combining this with its $\mathcal{E}$-convergence to $0$ and the completeness of $(\mathcal{F},\mathcal{E}_1)$, we have $\mathcal{E}(u(\boldsymbol{h}))=0$, a contradiction.
\end{proof}

We now turn to \ref{item_step3}, which completes the proof of Theorem \ref{thm_R1}.

\begin{proof}[Proof of Theorem \ref{thm_R1}]
For notational convenience, we may assume that all functions in $\mathcal{F}$ are quasi-continuous. By Remark \ref{rmk_R1}, suppose that (\ref{eq_R1_f}) holds, then by Proposition \ref{prop_Hin49} and Theorem \ref{thm_step1}, there exist $\boldsymbol{h}=(h_1,h_2)\in \mathcal{H}^2$ with $\Lambda_{\boldsymbol{h}}(K)>0$, and $\Theta\in C(\mathbb{S}^1)$, bounded above and below by positive constants, such that $\Psi_{\boldsymbol{h}}(\cdot)(x)=\Theta$ for $\Lambda_{\boldsymbol{h}}$-a.e. $x\in K$. Since $\Lambda_{\boldsymbol{h}}(K)>0$, and $\Lambda_{\boldsymbol{h}}(K^\partial)=0$ by (\ref{eq_bdy_zero}), there exists $w\in W_*$ such that $K_w\cap K^{\partial}=\emptyset$ and $\Lambda_{\boldsymbol{h}}(K_w)>0$. By regularity, there exists $f\in \mathcal{F}_D\cap C(K)$ with $0\le f\le 1$ on $K$ and $f=1$ on $K_w$. By Proposition \ref{prop_Hin34}, ${\boldsymbol{h}}_*({f}^p \Lambda_{\boldsymbol{h}})$ is absolutely continuous with respect to $\mathcal{L}^2$, and its Radon--Nikodym derivative satisfies $\left(\frac{\mathrm{d}{\boldsymbol{h}}_*({f}^p \Lambda_{\boldsymbol{h}})}{\mathrm{d}\mathcal{L}^2}\right)^{\frac{1}{p}}\in W^{1,p}(\mathbb{R}^2)$. Replacing $\boldsymbol{h}$ by $\boldsymbol{h}\circ F_w$, we may assume that $\Lambda_{\boldsymbol{h}}(K)>0$, ${\boldsymbol{h}}_*\Lambda_{\boldsymbol{h}}$ is absolutely continuous with respect to $\mathcal{L}^2$, and its Radon--Nikodym derivative $\rho:=\frac{\mathrm{d}(\boldsymbol{h}_*\Lambda_{\boldsymbol{h}})}{\mathrm{d}\mathcal{L}^2}\ne0$ is dominated by a non-negative function $\xi$ with $\xi^{\frac{1}{p}}\in W^{1,p}(\mathbb{R}^2)$, that is, $\rho\le \xi$ $\mathcal{L}^2$-a.e.

By \cite[Corollary 3.7 and Theorem 4.1]{BK13}, together with \cite[Theorem 8.5, Remark 8.6 and Definition 8.7]{MS25},
$$\beta_p > d_h \ \Longleftrightarrow\  p > \mathrm{dim}_{\mathrm{ARC}}(K,d),\qquad \beta_p \le d_h \ \Longleftrightarrow\  p \le \mathrm{dim}_{\mathrm{ARC}}(K,d),$$
where $\mathrm{dim}_{\mathrm{ARC}}(K,d)$ is the Ahlfors regular conformal dimension of $(K,d)$, which satisfies
$$1+\frac{\log 2}{\log 3}\le \mathrm{dim}_{\mathrm{ARC}}(K,d) < \frac{\log 8}{\log 3}=d_h;$$
see \cite[Remark 9.17]{MS25}.

We divide the proof into the following two cases.

\noindent\textbf{Case 1:} $\beta_p>d_h$. By the Morrey--Sobolev inequality \cite[Lemma 5.2]{Yan25c}, there exists an integer $k\ge0$ such that for any $n\ge0$, any $w\in W_n$, and any $x,y\in K_w$,
$$\lvert \boldsymbol{h}(x)-\boldsymbol{h}(y)\rvert^p\lesssim \lvert x-y\rvert^{\beta_p-d_h}\Lambda_{\boldsymbol{h}}\left(K_{\mathcal{N}_k(w)}\right),$$
where $\mathcal{N}_k(w)$ is given in (\ref{eq_k_nbhd}), hence
$$\mathrm{diam}\left(\boldsymbol{h}(K_w)\right)\lesssim \frac{1}{3^{\frac{\beta_p-d_h}{p}n}}\Lambda_{\boldsymbol{h}}\left(K_{\mathcal{N}_k(w)}\right)^{\frac{1}{p}}.$$
Choose
$$s\in \left(\frac{pd_h}{\beta_p},p\wedge2\right].$$
Such an $s$ exists by assumption; in particular, $\frac{pd_h}{\beta_p}<2$ follows from $d_h<2$ and $\beta_p>p$. By H\"older's inequality and (\ref{eq_sup_nbhd}),
\begin{align*}
&\sum_{w\in W_n}\mathrm{diam}\left(\boldsymbol{h}(K_w)\right)^s\lesssim\frac{1}{3^{(\beta_p-d_h)\frac{s}{p}n}}\sum_{w\in W_n}\Lambda_{\boldsymbol{h}}\left(K_{\mathcal{N}_k(w)}\right)^{\frac{s}{p}}\\
&\le\frac{1}{3^{(\beta_p-d_h)\frac{s}{p}n}}\left(\sum_{w\in W_n}\Lambda_{\boldsymbol{h}}\left(K_{\mathcal{N}_k(w)}\right)\right)^{\frac{s}{p}} 3^{d_h \left(1-\frac{s}{p}\right)n}\lesssim 3^{\left(d_h-\frac{\beta_p}{p}s\right)n}\Lambda_{\boldsymbol{h}}(K)^{\frac{s}{p}}\to0,
\end{align*}
hence $\mathcal{H}^s(\boldsymbol{h}(K))=0$. Since $s\le2$, $\mathcal{L}^2(\boldsymbol{h}(K))=0$. Using $\boldsymbol{h}_*\Lambda_{\boldsymbol{h}}\ll \mathcal{L}^2$, we have $0=(\boldsymbol{h}_*\Lambda_{\boldsymbol{h}})(\boldsymbol{h}(K))=\Lambda_{\boldsymbol{h}}(K)$, contradicting $\Lambda_{\boldsymbol{h}}(K)>0$.

\noindent\textbf{Case 2:} $\beta_p\le d_h$, or equivalently, $p\in(1,\mathrm{dim}_{\mathrm{ARC}}(K,d)]\subseteq(1,2)$. We follow the argument in the proof of \cite[Theorem 3.5(ii)]{Hin13}. The Poincar\'e inequality \hyperref[eq_PI]{$\text{PI}(\beta_p)$} implies the following Sobolev inequality (see, for example, \cite[Theorem 9.1]{BCLS95} or \cite[Theorem 4.5]{Bau24}): Let
$$p^*:=
\begin{cases}
\frac{pd_h}{d_h-\beta_p}&\text{if }\beta_p<d_h,\\
\text{any number in }(p,+\infty)&\text{if }\beta_p=d_h.
\end{cases}
$$
There exists a constant $C>0$ such that for any $f\in \mathcal{F}$,
\begin{equation}\label{eq_Sobolev}
\left(\int_K \lvert f\rvert^{p^*}\mathrm{d}m\right)^{\frac{p}{p^*}}\le C \left(\mathcal{E}(f)+\int_K \lvert f\rvert^p \mathrm{d}m\right).
\end{equation}

First, we prove that there exists $x_0\in \boldsymbol{h}(K)$ such that
\begin{enumerate}[label=(\alph*),ref=(\alph*)]
\item\label{item_Hina}
$$b:=\sup_{r>0}\dashint_{B(x_0,r)}\rho(y) \mathrm{d}y<+\infty;$$
\item\label{item_Hinb}
$$(\boldsymbol{h}_*m)\left(\overline{B}(x_0,r)\right)=o(r^{2-p})\text{ as }r\downarrow0;$$
\item\label{item_Hinc} there exist $a>0$ and $r_0>0$ such that $(\boldsymbol{h}_*m)(\overline{B}(x_0,r))\ge ar^2$ for any $r\in(0,r_0]$.
\end{enumerate}

Indeed, first, we have $\mathrm{Cap}^{\rho}(\boldsymbol{h}(K))>0$; otherwise, $0=\int_{\boldsymbol{h}(K)}\rho \mathrm{d}\mathcal{L}^2=(\boldsymbol{h}_*\Lambda_{\boldsymbol{h}})(\boldsymbol{h}(K))=\Lambda_{\boldsymbol{h}}(K)>0$, a contradiction. Hence by Proposition \ref{prop_Hin310}, $\mathrm{Cap}^{1,p}(\boldsymbol{h}(K))>0$. By Lemmas \ref{lem_Hin39} and \ref{lem_Hin312},
\begin{align*}
&\mathrm{Cap}^{1,p}\left(\left\{x\in \mathbb{R}^2:\sup_{r>0}\dashint_{B(x,r)}\rho(y) \mathrm{d}y=+\infty\right\}\right)\\
&=\mathrm{Cap}^{1,p}\left(\left\{x\in \mathbb{R}^2:\varlimsup_{r\downarrow0}\frac{(\boldsymbol{h}_*m)\left(\overline{B}(x,r)\right)}{r^{2-p}}>0\right\}\right)=0.
\end{align*}
Let $B$ be the union of these two sets, then $\mathrm{Cap}^{1,p}(B)=0$. Then by Proposition \ref{prop_Hin313}, $(\boldsymbol{h}_*m)(\boldsymbol{h}(K)\backslash B)=(\boldsymbol{h}_*m)(\mathbb{R}^2\backslash B)>0$. Let
$$A:=\left\{x\in \mathbb{R}^2:\varliminf_{r\downarrow0}\frac{(\boldsymbol{h}_*m)\left(\overline{B}(x,r)\right)}{r^2}=0\right\},$$
then by Lemma \ref{lem_Hin311}, $(\boldsymbol{h}_*m)(A)=0$. Therefore, $\boldsymbol{h}(K)\backslash(A\cup B)\ne\emptyset$, and any $x_0\in \boldsymbol{h}(K)\backslash(A\cup B)$ satisfies \ref{item_Hina}, \ref{item_Hinb}, and \ref{item_Hinc}. By replacing $\boldsymbol{h}$ by $\boldsymbol{h}-x_0$, we may assume that $x_0=0$.

For any $r\ge0$, let
$$\varphi(r)=(\lvert \boldsymbol{h}\rvert_*\Lambda_{\boldsymbol{h}})([0,r])=(\boldsymbol{h}_*\Lambda_{\boldsymbol{h}})(\overline{B}(0,r)),$$
then by \ref{item_Hina},
\begin{equation}\label{eq_phi_bdd}
\varphi(r)=\int_{\overline{B}(0,r)}\rho(y)\mathrm{d}y\le \pi br^2.
\end{equation}
Take a smooth function $g:[0,+\infty)\to[0,+\infty)$ satisfying that
$$g(t)=
\begin{cases}
1&\text{if }t\in[0,1],\\
\frac{2}{3}-\frac{\log \frac{t}{2}}{3\log \frac{3}{2}}&\text{if }t\in[2,3],\\
0&\text{if }t\in[4,+\infty),
\end{cases}
$$
and $-\frac{2}{t}\le g'(t)\le0$ for any $t\in[1,4]$. For any $\delta>0$, let $g_\delta=g(\frac{\cdot}{\delta})$ and $f_\delta=g_\delta(\lvert \boldsymbol{h}\rvert)$, then $f_\delta\in \mathcal{F}$. By (\ref{eq_Hin313c}) and (\ref{eq_phi_bdd}),
\begin{align}
&\mathcal{E}(f_\delta)\le 2^{p}\int_{K}  \lvert g_\delta'(\lvert \boldsymbol{h}\rvert)\rvert^p \mathrm{d}\Lambda_{\boldsymbol{h}}=2^p \int_{\delta}^{4\delta}\lvert g_\delta'(t)\rvert^p \mathrm{d}\varphi(t)\le 2^{2p}\int_{\delta}^{4\delta}\frac{1}{t^p}\mathrm{d}\varphi(t)\nonumber\\
&=2^{2p} \left(\frac{\varphi(t)}{t^p}|_\delta^{4\delta}-\int_\delta^{4\delta}\varphi(t)\mathrm{d}\frac{1}{t^p}\right)\le 2^{2p}\left(\frac{\pi b(4\delta)^2}{(4\delta)^p}+p\int_\delta^{4\delta}\frac{\pi b t^2}{t^{p+1}}\mathrm{d}t\right)=C_1 \delta^{2-p},\label{eq_fdelta_E}
\end{align}
where $C_1>0$ is a constant depending only on $b,p$. By \ref{item_Hinb},
\begin{equation}\label{eq_fdelta_Lp}
\int_K \lvert f_\delta\rvert^p \mathrm{d}m\le\int_{\{\lvert \boldsymbol{h}\rvert\le 4\delta\}} \mathrm{d}m=( \boldsymbol{h}_*m)(\overline{B}(0,4\delta))=o(\delta^{2-p})\text{ as }\delta\downarrow0.
\end{equation}
By \ref{item_Hinc}, for any $\delta\in(0,r_0]$,
$$\int_K\lvert f_\delta\rvert^{p^*} \mathrm{d}m\ge\int_{\{\lvert \boldsymbol{h}\rvert\le \delta\}}\mathrm{d}m=(\boldsymbol{h}_*m)(\overline{B}(0,\delta))\ge a\delta^2.$$
Combining this with (\ref{eq_fdelta_E}), (\ref{eq_fdelta_Lp}), and (\ref{eq_Sobolev}), as $\delta\downarrow0$,
$$ (a\delta^2)^{\frac{p}{p^*}}\le C \left(C_1\delta^{2-p}+o(\delta^{2-p})\right),$$
hence $\delta^{\frac{2p}{p^*}+p-2}\lesssim1$ as $\delta\downarrow0$. However, if $\beta_p<d_h$, then $\frac{2p}{p^*}+p-2=p-\frac{2\beta_p}{d_h}<0$, due to $d_h<2$ and $\beta_p>p$ (see \cite[Corollary 2.2]{Yan25e} or \cite[Theorem 9.8]{KS24aV5}); if $\beta_p=d_h$, then by taking $p^*>p$ sufficiently large, we still have $\frac{2p}{p^*}+p-2<0$. Therefore, $\lim_{\delta\downarrow0}\delta^{\frac{2p}{p^*}+p-2}=+\infty\lesssim1$, a contradiction.
\end{proof}

\section{A zero-one law}\label{sec_01law}

In this section, we continue to work with the canonical $p$-energy $(\mathcal{E}^{(c)},\mathcal{F}^{(c)})$ and its energy measure $\Gamma^{(c)}$. We omit the superscript $(c)$ when no ambiguity arises. The main result of this section is as follows.

\begin{theorem}[A zero-one law]\label{thm_01law}
Let $A\in \mathcal{B}(K)$ and let $\mu$ be a minimal energy-dominant measure. Assume that
\begin{equation}\label{eq_ss01}
\mu(F_i^{-1}(A)\mathbin{\triangle} A)=0\text{ for any }i=0,1,\ldots,7,\qquad \mu(\psi^{-1}(A)\mathbin{\triangle} A)=0\text{ for any }\psi\in D_4.
\end{equation}
Then either $\mu(A)=0$ or $\mu(K\backslash A)=0$.
\end{theorem}

\begin{remark}\label{rmk_01law}
If $A\in \mathcal{B}(K)$ satisfies $\mu(A)=0$, then
\begin{equation}\label{eq_muzero}
\mu(F_w(A))=\mu(F_w^{-1}(A))=0\text{ for any }w\in W_*,\qquad \mu(\psi(A))=0\text{ for any }\psi\in D_4.
\end{equation}
Hence (\ref{eq_ss01}) is equivalent to the identities
$$1_A\circ F_i=1_A\ \mu\text{-a.e. for any }i=0,1,\ldots,7,\qquad 1_A\circ \psi=1_A\ \mu\text{-a.e. for any }\psi\in D_4,$$
and by induction, it implies $\mu(F_w^{-1}(A)\mathbin{\triangle} A)=0$, or equivalently, $1_A\circ F_w=1_A$ $\mu$-a.e., for any $w\in W_*$.

We give the proof of (\ref{eq_muzero}) as follows. Recall that for any $A\in \mathcal{B}(K)$, $\mu(A)=0$ if and only if $\Gamma(f)(A)=0$ for any $f\in \mathcal{F}$. First, for any $\psi\in D_4$ and any $f\in \mathcal{F}$, by (\ref{eq_Gamma_D4}), $\Gamma(f)(\psi(A))=\Gamma(f\circ\psi)(A)=0$, hence $\mu(\psi(A))=0$. Second, for any $n\ge0$ and any $w\in W_n$, by (\ref{eq_Gamma_ss}) and (\ref{eq_bdy_zero}),
$$\Gamma(f)(F_w(A))=\rho^n\sum_{v\in W_n}\Gamma(f\circ F_v)(F_v^{-1}(F_w(A)))=\rho^n\Gamma(f\circ F_w)(A)=0,$$
hence $\mu(F_w(A))=0$.

To prove $\mu(F_w^{-1}(A))=0$, we first define folding maps $\varphi_w:K\to K_w$ for all $w\in W_*$ inductively on the length of $w$. First, $\varphi_\emptyset:K\to K_\emptyset=K$ is the identity. For any $i=0,1,\ldots,7$, we define the folding map $\varphi_i:K\to K_i$ as follows. The map $\varphi_i$ is the identity on $K_i$. For a point $x$ in another cell $K_j$, we successively reflect across the sides shared by adjacent cells, following a route from $K_j$ to $K_i$. The composition of these reflections maps $x$ into $K_i$; see Figure \ref{fig_fold}, where the arrows indicate the orientations after the successive reflections. It is easy to verify that the resulting map $\varphi_i$ is well-defined. Assume that the folding maps $\varphi_w:K\to K_w$ are defined for all $w\in W_*$ with $\lvert w\rvert\le n$. Then for $n+1$, for any $w\in W_n$ and any $i=0,1,\ldots,7$, let $\varphi_{wi}:K\to K_{wi}$ be given by
$$\varphi_{wi}:=F_w\circ \varphi_i\circ F_w^{-1}\circ \varphi_w.$$
By induction, we obtain the folding maps $\varphi_w$ for all $w\in W_*$. It is easy to see that for any $w,v\in W_*$,
$$\varphi_{wv}=F_w\circ \varphi_v\circ F_w^{-1}\circ \varphi_w=F_{wv}\circ(F_v^{-1}\circ \varphi_v)\circ (F_w^{-1}\circ \varphi_w).$$
Or equivalently, let $\theta_w:=F_w^{-1}\circ \varphi_w:K\to K$, then for any $w,v\in W_*$,
$$\theta_{wv}=\theta_v\circ \theta_w.$$
Moreover, by induction, for any $w\in W_*$,
\begin{equation}\label{eq_varphiF}
\varphi_w\circ F_w=F_w.
\end{equation}
For explicit coordinate formulas for the folding maps, see \cite[Definition 8.9]{MS25}; see also \cite[Subsection 2.2]{BBKT10} and \cite[Section 4.3]{Kig23}.

\begin{figure}[htbp]
\centering
\begin{tikzpicture}[scale=1.5]
  \tikzset{
    arrow mark/.style={
      postaction={decorate},
      decoration={markings,#1}
    },
    one/.style={arrow mark={
      mark=at position .50 with
        {\arrow{Stealth[length=5pt,width=5pt]}}
    }},
    two/.style={arrow mark={
      mark=at position .40 with
        {\arrow{Stealth[length=5pt,width=5pt]}},
      mark=at position .60 with
        {\arrow{Stealth[length=5pt,width=5pt]}}
    }},
    three/.style={arrow mark={
      mark=at position .32 with
        {\arrow{Stealth[length=5pt,width=5pt]}},
      mark=at position .50 with
        {\arrow{Stealth[length=5pt,width=5pt]}},
      mark=at position .68 with
        {\arrow{Stealth[length=5pt,width=5pt]}}
    }},
    four/.style={arrow mark={
      mark=at position .26 with
        {\arrow{Stealth[length=5pt,width=5pt]}},
      mark=at position .42 with
        {\arrow{Stealth[length=5pt,width=5pt]}},
      mark=at position .58 with
        {\arrow{Stealth[length=5pt,width=5pt]}},
      mark=at position .74 with
        {\arrow{Stealth[length=5pt,width=5pt]}}
    }}
  }

  \draw[thick,three] (1,3)--(0,3);
  \draw[thick,three] (1,3)--(2,3);
  \draw[thick,three] (3,3)--(2,3);

  \draw[thick,one] (0,2)--(1,2);
  \draw[thick,one] (2,2)--(1,2);
  \draw[thick,one] (2,2)--(3,2);

  \draw[thick,three] (1,1)--(0,1);
  \draw[thick,three] (1,1)--(2,1);
  \draw[thick,three] (3,1)--(2,1);

  \draw[thick,one] (0,0)--(1,0);
  \draw[thick,one] (2,0)--(1,0);
  \draw[thick,one] (2,0)--(3,0);

  \draw[thick,four] (0,3)--(0,2);
  \draw[thick,four] (0,1)--(0,2);
  \draw[thick,four] (0,1)--(0,0);

  \draw[thick,two] (1,2)--(1,3);
  \draw[thick,two] (1,2)--(1,1);
  \draw[thick,two] (1,0)--(1,1);

  \draw[thick,four] (2,3)--(2,2);
  \draw[thick,four] (2,1)--(2,2);
  \draw[thick,four] (2,1)--(2,0);

  \draw[thick,two] (3,2)--(3,3);
  \draw[thick,two] (3,2)--(3,1);
  \draw[thick,two] (3,0)--(3,1);

  \node at (.5,.5)   {$K_{0}$};
  \node at (1.5,.5)  {$K_{1}$};
  \node at (2.5,.5)  {$K_{2}$};
  \node at (2.5,1.5) {$K_{3}$};
  \node at (2.5,2.5) {$K_{4}$};
  \node at (1.5,2.5) {$K_{5}$};
  \node at (.5,2.5)  {$K_{6}$};
  \node at (.5,1.5) {$K_{7}$};
\end{tikzpicture}
\caption{The folding map $\varphi_i:K\to K_i$}
\label{fig_fold}
\end{figure}
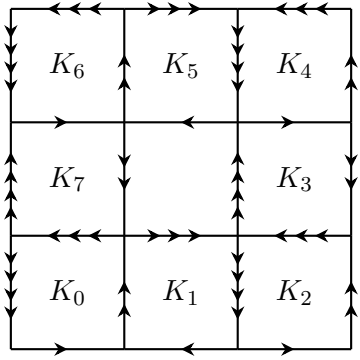

Now we prove $\mu(F_w^{-1}(A))=0$. For any $f\in \mathcal{F}\cap C(K)$, let $g:=f\circ \theta_w=f\circ F_w^{-1}\circ\varphi_w$, then by \ref{item_SSEe}, $g\in \mathcal{F}\cap C(K)$. By (\ref{eq_varphiF}), $g\circ F_w=(f\circ F_w^{-1}\circ\varphi_w)\circ F_w=f$, then (\ref{eq_Gamma_ss}) gives
$$0=\Gamma(g)(A)=\rho^n\sum_{v\in W_n}\Gamma(g\circ F_v)(F_v^{-1}(A))\ge \rho^n \Gamma(g\circ F_w)(F_w^{-1}(A))=\rho^n\Gamma(f)(F_w^{-1}(A)).$$
Hence $\Gamma(f)(F_w^{-1}(A))=0$. For any $f\in \mathcal{F}$, since $\mathcal{F}\cap C(K)$ is $\mathcal{E}_1$-dense in $\mathcal{F}$, there exists $\{f_n\}\subseteq \mathcal{F}\cap C(K)$ that $\mathcal{E}_1$-converges to $f$. For any $n$, $\Gamma(f_n)(F_w^{-1}(A))=0$, hence
\begin{align*}
&\Gamma(f)(F_w^{-1}(A))^{\frac{1}{p}}=\lvert \Gamma(f)(F_w^{-1}(A))^{\frac{1}{p}}-\Gamma(f_n)(F_w^{-1}(A))^{\frac{1}{p}}\rvert\\
&\le \Gamma(f-f_n)(F_w^{-1}(A))^{\frac{1}{p}}\le \mathcal{E}(f-f_n)^{\frac{1}{p}}\to0.
\end{align*}
Therefore, $\mu(F_w^{-1}(A))=0$. This completes the proof of (\ref{eq_muzero}).
\end{remark}

We say that a functional $Q:\mathcal{F}\to[0,+\infty)$ is a self-similar component (of $\mathcal{E}$) if
\begin{enumerate}[label=({Q}\arabic*),ref=({Q}\arabic*)]
\item\label{item_Q1} $Q^{\frac{1}{p}}$ is a seminorm on $\mathcal{F}$;
\item\label{item_Q2} for any $f\in\mathcal{F}$, $Q(f)\le\mathcal{E}(f)$;
\item\label{item_Q3} for any $f\in\mathcal{F}$, $Q(f)=\rho\sum_{i=0}^7Q(f\circ F_i)$;
\item\label{item_Q4} for any $f\in\mathcal{F}$ and any $\psi\in D_4$, $Q(f\circ\psi)=Q(f)$;
\item\label{item_Q5} for any $f\in\mathcal{F}$ and any Lipschitz function $\varphi:\mathbb{R}\to\mathbb{R}$, $Q(\varphi\circ f)\le \mathrm{Lip}(\varphi)^pQ(f)$ .
\end{enumerate}
Its kernel is $Z(Q):=\{f\in\mathcal{F}:Q(f)=0\}$.

We have some basic properties of the kernel. Their proofs are elementary and are therefore omitted.

\begin{lemma}\label{lem_kernel}
Let $Q$ be a self-similar component with kernel $Z=Z(Q)$.
\begin{enumerate}[label=(\arabic*),ref=(\arabic*)]
\item\label{item_kernel1} $Z$ is a closed linear subspace of $\mathcal{F}$ containing the constants.
\item\label{item_kernel2} For any $f\in Z$ and any $w\in W_*$, $f\circ F_w\in Z$.
\item\label{item_kernel3} For any $f\in Z$ and any $\psi\in D_4$, $f\circ\psi\in Z$.
\item\label{item_kernel4} For any $f\in Z\cap C(K)$ and any Lipschitz function $\varphi$, $\varphi\circ f\in Z\cap C(K)$. The set $Z\cap C(K)$ is an algebra and a lattice.
\item\label{item_kernel5} If $g\in C(K)$ and $n\ge0$ satisfy $g\circ F_w\in Z$ for any $ w\in W_n$, then $g\in Z\cap C(K)$.
\end{enumerate}
\end{lemma}

\begin{lemma}\label{lem_corridor}
Let $Q$ be a self-similar component with kernel $Z=Z(Q)$. If there exists $f\in Z\cap C(K)$ which is non-constant on one of the four sides $\ell_L,\ell_R,\ell_B,\ell_T$, then there exists $g\in Z\cap C(K)$ such that
$$g=0\text{ on }\ell_L,\qquad g=1\text{ on }\ell_R.$$
\end{lemma}

\begin{proof}
By Lemma \ref{lem_kernel}\ref{item_kernel3}, we may assume that $f\in Z\cap C(K)$ is non-constant on $\ell_B$, then there exist $x_0,x_1\in [0,1]$ with $x_0<x_1$ such that $\alpha:=f(x_0,0)\ne f(x_1,0)=:\beta$, let $\varepsilon:=\frac{1}{3}\lvert \alpha-\beta\rvert>0$ and define $\varphi:\mathbb{R}\to \mathbb{R}$ by $t\mapsto ((\frac{\lvert t-\alpha\rvert}{\varepsilon}-1)\vee0)\wedge1$; see Figure \ref{fig_phi_cutoff}. Let $h:=\varphi\circ f$, then by Lemma \ref{lem_kernel}\ref{item_kernel4}, $h\in Z\cap C(K)$. By continuity, there exist $n\ge0$ sufficiently large and $u,v\in W_n$, such that $(x_0,0)\in K_u$, $(x_1,0)\in K_v$, $h=0$ on $K_u$, and $h=1$ on $K_v$.

\begin{figure}[ht]
\centering
\begin{tikzpicture}[x=1.2cm,y=2cm,>=stealth]
  \draw[->] (-3.4,0) -- (3.4,0) node[right] {$t$};

  \draw[dotted] (-2,0) -- (-2,1);
  \draw[dotted] (2,0) -- (2,1);

  \draw[very thick,blue!70!black]
    (-3,1) -- (-2,1) -- (-1,0) -- (1,0)
           -- (2,1) -- (3,1);

  \node[above] at (-2.5,1) {$\varphi$};
  \node[left] at (-3,1) {$1$};

  \foreach \x/\lab in {
    -2/{\alpha-2\varepsilon},
    -1/{\alpha-\varepsilon},
     0/{\alpha},
     1/{\alpha+\varepsilon},
     2/{\alpha+2\varepsilon}
  } {
    \draw (\x,0.035) -- (\x,-0.035);
    \node[below] at (\x,-0.05) {$\lab$};
  }
\end{tikzpicture}
\caption{The function $\varphi$}
\label{fig_phi_cutoff}
\end{figure}
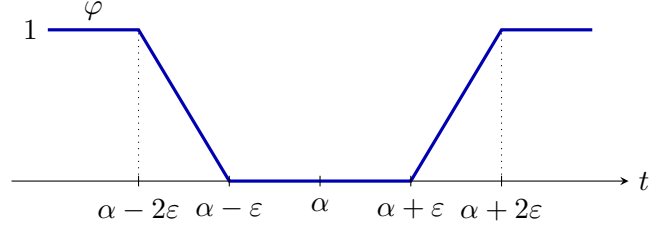

We now use a folding argument to construct the desired function $g$ from the function $h$. For $r,s=0,1,\ldots,3^n-1$, let $F^{(n)}_{r,s}:\mathbb{R}^2\to \mathbb{R}^2$ be given by
$$F^{(n)}_{r,s}(x,y):=\frac{1}{3^n}\left(x+r,y+s\right),\ (x,y)\in \mathbb{R}^2.$$
Then exactly $3^n$ level-$n$ cells intersect $\ell_B$, that is, $F^{(n)}_{r,0}(K)$ for $r=0,1,\ldots,3^{n}-1$, whose union is $K\cap([0,1]\times[0,\frac{1}{3^n}])$, and there exist $a,b=0,1,\ldots,3^{n}-1$ with $a<b$ such that $K_u=F^{(n)}_{a,0}(K)$ and $K_v=F^{(n)}_{b,0}(K)$. First, for any $r=0,1,\ldots,3^n-1$, define
$$g|_{F^{(n)}_{r,0}(K)}:=
\begin{cases}
0&\text{if }r\le a,\\
h|_{F^{(n)}_{r,0}(K)} &\text{if }a<r<b,\\
1&\text{if }r\ge b.
\end{cases}
$$
Second, for any $s=1,2,\ldots,3^{n}-1$ and any $r=0,1,\ldots,3^n-1$ such that $F^{(n)}_{r,s}(K)\subseteq K$, define $g$ on $F^{(n)}_{r,s}(K)$ by
$$g(x,y):=
\begin{cases}
g\left(x,y-\frac{s}{3^n}\right) &\text{if $s$ is even},\\
g\left(x,\frac{s+1}{3^n}-y\right)&\text{if $s$ is odd}.
\end{cases}
$$

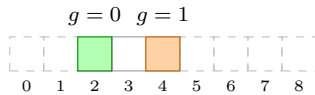
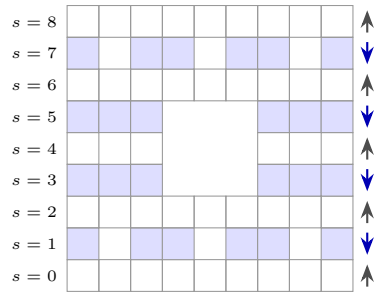
\begin{figure}[htbp]
\centering
\begin{subfigure}[t]{0.40\textwidth}
\centering
\begin{minipage}[c][4.5cm][c]{\linewidth}
\centering
\begin{tikzpicture}[x=0.45cm,y=0.45cm,>=Stealth]
  \foreach \r in {0,...,8}{
    \pgfmathtruncatemacro{\used}{and(\r>=2,\r<=4)}
    \ifnum\used=1
      \draw[gray!80] (\r,0) rectangle ++(1,1);
    \else
      \draw[gray!50,dashed] (\r,0) rectangle ++(1,1);
    \fi
    \node[font=\tiny] at (\r+0.5,-0.45) {$\r$};
  }

  \fill[green!30] (2,0) rectangle ++(1,1);
  \draw[green!50!black] (2,0) rectangle ++(1,1);
  \fill[orange!35] (4,0) rectangle ++(1,1);
  \draw[orange!70!black] (4,0) rectangle ++(1,1);

  \node[font=\scriptsize] at (2.5,1.6) {$g=0$};
  \node[font=\scriptsize] at (4.5,1.6) {$g=1$};
\end{tikzpicture}
\end{minipage}
\caption{Bottom-row cells}
\label{fig_bottom}
\end{subfigure}
\hfill
\begin{subfigure}[t]{0.56\textwidth}
\centering
\begin{minipage}[c][4.5cm][c]{\linewidth}
\centering
\begin{tikzpicture}[x=0.42cm,y=0.42cm,>=Stealth]
  \foreach \r in {0,...,8}{
    \foreach \s in {0,...,8}{
      \pgfmathtruncatemacro{\hole}{
        or(and(mod(\r,3)==1,mod(\s,3)==1),
           and(int(\r/3)==1,int(\s/3)==1))
      }
      \pgfmathtruncatemacro{\odd}{mod(\s,2)}
      \ifnum\hole=0
        \ifnum\odd=1
          \fill[blue!13] (\r,\s) rectangle ++(1,1);
        \fi
        \draw[gray!80] (\r,\s) rectangle ++(1,1);
      \fi
    }
  }

  \foreach \s in {0,...,8}{
    \node[left,font=\tiny] at (0,\s+0.5) {$s=\s$};
    \pgfmathtruncatemacro{\odd}{mod(\s,2)}
    \ifnum\odd=0
      \draw[->,thick,black!70]
        (9.45,\s+0.15)--(9.45,\s+0.85);
    \else
      \draw[->,thick,blue!70!black]
        (9.45,\s+0.85)--(9.45,\s+0.15);
    \fi
  }
\end{tikzpicture}
\end{minipage}
\caption{Vertical folding}
\label{fig_vertical}
\end{subfigure}
\caption{The level-$2$ cell configuration}
\label{fig_folding}
\end{figure}

Figure \ref{fig_folding} illustrates the folding construction for $n=2$, $a=2$, and $b=4$. Consider a level-$n$ cell $F^{(n)}_{r,s}(K)$. First, on the bottom row $s=0$, $g=0$ on cells with $r\le a$, $g=h$ on cells with $a<r<b$, and $g=1$ on cells with $r\ge b$. In particular, $g=0$ on $F^{(n)}_{a,0}(K)$ and $g=1$ on $F^{(n)}_{b,0}(K)$, see Figure \ref{fig_bottom}. These values match across adjacent bottom-row cells. Then, on any cell in a row $s\ge1$, $g$ is obtained from its bottom-row counterpart by preserving the vertical orientation when $s$ is even and reversing it when $s$ is odd, see Figure \ref{fig_vertical}, where the arrows at the right indicate these orientations. This alternating construction makes the values match across adjacent rows.

It is easy to verify that $g\in C(K)$, and on each level-$n$ cell, $g$ is either identically $0$, identically $1$, or equal to $h$ on a level-$n$ cell, possibly after a suitable reflection, hence Lemma \ref{lem_kernel}\ref{item_kernel2}, \ref{item_kernel3} and \ref{item_kernel5} yield $g\in Z\cap C(K)$. Moreover, by construction, $g=0$ on $\cup_{s=0}^{3^n-1}F^{(n)}_{0,s}(K)\supseteq \ell_L$ and $g=1$ on $\cup_{s=0}^{3^n-1}F^{(n)}_{3^{n}-1,s}(K)\supseteq \ell_R$.
\end{proof}

\begin{lemma}\label{lem_interpolation}
Let $Q$ be a self-similar component with kernel $Z=Z(Q)$. If there exists $f\in Z\cap C(K)$ with $f=0$ on $\ell_L$ and $f=1$ on $\ell_R$, then $Q=0$, that is, $Z=\mathcal{F}$.
\end{lemma}

\begin{proof}
By Lemma \ref{lem_kernel}\ref{item_kernel4}, replacing $f$ by $(f\vee0)\wedge1$ allows us to assume that $0\le f\le 1$ on $K$. By Lemma \ref{lem_kernel}\ref{item_kernel3}, the function $((x,y)\mapsto f(x,1-y))\in Z\cap C(K)$. By replacing $f$ by $(x,y)\mapsto \frac{1}{2}(f(x,y)+f(x,1-y))$ and Lemma \ref{lem_kernel}\ref{item_kernel1}, we may assume that $f(x,0)=f(x,1)=:\theta(x)$ for any $x\in[0,1]$, where $\theta\in C([0,1])$ satisfies $0\le \theta\le 1$ on $[0,1]$. Moreover, the function $g=((x,y)\mapsto f(y,x))\in Z\cap C(K)$ satisfies $g=0$ on $\ell_B$, $g=1$ on $\ell_T$, and $g(0,y)=g(1,y)=\theta(y)$ for any $y\in[0,1]$. Let
$$\psi_{00}:=(1-f)(1-g),\qquad\psi_{01}:=(1-f)g,\qquad\psi_{10}:=f(1-g),\qquad \psi_{11}:=fg,$$
then $\psi_{ij}\in Z\cap C(K)$, $\psi_{ij}(k,l)=\delta_{ik}\delta_{jl}$ for any $i,j,k,l=0,1$, and $\sum_{i,j=0,1}\psi_{ij}=1$ on $K$.

For any $a=(a_{00},a_{10},a_{01},a_{11})\in\mathbb{R}^4$ let
$$I_a:=\sum_{i,j=0,1}a_{ij}\psi_{ij}\in Z\cap C(K).$$
It is easy to verify that $I_a(i,j)=a_{ij}$ for any $i,j=0,1$, and for any $(x,y)\in K$,
\begin{align*}
&I_a(0,y)=(1-\theta(y))a_{00}+\theta(y)a_{01},\qquad I_a(1,y)=(1-\theta(y))a_{10}+\theta(y)a_{11},\\
&I_a(x,0)=(1-\theta(x))a_{00}+\theta(x)a_{10},\qquad I_a(x,1)=(1-\theta(x))a_{01}+\theta(x)a_{11}.
\end{align*}
Moreover, since $\sum_{i,j=0,1}\psi_{ij}=1$ on $K$,
\begin{align}
&\mathcal{E}(I_a)^{\frac{1}{p}}=\mathcal{E}\left(\sum_{i,j=0,1}\left(a_{ij}-a_{00}\right)\psi_{ij}\right)^{\frac{1}{p}}\le \sum_{i,j=0,1}\lvert a_{ij}-a_{00}\rvert \mathcal{E}(\psi_{ij})^{\frac{1}{p}}\nonumber\\
&\le 2 \left(\lvert a_{00}-a_{01}\rvert+\lvert a_{01}-a_{11}\rvert+\lvert a_{10}-a_{11}\rvert+\lvert a_{00}-a_{10}\rvert\right)\max_{i,j=0,1}\mathcal{E}(\psi_{ij})^{\frac{1}{p}}.\label{eq_EIa}
\end{align}

For any $n\ge0$, let $V_n:=\cup_{w\in W_n}\{F_w(i,j):i,j=0,1\}$ be the set of corners of level-$n$ cells. For any $b:V_n\to \mathbb{R}$, let $a:W_n\to \mathbb{R}^4$ be given by
$$a(w):=\left(b(F_w(0,0)),b(F_w(1,0)),b(F_w(0,1)),b(F_w(1,1))\right),\ w\in W_n,$$
that is, $a(w)$ lists the values of $b$ at the corners of $K_w$. Define $J_n(b)$ on $K$ by specifying, for any $w\in W_n$,
\begin{equation}\label{eq_Jnb}
J_n(b)\circ F_w=I_{a(w)}.
\end{equation}

We claim that $J_n(b)$ is well-defined, and $J_n(b)\in Z\cap C(K)$. It suffices to show that for any distinct $u,v\in W_n$ with $K_u\cap K_v\ne\emptyset$, the above definitions agree on $K_u\cap K_v$. Indeed, for any such $u,v$, $K_u\cap K_v$ is either a one-point set or a line segment. If $K_u\cap K_v=\{z\}\subseteq V_n$, then $I_{a(u)}$ and $I_{a(v)}$ both give the value $b(z)$ on $z$, indeed, $I_{a(u)}(F_u^{-1}(z))=I_{a(v)}(F_v^{-1}(z))=b(z)$. If $K_u\cap K_v$ is a line segment, without loss of generality, we may assume that it is a vertical line segment with the same parametrization $t\mapsto F_u(1,t)=F_v(0,t)$, $t\in[0,1]$, see Figure \ref{fig_inter_uv}, where $a_{ij}=b(F_u(i,j))$ and $a'_{ij}=b(F_v(i,j))$ are the values of $b$ at the corners of cells $K_u$ and $K_v$, notably $a_{10}=a'_{00}$ and $a_{11}=a'_{01}$. Then on the line segment, $I_{a(u)}$ gives at $F_u(1,t)$ the value $I_{a(u)}(1,t)=(1-\theta(t))a_{10}+\theta(t)a_{11}$, while $I_{a(v)}$ gives at $F_v(0,t)$ the value $I_{a(v)}(0,t)=(1-\theta(t))a'_{00}+\theta(t)a'_{01}$. Hence $J_n(b)$ is well-defined and $J_n(b)\in C(K)$. Moreover, Lemma \ref{lem_kernel}\ref{item_kernel5} gives $J_n(b)\in Z$.

\begin{figure}[ht]
\centering
\begin{tikzpicture}[scale=2.6,>=Stealth]
\draw[thick] (0,0) rectangle (1,1);
\draw[thick] (1,0) rectangle (2,1);
\foreach \p/\l/\a in {(0,0)/{a_{00}}/below left,(1,0)/{a_{10}=a'_{00}}/below,(0,1)/{a_{01}}/above left,(1,1)/{a_{11}=a'_{01}}/above,(2,0)/{a'_{10}}/below right,(2,1)/{a'_{11}}/above right}{\fill \p circle (0.7pt); \node[\a,font=\footnotesize] at \p {$\l$};}
\node at (0.5,0.5) {$K_u$};
\node at (1.5,0.5) {$K_v$};
\end{tikzpicture}
\caption{Two level-$n$ cells $K_u$ and $K_v$}
\label{fig_inter_uv}
\end{figure}
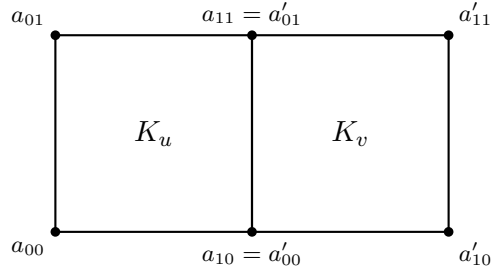

For any $f\in \mathcal{F}$ and any $n\ge0$, let $b:V_n\to \mathbb{R}$ be given by
$$b(z):=\frac{1}{\# \{w\in W_n:z\in K_w\}}\sum_{\substack{w\in W_n\\z\in K_w}}f_{K_w},\ z\in V_n.$$
Let $f_n:=J_n(b)\in Z\cap C(K)$ be given as above, then by (\ref{eq_Jnb}) and (\ref{eq_EIa}),
\begin{align*}
&\mathcal{E}(f_n)=\rho^{n}\sum_{w\in W_n}\mathcal{E}(J_n(b)\circ F_w)\\
&\lesssim \rho^n \sum_{w\in W_n}\left(\lvert b(F_w(0,0))-b(F_w(0,1))\rvert^p+\lvert b(F_w(0,1))-b(F_w(1,1))\rvert^p\right.\\
&\hspace{60pt}\left.+\lvert b(F_w(1,0))-b(F_w(1,1))\rvert^p+\lvert b(F_w(0,0))-b(F_w(1,0))\rvert^p\right)\\
&\le 2^{p-1}\rho^n \sum_{w\in W_n} \left(\lvert b(F_w(0,0))-f_{K_w}\rvert^p +\lvert b(F_w(0,1))-f_{K_w}\rvert^p+\ldots\right)\\
&\lesssim \rho^n\sum_{w\in W_n}\sum_{i,j=0,1}\lvert b(F_w(i,j))-f_{K_w}\rvert^p.
\end{align*}
For any $i,j=0,1$,
\begin{align}
&\lvert b(F_w(i,j))-f_{K_w}\rvert\nonumber\\
&\le \frac{1}{\# \{v\in W_n:F_{w}(i,j)\in K_v\}}\sum_{\substack{v\in W_n\\F_{w}(i,j)\in K_v}}\lvert f_{K_v}-f_{K_w} \rvert\le\sum_{\substack{v\in W_n\\ K_v\cap K_w\ne\emptyset}}\lvert f_{K_v}-f_{K_w} \rvert.\label{eq_bFwij}
\end{align}
Since $\sup_{w\in W_n}\#\{v\in W_n:K_v\cap K_w\ne\emptyset\}\le8$,
\begin{align*}
\mathcal{E}(f_n)\lesssim \rho^n\sum_{w\in W_n}\sum_{\substack{v\in W_n\\ K_v\cap K_w\ne\emptyset}}\lvert f_{K_v}-f_{K_w} \rvert^p\lesssim\rho^n\sum_{\substack{v,w\in W_n\\ K_v\cap K_w\ne\emptyset}}\lvert f_{K_v}-f_{K_w}\rvert^p\lesssim \mathcal{E}(f),
\end{align*}
where the last inequality follows from \cite[Theorem 1.1(iii)]{MS25}. Moreover,
\begin{align*}
\int_K \lvert f_n-f\rvert^p \mathrm{d}m\le 2^{p-1}\sum_{w\in W_n}\int_{K_w}\lvert f_n-f_{K_w}\rvert^p \mathrm{d}m+2^{p-1}\sum_{w\in W_n}\int_{K_w}\lvert f-f_{K_w}\rvert^p \mathrm{d}m,
\end{align*}
where by (\ref{eq_P9a}) and the fact $\rho\ge \frac{1}{2}$ (see Lemma \ref{lem_rhochi}),
$$\sum_{w\in W_n}\int_{K_w}\lvert f-f_{K_w}\rvert^p \mathrm{d}m\lesssim \frac{1}{8^n}\sum_{w\in W_n}\mathcal{E}(f\circ F_w)=\frac{1}{(8\rho)^n}\mathcal{E}(f)\to0.$$
For any $w\in W_n$, since $\sum_{i,j=0,1}\psi_{ij}=1$ on $K$, by (\ref{eq_bFwij}),
\begin{align*}
&\int_{K_w}\lvert f_n-f_{K_w}\rvert^p \mathrm{d}m=\frac{1}{8^n}\int_{K}\lvert \sum_{i,j=0,1}\left(b(F_w(i,j))-f_{K_w}\right)\psi_{ij}\rvert^p \mathrm{d}m\\
&\lesssim \frac{1}{8^n}\sum_{i,j=0,1}\lvert b(F_w(i,j))-f_{K_w}\rvert^p\lesssim \frac{1}{8^n}\sum_{\substack{v\in W_n\\ K_v\cap K_w\ne\emptyset}}\lvert f_{K_v}-f_{K_w} \rvert^p.
\end{align*}
Hence by \cite[Theorem 1.1(iii)]{MS25} and the fact $\rho\ge \frac{1}{2}$ again,
\begin{align*}
\sum_{w\in W_n}\int_{K_w}\lvert f_n-f_{K_w}\rvert^p \mathrm{d}m\lesssim \frac{1}{8^n}\sum_{\substack{v,w\in W_n\\ K_v\cap K_w\ne\emptyset}}\lvert f_{K_v}-f_{K_w} \rvert^p\lesssim \frac{1}{(8\rho)^n}\mathcal{E}(f)\to0.
\end{align*}
Therefore, $\{f_n\}_{n\ge0}\subseteq Z\cap C(K)$ is $L^p(K;m)$-convergent to $f$ and is $\mathcal{E}$-bounded. By Lemma \ref{lem_BAM}, for any $n\ge0$, there exist an integer $I_n\ge n$ and coefficients $\lambda^{(n)}_k\ge0$, $k=n,\ldots,I_n$, such that $\sum_{k=n}^{I_n}\lambda^{(n)}_k=1$, and $\{\sum_{k=n}^{I_n}\lambda^{(n)}_kf_k\}_{n\ge0}$ is $\mathcal{E}_1$-convergent to $f$. By Lemma \ref{lem_kernel}\ref{item_kernel1}, $\{\sum_{k=n}^{I_n}\lambda^{(n)}_kf_k\}_{n\ge0}\subseteq Z\cap C(K)$ and $f\in Z$. Since $f\in \mathcal{F}$ is arbitrary, $\mathcal{F}=Z$, that is, $Q=0$.
\end{proof}

We have the following dichotomy.

\begin{theorem}\label{thm_dichotomy}
Let $Q$ be a self-similar component with kernel $Z=Z(Q)$. Then either $Z\cap C(K)$ consists only of constant functions or $Q=0$ (equivalently, $Z=\mathcal{F}$).
\end{theorem}

\begin{proof}
Assume that there exists a non-constant function $f\in Z\cap C(K)$. First, we claim that there exists $w\in W_*$ such that $f\circ F_w$ is non-constant on one of the four sides $\ell_L,\ell_R,\ell_B,\ell_T$. Indeed, otherwise, by continuity, for any $w\in W_*$, $f$ is constant on $K^\partial_w$, hence for any $n\ge0$, $f$ is constant on $K^\partial_n$, then $f$ is constant on $K^\partial_*$. Since $K^\partial_*$ is dense in $K$, $f$ is constant on $K$, a contradiction. By replacing $f$ by $f\circ F_w$ and Lemma \ref{lem_kernel}, we may assume that $f\in Z\cap C(K)$ is non-constant on one of the four sides $\ell_L,\ell_R,\ell_B,\ell_T$. By Lemma \ref{lem_corridor}, there exists $g\in Z\cap C(K)$ such that $g=0$ on $\ell_L$ and $g=1$ on $\ell_R$. Then Lemma \ref{lem_interpolation} gives $Z=\mathcal{F}$.
\end{proof}

For any $A\in \mathcal{B}(K)$, let $Q_A:\mathcal{F}\to[0,+\infty)$ be given by
$$Q_A(f):=\Gamma(f)(A),\ f\in \mathcal{F}.$$

\begin{lemma}\label{lem_invcom}
Let $A\in \mathcal{B}(K)$ satisfy (\ref{eq_ss01}). Then the following hold:
\begin{enumerate}[label=(\alph*),ref=(\alph*)]
\item\label{item_invcom1} For any $f\in\mathcal{F}$, any $w\in W_*$, and any $B\in \mathcal{B}(K)$,
$$\Gamma(f)(A\cap F_w(B))=\rho^{\lvert w\rvert}\Gamma(f\circ F_w)(A\cap B),$$
in particular,
$$\Gamma(f)(A\cap K_w)=\rho^{\lvert w\rvert}Q_A(f\circ F_w).$$
\item\label{item_invcom2} The functional $Q_A:\mathcal{F}\to[0,+\infty)$ is a self-similar component of $\mathcal{E}$.
\item\label{item_invcom3} The set $K\backslash A$ also satisfies (\ref{eq_ss01}), and $Q_A+Q_{K\backslash A}=\mathcal{E}$.
\end{enumerate}
\end{lemma}

\begin{proof}
\ref{item_invcom1} By (\ref{eq_ss01}), $1_{A}\circ F_w=1_A$ $\Gamma(f\circ F_w)$-a.e. Let $n=\lvert w\rvert$, then by (\ref{eq_Gamma_ss}) and (\ref{eq_bdy_zero}),
\begin{align*}
&\Gamma(f)(A\cap F_w(B))=\rho^n\sum_{v\in W_n}\Gamma(f\circ F_v)(F_v^{-1}(A\cap F_w(B)))\\
&=\rho^{\lvert w\rvert}\Gamma(f\circ F_w)(F_w^{-1}(A)\cap B)=\rho^{\lvert w\rvert}\Gamma(f\circ F_w)(A\cap B).
\end{align*}
In particular, taking $B=K$,
$$\Gamma(f)(A\cap K_w)=\rho^{\lvert w\rvert}\Gamma(f\circ F_w)(A)=\rho^{\lvert w\rvert}Q_A(f\circ F_w).$$

\ref{item_invcom2} \ref{item_Q1}, \ref{item_Q2}, and \ref{item_Q5} follow directly from the definition and the properties of energy measures. For \ref{item_Q3}, by (\ref{eq_Gamma_ss}) and (\ref{eq_ss01}), for any $f\in \mathcal{F}$,
$$Q_A(f)=\Gamma(f)(A)=\rho\sum_{i=0}^7\Gamma(f\circ F_i)(F_i^{-1}(A))=\rho\sum_{i=0}^7\Gamma(f\circ F_i)(A)=\rho\sum_{i=0}^7Q_A(f\circ F_i).$$
For \ref{item_Q4}, by (\ref{eq_ss01}) and (\ref{eq_Gamma_D4}), for any $f\in \mathcal{F}$,
$$Q_A(f)=\Gamma(f)(A)=\Gamma(f)(\psi(A))=\Gamma(f\circ \psi)(A)=Q_A(f\circ \psi).$$

\ref{item_invcom3} It follows directly by noting that for any $i=0,1,\ldots,7$ and any $\psi\in D_4$,
$$1_{K\backslash A}=1-1_A=1\circ F_i-1_A\circ F_i=(1-1_A)\circ F_i=1_{K\backslash A}\circ F_i\ \mu\text{-a.e.},$$
$$1_{K\backslash A}\circ \psi=\left(1-1_{A}\right)\circ \psi=1\circ \psi-1_A\circ \psi=1-1_A=1_{K\backslash A}\ \mu\text{-a.e.}$$
\end{proof}

\begin{lemma}\label{lem_equivalent}
Let $A\in \mathcal{B}(K)$ satisfy (\ref{eq_ss01}). The following are equivalent.
\begin{enumerate}[label=(\roman*),ref=(\roman*)]
\item\label{item_equiv_Q1} $Q_{K\backslash A}=0$.
\item\label{item_equiv_Q2} $Q_A=\mathcal{E}$.
\item\label{item_equiv_Q3} $\mu(K\backslash A)=0$.
\item\label{item_equiv_Q4} There exists a constant $c>0$ such that $Q_A\ge c \mathcal{E}$.
\end{enumerate}
\end{lemma}

\begin{proof}
The equivalences \ref{item_equiv_Q1} $\Longleftrightarrow$ \ref{item_equiv_Q2} and \ref{item_equiv_Q1} $\Longleftrightarrow$ \ref{item_equiv_Q3} follow directly from the definitions, (\ref{eq_ss01}), and Lemma \ref{lem_invcom}. \ref{item_equiv_Q2} $\Longrightarrow$ \ref{item_equiv_Q4}: Trivial. \ref{item_equiv_Q4} $\Longrightarrow$ \ref{item_equiv_Q1}: For any $f\in \mathcal{F}$, we show that $Q_{K\backslash A}(f)=\Gamma(f)(K\backslash A)=0$. Indeed, by assumption, for any $w\in W_*$, by Lemma \ref{lem_invcom}\ref{item_invcom1},
$$Q_A(f\circ F_w)=\frac{1}{\rho^{\lvert w\rvert}}\Gamma(f)(A\cap K_w)\ge c \mathcal{E}(f\circ F_w),$$
hence $\Gamma(f)(A\cap K_w)\ge c\rho^{\lvert w\rvert}\mathcal{E}(f\circ F_w)=c\Gamma(f)(K_w)$ for any $w\in W_*$. By (\ref{eq_bdy_zero}) and the monotone class theorem, $\Gamma(f)(A\cap B)\ge c\Gamma(f)(B)$ for any $B\in \mathcal{B}(K)$; in particular, by taking $B=K\backslash A$, $\Gamma(f)(K\backslash A)=0$.
\end{proof}

We give the proof of Theorem \ref{thm_01law} as follows.

\begin{proof}[Proof of Theorem \ref{thm_01law}]
Assume that $\mu(A)>0$. By Lemma \ref{lem_equivalent}, to show that $\mu(K\backslash A)=0$, we only need to show that $Q_{K\backslash A}=0$. Then by Theorem \ref{thm_dichotomy}, it suffices to show that there exists a non-constant function in $Z(Q_{K\backslash A})\cap C(K)$. We construct such a function using a blowup argument similar to that used in the proof of Theorem \ref{thm_step1}.

Since $\mu(A)>0$, there exists $f\in \mathcal{F}$ such that $\Gamma(f)(A)>0$. By Lemma \ref{lem_Hstar}, there exist $k\ge0$ and $h\in \mathcal{H}_k$ such that $\Gamma(h)(A)>0$ by noting that $\lvert \Gamma(f)(A)^{\frac{1}{p}}-\Gamma(h)(A)^{\frac{1}{p}}\rvert\le \Gamma(f-h)(A)^{\frac{1}{p}}\le \mathcal{E}(f-h)^{\frac{1}{p}}$. Then there exists $v\in W_k$ such that $\Gamma(h)(A\cap K_v)>0$. Replacing $h$ by $h\circ F_v$, we may assume that $h\in \mathcal{H}$ and $\Gamma(h)(A)>0$. By normalization, we may assume that $\Gamma(h)(K)=1$, then $\Gamma(h)$ is a probability measure on $K$, let $\mathbb{E}$ denote the corresponding expectation. Let $\delta:=\Gamma(h)(A)>0$. For any $n\ge0$, let
$$A_n:=\{w\in W_n:K_{\mathcal{N}_3(w)}\cap K^\partial\ne\emptyset\},$$
then $K_{A_n}\downarrow K^\partial$. By (\ref{eq_bdy_zero}), $\Gamma(h)(K^\partial)=0$, then there exists $n_0\ge0$ such that for any $n\ge n_0$, $\Gamma(h)(K_{A_n})<\frac{\delta}{3}$. For any $n\ge0$, let
$$G_n:=\left\{w\in W_n:\Gamma(h)(K_w)<\frac{\delta}{147}\Gamma(h)\left(K_{\mathcal{N}_3(w)}\right)\right\}.$$
By (\ref{eq_bdy_zero}) and (\ref{eq_sup_nbhd}),
\begin{align*}
&\Gamma(h)\left(K_{G_n}\right)=\sum_{w\in G_n}\Gamma(h)(K_w)\le \frac{\delta}{147}\sum_{w\in G_n}\Gamma(h)\left(K_{\mathcal{N}_3(w)}\right)=\frac{\delta}{147}\sum_{w\in G_n}\sum_{v\in \mathcal{N}_3(w)}\Gamma(h) \left(K_v\right)\\
&\le \frac{\delta}{147}\sum_{v\in W_n} \Gamma(h)(K_v)\cdot \#\{w\in W_n:v\in \mathcal{N}_3(w)\}\le \frac{\delta}{147}\cdot (2\cdot3+1)^2\sum_{v\in W_n}\Gamma(h)(K_v)=\frac{\delta}{3}.
\end{align*}
Let $K_\infty:=\varliminf_{n\to+\infty}K_{G_n}=\cup_{N\ge0}\cap_{n\ge N}K_{G_n}$, then Fatou's lemma gives
$$\Gamma(h)(K_\infty)\le \varliminf_{n\to+\infty}\Gamma(h)(K_{G_n})\le \frac{\delta}{3}.$$
Let $B:=A\backslash (K_{A_{n_0}}\cup K_\infty\cup K^\partial_*)$, then $\Gamma(h)(B)\ge \delta-\frac{\delta}{3}-\frac{\delta}{3}=\frac{\delta}{3}$.

For any $n\ge0$, let $\mathcal{G}_n$ be the $\sigma$-algebra generated by $\{K_w:w\in W_n\}$, then $\{\mathcal{G}_n\}_{n\ge0}$ is a filtration, and $\mathcal{B}(K)=\mathcal{G}_\infty:=\sigma (\cup_{n\ge0}\mathcal{G}_n)$. For any $n\ge0$, let
$$Y_n:=\mathbb{E}[1_{B}|\mathcal{G}_n],$$
then $\{Y_n\}_{n\ge0}$ is a non-negative uniformly integrable martingale. By the martingale convergence theorem,
\begin{equation}
Y_n\to \mathbb{E}[1_{B}|\mathcal{G}_\infty]=1_{B}\ \Gamma(h)\text{-a.e. on $K$ and in }L^1(K;\Gamma(h)).
\end{equation}
Recall that for any $x\in K\backslash K^\partial_*$ and any $n\ge0$, $[x]_n$ is the unique word in $W_n$ such that $x\in K_{[x]_n}$. By (\ref{eq_bdy_zero}), $\Gamma(h)(K^\partial_*)=0$, then
\begin{equation*}
Y_n=\sum_{\substack{w\in W_n\\ \Gamma(h)(K_w)>0}}\frac{\Gamma(h)(B\cap K_w)}{\Gamma(h)(K_w)}1_{K_w}\ \Gamma(h)\text{-a.e. on }K.
\end{equation*}
In particular, for $\Gamma(h)$-a.e. $x\in K\backslash K^\partial_*$,
$$
Y_n(x)=
\begin{cases}
\frac{\Gamma(h)(B\cap K_{[x]_n})}{\Gamma(h)(K_{[x]_n})}&\text{if }\Gamma(h)(K_{[x]_n})>0,\\
0&\text{if }\Gamma(h)(K_{[x]_n})=0.
\end{cases}
$$
Since $\Gamma(h)(B)>0$, there exists $x\in B$ such that $\lim_{n\to+\infty}Y_n(x)=1_B(x)=1$, which implies $\Gamma(h)(B\cap K_{[x]_n})>0$. Since $x\in K\backslash K_\infty=\cap_{N\ge0}\cup_{n\ge N}(K\backslash K_{G_n})$, there exists a subsequence $\{n_k\}_{k\ge1}\subseteq\{n_0,n_0+1,\ldots\}$ such that $\Gamma(h)(K_{[x]_{n_k}})\ge \frac{\delta}{147}\Gamma(h)(K_{\mathcal{N}_3([x]_{n_k})})$ for any $k\ge1$; in particular, $\Gamma(h)(B\cap K_{[x]_{n_k}})>0$ and $\lim_{k\to+\infty}Y_{n_k}(x)=1$. Let
$$g_k:=\left(\frac{\rho^{n_k}}{\Gamma(h)(K_{[x]_{n_k}})}\right)^{\frac{1}{p}}\left(h-h_{K_{[x]_{n_k}}}\right),$$
then $(g_k)_{K_{[x]_{n_k}}}=0$,
$$\Gamma(g_k)(K_{[x]_{n_k}})=\frac{\rho^{n_k}}{\Gamma(h)(K_{[x]_{n_k}})}\Gamma(h)(K_{[x]_{n_k}})=\rho^{n_k},$$
and
$$\Gamma(g_k)(K_{\mathcal{N}_3([x]_{n_k})})=\frac{\rho^{n_k}}{\Gamma(h)(K_{[x]_{n_k}})}\Gamma(h)(K_{\mathcal{N}_3([x]_{n_k})})\le \frac{147}{\delta}\rho^{n_k}.$$
Since $x\not\in K_{A_{n_0}}$, Proposition \ref{prop_Hin418} applies to $\{g_k\}_k$ and $\{[x]_{n_k}\}_{k\ge1}$, then there exist a subsequence, still denoted by $\{g_k\circ F_{[x]_{n_k}}\}_{k\ge1}$, and a function $g\in \mathcal{H}\cap C(K)$, such that $\{g_k\circ F_{[x]_{n_k}}\}_{k\ge1}$ is $\mathcal{E}_1$-convergent and uniformly convergent to $g$ on $K$. Since
$$\mathcal{E}(g_k\circ F_{[x]_{n_k}})=\frac{1}{\rho^{n_k}}\Gamma(g_k)(K_{[x]_{n_k}})=1,$$
we have $\mathcal{E}(g)=\lim_{k\to+\infty}\mathcal{E}(g_k\circ F_{[x]_{n_k}})=1$, which implies $g$ is non-constant. By Lemma \ref{lem_invcom}\ref{item_invcom1} and \ref{item_invcom3},
\begin{align*}
&Q_{K\backslash A}(g_k\circ F_{[x]_{n_k}})=\frac{1}{\rho^{n_k}}\Gamma(g_k)(K_{[x]_{n_k}}\backslash A)=\frac{\Gamma(h)(K_{[x]_{n_k}}\backslash A)}{\Gamma(h)(K_{[x]_{n_k}})}\le 1-Y_{n_k}(x)\to0,
\end{align*}
hence $Q_{K\backslash A}(g)=\lim_{k\to+\infty}Q_{K\backslash A}(g_k\circ F_{[x]_{n_k}})=0$, that is, $g\in Z(Q_{K\backslash A})$. In summary, $g\in Z(Q_{K\backslash A})\cap C(K)$ is non-constant. This completes the proof.
\end{proof}

\section{Proof of Theorem \ref{thm_unique}}\label{sec_unique}

Let $(\mathcal{E}^{(c)},\mathcal{F}^{(c)})\in \mathfrak{E}$ be the canonical $p$-energy, and let $\Gamma^{(c)}$, $\rho^{(c)}$, and $\beta^{(c)}_p$ denote the associated energy measure and parameters. For any $(\mathcal{E},\mathcal{F})\in \mathfrak{E}$, let $\Gamma$, $\rho$, and $\beta_p$ denote the corresponding energy measure and parameters.

We have the following quasi-uniqueness result.

\begin{lemma}[{\cite[Propositions 10.26 and 10.29]{MS23arXiv}}]\label{lem_quasi_unique}
We have $\rho=\rho^{(c)}$, $\mathcal{F}=\mathcal{F}^{(c)}$, and there exists a constant $C\ge1$ such that
$$\frac{1}{C}\mathcal{E}^{(c)}(f)\le \mathcal{E}(f)\le C \mathcal{E}^{(c)}(f)\text{ for any }f\in \mathcal{F}=\mathcal{F}^{(c)}.$$
Hence for any $f\in \mathcal{F}=\mathcal{F}^{(c)}$ and any $A\in \mathcal{B}(K)$,
\begin{equation}\label{eq_gamma_gammac}
\frac{1}{C}\Gamma^{(c)}(f)(A)\le \Gamma(f)(A)\le C\Gamma^{(c)}(f)(A).
\end{equation}
\end{lemma}

For simplicity, we use $\mathcal{F}$ to denote both $\mathcal{F}$ and $\mathcal{F}^{(c)}$.

We divide the proof of uniqueness into two steps:

\begin{enumerate}[label=Step \arabic*:,ref=Step \arabic*]
\item\label{main_step1} We use Theorem \ref{thm_R1} to show that there exists a Borel function $a$ on $K$ such that
$$\Gamma(f)=a \Gamma^{(c)}(f)\text{ for any }f\in \mathcal{F}=\mathcal{F}^{(c)}.$$
\item\label{main_step2} We use Theorem \ref{thm_01law} to show that $a$ is indeed a constant function.
\end{enumerate}

As a consequence of Lemma \ref{lem_quasi_unique}, we have several properties of $(\mathcal{E},\mathcal{F})$ that follow from the corresponding properties of $(\mathcal{E}^{(c)},\mathcal{F}^{(c)})$.

\begin{lemma}\label{lem_meas_PI}
\begin{enumerate}[label=(\arabic*)]
\item\label{item_meas5} For any $f\in \mathcal{F}$ and any $w\in W_*$, $\Gamma(f)(K^\partial_*)=0$, hence
$$\Gamma(f)(K_w)=\rho^{\lvert w\rvert}\mathcal{E}(f\circ F_w).$$

\item\label{item_meas6} For any $f\in \mathcal{F}$, $w\in W_*$, and $A\in \mathcal{B}(K)$, $\Gamma(f)(F_w(A))=\rho^{\lvert w\rvert}\Gamma(f\circ F_w)(A)$, hence for any bounded Borel function $\phi$,
\begin{equation*}
\int_{K_w}\phi \mathrm{d}\Gamma(f)=\rho^{\lvert w\rvert}\int_K (\phi\circ F_w) \mathrm{d}\Gamma(f\circ F_w).
\end{equation*}
\end{enumerate}
\end{lemma}

\begin{proof}
\ref{item_meas5} For any $f\in \mathcal{F}$, by (\ref{eq_bdy_zero}) and (\ref{eq_gamma_gammac}), $\Gamma(f)(K_*^\partial)=0$. Then for any $n\ge0$ and any $w\in W_n$, by (\ref{eq_Gamma_ss}),
$$\Gamma(f)(K_w)=\rho^{n}\sum_{v\in W_n}\Gamma(f\circ F_v)(F_v^{-1}(K_w\cap K_v))=\rho^n \Gamma(f\circ F_w)(K)=\rho^n \mathcal{E}(f\circ F_w),$$
where for any $v\ne w$, $F_v^{-1}(K_w\cap K_v)\subseteq K^\partial\subseteq K^\partial_*$, hence $\Gamma(f\circ F_v)(F_v^{-1}(K_w\cap K_v))=0$.

\ref{item_meas6} For any $f\in \mathcal{F}$, $n\ge0$, $w\in W_n$, and $A\in \mathcal{B}(K)$, by (\ref{eq_Gamma_ss}) and \ref{item_meas5},
$$\Gamma(f)(F_w(A))=\rho^n\sum_{v\in W_n}\Gamma(f\circ F_v)(F_v^{-1}(F_w(A)))=\rho^n\Gamma(f\circ F_w)(A).$$
Then for any $B\in \mathcal{B}(K_w)$, let $A=F_w^{-1}(B)\in \mathcal{B}(K)$,
$$\int_{K_w}1_B \mathrm{d}\Gamma(f)=\rho^{n}\int_{F_w^{-1}(B)}\mathrm{d}\Gamma(f\circ F_w)=\rho^{n}\int_{K}(1_B\circ F_w) \mathrm{d}\Gamma(f\circ F_w).$$
By linearity, approximation by simple functions, and the dominated convergence theorem, for any bounded Borel function $\phi$,
\begin{align*}
&\int_{K_w}\phi \mathrm{d}\Gamma(f)=\int_{K_w}(\phi|_{K_w}) \mathrm{d}\Gamma(f)\\
&=\rho^{n}\int_K \left((\phi|_{K_w})\circ F_w \right)\mathrm{d}\Gamma(f\circ F_w)=\rho^{n}\int_K (\phi\circ F_w )\mathrm{d}\Gamma(f\circ F_w).
\end{align*}
\end{proof}

Since $(\mathcal{F},\mathcal{E}^{(c)}_1)$ is separable, by Lemma \ref{lem_quasi_unique}, there exists a countable $\mathbb{Q}$-linear space $\mathcal{D}\subseteq \mathcal{F}\cap C(K)$ which is dense in $\mathcal{F}$ with respect to both $\mathcal{E}^{(c)}_1$ and $\mathcal{E}_1$. Write $\mathcal{D}=\{f_k:k\ge1\}$. For any $k\ge1$, let $b_k=\frac{1}{2^k\left(\mathcal{E}^{(c)}(f_k)+1\right)}\in(0,+\infty)$. Let
$$\mu:=\sum_{k\ge1}b_k\Gamma^{(c)}(f_k).$$
By \cite[Lemma 9.20(b)]{MS25}, $\mu$ is a minimal energy-dominant measure of $(\mathcal{E}^{(c)},\mathcal{F})$, then for any $f\in \mathcal{F}$, $\Gamma^{(c)}(f)\ll \mu$ and $\Gamma(f)\ll \mu$ by (\ref{eq_gamma_gammac}).

\begin{lemma}\label{lem_martingale}
There exists a Borel $\mu$-null set $N_0\supseteq K^\partial_*$ such that $R_x(f)$ and $T_x(f)$ are well-defined and non-negative for all $x\in K\backslash N_0$ and $f\in \mathcal{D}$. The following conditions hold. For any $f\in \mathcal{D}$, $x\mapsto R_x(f)$ and $x\mapsto T_x(f)$ are Borel measurable, and there exist $\mu$-versions of $\frac{\mathrm{d}\Gamma^{(c)}(f)}{\mathrm{d}\mu}$ and $\frac{\mathrm{d}\Gamma(f)}{\mathrm{d}\mu}$, denoted by the same symbols, such that for any $x\in K\backslash N_0$,
$$R_x(f)^p=\frac{\mathrm{d}\Gamma^{(c)}(f)}{\mathrm{d}\mu}(x),\qquad T_x(f)^p=\frac{\mathrm{d}\Gamma(f)}{\mathrm{d}\mu}(x).$$
Moreover, for any $x\in K\backslash N_0$,
\begin{enumerate}[label=(\arabic*)]
\item $R_x$ and $T_x$ are $\mathbb{Q}$-seminorms on $\mathcal{D}$, that is, for any $f,g\in \mathcal{D}$ and any $c\in \mathbb{Q}$, $R_x(f+g)\le R_x(f)+R_x(g)$ and $R_x(cf)=\lvert c\rvert R_x(f)$, and similarly for $T_x$;
\item for any $f\in \mathcal{D}$,
$$\frac{1}{C^{\frac{1}{p}}}R_x(f)\le T_x(f)\le C^{\frac{1}{p}}R_x(f),$$
where $C\ge1$ is the constant in Lemma \ref{lem_quasi_unique};
\item $\sum_{k\ge1}b_k R_x(f_k)^p=1$.
\end{enumerate}
\end{lemma}

\begin{proof}
For any $f,g\in \mathcal{D}$ and any $c\in \mathbb{Q}$, by the properties of energy measures, there exist Borel $\mu$-null sets $N_{f,g}^{(c)}$, $N_{f,g}$, $N_{c,f}^{(c)}$, $N_{c,f}$, $N_f$, $N_1$, all containing $K^\partial_*$, such that
$$\left(\frac{\mathrm{d}\Gamma^{(c)}(f+g)}{\mathrm{d}\mu}(x)\right)^{\frac{1}{p}}\le \left(\frac{\mathrm{d}\Gamma^{(c)}(f)}{\mathrm{d}\mu}(x)\right)^{\frac{1}{p}}+\left(\frac{\mathrm{d}\Gamma^{(c)}(g)}{\mathrm{d}\mu}(x)\right)^{\frac{1}{p}}\text{ for any }x\in K\backslash N^{(c)}_{f,g},$$
$$\left(\frac{\mathrm{d}\Gamma(f+g)}{\mathrm{d}\mu}(x)\right)^{\frac{1}{p}}\le \left(\frac{\mathrm{d}\Gamma(f)}{\mathrm{d}\mu}(x)\right)^{\frac{1}{p}}+\left(\frac{\mathrm{d}\Gamma(g)}{\mathrm{d}\mu}(x)\right)^{\frac{1}{p}}\text{ for any }x\in K\backslash N_{f,g},$$
$$\frac{\mathrm{d}\Gamma^{(c)}(cf)}{\mathrm{d}\mu}(x)=\lvert c\rvert^p \frac{\mathrm{d}\Gamma^{(c)}(f)}{\mathrm{d}\mu}(x)\text{ for any }x\in K\backslash N^{(c)}_{c,f},$$
$$\frac{\mathrm{d}\Gamma(cf)}{\mathrm{d}\mu}(x)=\lvert c\rvert^p \frac{\mathrm{d}\Gamma(f)}{\mathrm{d}\mu}(x)\text{ for any }x\in K\backslash N_{c,f},$$
$$\frac{1}{C}\frac{\mathrm{d}\Gamma^{(c)}(f)}{\mathrm{d}\mu}(x)\le \frac{\mathrm{d}\Gamma(f)}{\mathrm{d}\mu}(x)\le C\frac{\mathrm{d}\Gamma^{(c)}(f)}{\mathrm{d}\mu}(x)\text{ for any }x\in K\backslash N_f,$$
$$\sum_{k\ge1}b_k \frac{\mathrm{d}\Gamma^{(c)}(f_k)}{\mathrm{d}\mu}(x)=1\text{ for any }x\in K\backslash N_1.$$
Since $\mathcal{D}$ is countable, let
$$N_0:=\left(\bigcup_{f,g\in \mathcal{D}}(N^{(c)}_{f,g}\cup N_{f,g})\right)\bigcup\left(\bigcup_{c\in \mathbb{Q},f\in \mathcal{D}}(N^{(c)}_{c,f}\cup N_{c,f})\right)\bigcup\left(\bigcup_{f\in \mathcal{D}}N_f\right)\bigcup N_1,$$
then $N_0\supseteq K^\partial_*$ is a Borel $\mu$-null set. For any $f\in \mathcal{D}$ and any $x\in K\backslash N_0$, let
$$R_x(f):=\left(\frac{\mathrm{d}\Gamma^{(c)}(f)}{\mathrm{d}\mu}(x)\right)^{\frac{1}{p}},\qquad T_x(f):=\left(\frac{\mathrm{d}\Gamma(f)}{\mathrm{d}\mu}(x)\right)^{\frac{1}{p}}.$$
Then the results follow directly.
\end{proof}

For any $x\in K\backslash N_0$, let
$$\mathrm{ker}(R_x):=\{f\in \mathcal{D}:R_x(f)=0\},$$
then $\mathrm{ker}(R_x)$ is a $\mathbb{Q}$-linear subspace of $\mathcal{D}$. Let $\mathcal{D}/\mathrm{ker}(R_x)=\{[f]_x:f\in \mathcal{D}\}$ denote the quotient space, where $[f]_x$ is the equivalence class containing $f$. Define $\lVert [f]_x\rVert_x:=R_x(f)$. This is well-defined: its value is independent of the choice of representative $f$ in $[f]_x$. Thus $(\mathcal{D}/\mathrm{ker}(R_x),\lVert \cdot\rVert_x)$ is a normed linear space over $\mathbb{Q}$. Let $(V_x,\lVert \cdot\rVert_x)$ denote its completion, that is,
\begin{equation}\label{eq_Vx}
V_x:=\overline{\left\{[f]_x:f\in \mathcal{D}\right\}}^{\,\lVert \cdot\rVert_x},
\end{equation}
then $(V_x,\lVert \cdot\rVert_x)$ is a normed linear space over $\mathbb{R}$. Moreover, for any $f\in \mathcal{D}$, $\mathbb{R}[f]_x:=\{t[f]_x:t\in \mathbb{R}\}\subseteq V_x$, and
\begin{equation}\label{eq_Vx_norm}
\lVert t[f]_x\rVert_x=\lvert t\rvert\cdot\lVert [f]_x\rVert_x=\lvert t\rvert R_x(f).
\end{equation}

We apply Theorem \ref{thm_R1} through the following result.

\begin{theorem}\label{thm_dim1}
For $\mu$-a.e. $x\in K\backslash N_0$, we have
$$\mathrm{dim}_{\mathbb{R}}V_x=1.$$
\end{theorem}

\begin{proof}
To begin with, we show that for $x\in K\backslash N_0$, if $\mathrm{dim}_{\mathbb{R}}V_x\ge2$, then there exist $f_1,f_2\in \mathcal{D}$ such that $[f_1]_x$, $[f_2]_x$ are linearly independent in $V_x$ over $\mathbb{R}$. Indeed, suppose that this is false. First, we have $\mathcal{D}\backslash \mathrm{ker}(R_x)\ne \emptyset$, otherwise, $\mathcal{D}/\mathrm{ker}(R_x)=\{[0]_x\}$, which implies $V_x=\{[0]_x\}$, contradicting $\mathrm{dim}_{\mathbb{R}}V_x\ge2$. Take $f\in \mathcal{D}\backslash \mathrm{ker}(R_x)$, then $[f]_x\ne[0]_x$ in $V_x$, hence $\mathbb{R}[f]_x$ is a $1$-dimensional subspace of $V_x$ over $\mathbb{R}$. By assumption, for any $g\in \mathcal{D}$, $[g]_x\in \mathbb{R}[f]_x\subseteq V_x$. Then (\ref{eq_Vx}) gives
$$V_x=\overline{\left\{[g]_x:g\in \mathcal{D}\right\}}^{\,\lVert \cdot\rVert_x}\subseteq \mathbb{R}[f]_x\subseteq V_x,$$
hence $\mathrm{dim}_{\mathbb{R}}V_x=1$, a contradiction. Therefore,
\begin{align}
&\left\{x\in K\backslash N_0:\mathrm{dim}_{\mathbb{R}}V_x\ge2\right\}\nonumber\\
&=\bigcup_{f_1,f_2\in \mathcal{D}}\left\{x\in K\backslash N_0:\text{$[f_1]_x$, $[f_2]_x$ are linearly independent in $V_x$ over $\mathbb{R}$}\right\}.\label{eq_dim1a}
\end{align}

Then, for any $f_1,f_2\in \mathcal{D}$, we show that
\begin{align}
&\left\{x\in K\backslash N_0:\text{$[f_1]_x$, $[f_2]_x$ are linearly independent in $V_x$ over $\mathbb{R}$}\right\}\nonumber\\
&=\left\{x\in K\backslash N_0:\inf_{\boldsymbol{\lambda}=(\lambda_1,\lambda_2)\in \mathbb{Q}^2\backslash\{0\}}\frac{R_x(\lambda_1f_1+\lambda_2f_2)}{\lvert \boldsymbol{\lambda}\rvert}>0\right\}.\label{eq_dim1b}
\end{align}
Indeed, for any $x\in K\backslash N_0$, by Lemma \ref{lem_martingale},
$$\boldsymbol{\lambda}=(\lambda_1,\lambda_2)\mapsto \frac{\lVert \lambda_1[f_1]_x+\lambda_2[f_2]_x\rVert_x}{\lvert\boldsymbol{\lambda}\rvert},$$
is a continuous function on $\mathbb{R}^2\backslash\{0\}$, hence
\begin{align*}
&\inf_{\boldsymbol{\lambda}=(\lambda_1,\lambda_2)\in \mathbb{R}^2\backslash\{0\}}\frac{\lVert \lambda_1[f_1]_x+\lambda_2[f_2]_x\rVert_x}{\lvert\boldsymbol{\lambda}\rvert}\\
&=\inf_{\boldsymbol{\lambda}=(\lambda_1,\lambda_2)\in \mathbb{Q}^2\backslash\{0\}}\frac{\lVert \lambda_1[f_1]_x+\lambda_2[f_2]_x\rVert_x}{\lvert\boldsymbol{\lambda}\rvert}=\inf_{\boldsymbol{\lambda}=(\lambda_1,\lambda_2)\in \mathbb{Q}^2\backslash\{0\}}\frac{R_x(\lambda_1f_1+\lambda_2f_2)}{\lvert\boldsymbol{\lambda}\rvert}.
\end{align*}
Since $[f_1]_x$, $[f_2]_x$ are linearly independent in $V_x$ over $\mathbb{R}$ if and only if $\lVert \lambda_1[f_1]_x+\lambda_2[f_2]_x\rVert_x>0$ for any $\boldsymbol{\lambda}=(\lambda_1,\lambda_2)\in \mathbb{R}^2\backslash\{0\}$, we obtain (\ref{eq_dim1b}).

Next, suppose that $\mu(\{x\in K\backslash N_0:\mathrm{dim}_{\mathbb{R}}V_x\ge2\})>0$. Combining (\ref{eq_dim1a}) with (\ref{eq_dim1b}), we find $f_1,f_2\in \mathcal{D}$ and $\delta>0$ such that
$$\mu \left(\left\{x\in K\backslash N_0:\inf_{\boldsymbol{\lambda}=(\lambda_1,\lambda_2)\in \mathbb{Q}^2\backslash\{0\}}\frac{R_x(\lambda_1f_1+\lambda_2f_2)}{\lvert \boldsymbol{\lambda}\rvert}>\delta\right\}\right)>0.$$
Let $A$ denote the set in the parentheses above, then for any $x\in A$, $R_x(f_1)>\delta$ and $R_x(f_2)>\delta$. Let $\boldsymbol{f}=(f_1,f_2)\in \mathcal{D}^2\subseteq \mathcal{F}^2$, then for any $\boldsymbol{\lambda}=(\lambda_1,\lambda_2)\in \mathbb{S}^1$ and any $x\in A$,
$$\Psi_{\boldsymbol{f}}(\boldsymbol{\lambda})(x)=\frac{\lVert \lambda_1[f_1]_x+\lambda_2[f_2]_x\rVert_{x}^p}{\lVert [f_1]_x\rVert_x^p+\lVert [f_2]_x\rVert_x^p}=\frac{\lVert \lambda_1[f_1]_x+\lambda_2[f_2]_x\rVert_{x}^p}{R_x(f_1)^p+R_x(f_2)^p},$$
hence
$$\inf_{\boldsymbol{\lambda}\in \mathbb{S}^1}\Psi_{\boldsymbol{f}}(\boldsymbol{\lambda})(x)=\inf_{\boldsymbol{\lambda}=(\lambda_1,\lambda_2)\in \mathbb{Q}^2\backslash\{0\}}\frac{R_x(\lambda_1f_1+\lambda_2f_2)^p}{\lvert \boldsymbol{\lambda}\rvert^p \left(R_x(f_1)^p+R_x(f_2)^p\right)}\ge \frac{\delta^p}{R_x(f_1)^p+R_x(f_2)^p}>0,$$
which implies $\boldsymbol{f}$ is $p$-independent on $A$. Moreover,
\begin{align*}
&\Lambda_{\boldsymbol{f}}(A)=\int_{A}\left(\frac{\mathrm{d}\Gamma^{(c)}(f_1)}{\mathrm{d}\mu}(x)+\frac{\mathrm{d}\Gamma^{(c)}(f_2)}{\mathrm{d}\mu}(x)\right)\mu(\mathrm{d}x)\\
&=\int_A \left(R_x(f_1)^p+R_x(f_2)^p\right)\mu(\mathrm{d}x)\ge 2\delta^p\mu(A)>0.
\end{align*}
Therefore,
$$\Lambda_{\boldsymbol{f}}\left(\left\{\text{$\boldsymbol{f}$ is $p$-independent}\right\}\right)\ge \Lambda_{\boldsymbol{f}}(A)\ge 2\delta^p\mu(A)>0,$$
contradicting Theorem \ref{thm_R1}. Hence, $\mathrm{dim}_{\mathbb{R}}V_x\le 1$ for $\mu$-a.e. $x\in K\backslash N_0$.

Finally, let
$$N_1:=\left\{x\in K\backslash N_0:\mathrm{dim}_{\mathbb{R}}V_x=0\right\}=\left\{x\in K\backslash N_0:R_x(f)=0\text{ for any }f\in \mathcal{D}\right\},$$
then for any $f\in \mathcal{D}$,
$$\Gamma^{(c)}(f)(N_1)=\int_{N_1}R_x(f)^p \mu(\mathrm{d}x)=0.$$
For any $f\in \mathcal{F}$, since $\mathcal{D}$ is $\mathcal{E}^{(c)}_1$-dense in $\mathcal{F}$, there exists $\{f_n\}_n\subseteq \mathcal{D}$ that $\mathcal{E}^{(c)}_1$-converges to $f$, which implies
$$\Gamma^{(c)}(f)(N_1)^{\frac{1}{p}}=\lvert \Gamma^{(c)}(f)(N_1)^{\frac{1}{p}}-\Gamma^{(c)}(f_n)(N_1)^{\frac{1}{p}}\rvert\le \Gamma^{(c)}(f-f_n)(N_1)^{\frac{1}{p}}\le \mathcal{E}^{(c)}(f-f_n)^{\frac{1}{p}}\to0.$$
Hence, $\Gamma^{(c)}(f)(N_1)=0$. Since $f\in \mathcal{F}$ is arbitrary, $\mu(N_1)=0$. Therefore, $\mathrm{dim}_{\mathbb{R}}V_x=1$ for $\mu$-a.e. $x\in K\backslash N_0$.
\end{proof}

By enlarging $N_0$ by the $\mu$-null set in Theorem \ref{thm_dim1}, we may assume that $\mathrm{dim}_{\mathbb{R}}V_x=1$ for any $x\in K\backslash N_0$. The following result completes \ref{main_step1}.

\begin{proposition}\label{prop_step1}
Let $a:K\to [0,+\infty)$ be given by
$$a(x):=
\begin{cases}
\sum_{k\ge1}b_kT_x(f_k)^p&\text{if }x\in K\backslash N_0,\\
1&\text{if }x\in N_0.
\end{cases}
$$
Then $a$ is Borel measurable, $\frac{1}{C}\le a\le C$ on $K$, where $C\ge1$ is the constant in Lemma \ref{lem_quasi_unique}, and for any $f\in \mathcal{F}$,
$$\Gamma(f)=a\Gamma^{(c)}(f).$$
\end{proposition}

\begin{proof}
By Lemma \ref{lem_martingale}, $a$ is Borel measurable. Fix any $x\in K\backslash N_0$. Since $\mathrm{dim}_{\mathbb{R}}V_x=1$, there exists $g\in \mathcal{D}$ such that $R_x(g)>0$, hence $V_x\ni [g]_x\ne0$, and for any $f\in \mathcal{D}$, $[f]_x$, $[g]_x$ are linearly dependent in $V_x$ over $\mathbb{R}$, and there exists $t\in \mathbb{R}$ such that $[f]_x=t[g]_x\in V_x$. By (\ref{eq_Vx}) and (\ref{eq_Vx_norm}), there exists $\{r_n\}_n\subseteq \mathbb{Q}$ with $\lim_{n\to+\infty}r_n=t$, such that
$$\lVert [f]_x-r_n[g]_x\rVert_x=\lVert [f-r_ng]_x\rVert_x=R_x(f-r_ng)\to0.$$
Lemma \ref{lem_martingale} gives $T_x(g)>0$, $T_x(f-r_ng)\to0$, and
$$T_x(f)=\lim_{n\to+\infty}T_x(r_ng)=\lim_{n\to+\infty}\lvert r_n\rvert T_x(g)=\frac{T_x(g)}{R_x(g)}\lim_{n\to+\infty}\lvert r_n\rvert R_x(g)=\frac{T_x(g)}{R_x(g)}R_x(f),$$
moreover, $a(x)=\frac{T_x(g)^p}{R_x(g)^p}\sum_{k\ge1}b_k R_x(f_k)^p=\frac{T_x(g)^p}{R_x(g)^p}\in[\frac{1}{C},C]$, and $T_x(f)^p=a(x)R_x(f)^p$.

Fix any $A\in\mathcal{B}(K)$. First, for any $f\in \mathcal{D}$,
$$\Gamma(f)(A)=\int_A \frac{\mathrm{d}\Gamma(f)}{\mathrm{d}\mu}\mathrm{d}\mu=\int_A T_x(f)^p \mu(\mathrm{d}x)=\int_Aa(x)R_x(f)^p \mu(\mathrm{d}x)=\int_A a \mathrm{d}\Gamma^{(c)}(f).$$
Second, for any $f\in \mathcal{F}$, there exists $\{f_n\}_n\subseteq \mathcal{D}$ that converges to $f$ both in $\mathcal{E}^{(c)}_1$ and in $\mathcal{E}_1$. For any $n$, $\Gamma(f_n)(A)=\int_A a \mathrm{d}\Gamma^{(c)}(f_n)$, where
$$\lvert \Gamma(f)(A)^{\frac{1}{p}}-\Gamma(f_n)(A)^{\frac{1}{p}}\rvert\le \Gamma(f-f_n)(A)^{\frac{1}{p}}\le \mathcal{E}(f-f_n)^{\frac{1}{p}}\to0,$$
and
\begin{equation}\label{eq_temp_ffn}
\lvert \left(\int_A a \mathrm{d}\Gamma^{(c)}(f)\right)^{\frac{1}{p}}-\left(\int_A a \mathrm{d}\Gamma^{(c)}(f_n)\right)^{\frac{1}{p}}\rvert\le \left(\int_A a \mathrm{d}\Gamma^{(c)}(f-f_n)\right)^{\frac{1}{p}}\le \left(C \mathcal{E}^{(c)}(f-f_n)\right)^{\frac{1}{p}}\to0.
\end{equation}
Hence $\Gamma(f)(A)=\int_A a \mathrm{d}\Gamma^{(c)}(f)$.
\end{proof}

The function $a$ has the following invariance properties.

\begin{proposition}\label{prop_ss_a}
Let $a$ be the function given in Proposition \ref{prop_step1}. Then for any $w\in W_*$ and any $\psi\in D_4$,
$$a\circ F_w=a\ \mu\text{-a.e.},\qquad a\circ\psi=a\ \mu\text{-a.e.}$$
\end{proposition}

\begin{proof}
For any $f\in \mathcal{F}$, $w\in W_*$, and $A\in \mathcal{B}(K)$, by Proposition \ref{prop_step1} and Lemma \ref{lem_meas_PI}\ref{item_meas6},
$$\Gamma(f)(F_w(A))=\rho^{\lvert w\rvert}\Gamma(f\circ F_w)(A)=\rho^{\lvert w\rvert}\int_A a \mathrm{d}\Gamma^{(c)}(f\circ F_w),$$
and
\begin{align*}
&\Gamma(f)(F_w(A))=\int_{F_wA}a \mathrm{d}\Gamma^{(c)}(f)=\int_{K_w}(a1_{F_wA})\mathrm{d}\Gamma^{(c)}(f)\\
&=\rho^{\lvert w\rvert}\int_K \left((a1_{F_wA})\circ F_w\right)\mathrm{d}\Gamma^{(c)}(f\circ F_w)=\rho^{\lvert w\rvert}\int_A \left(a\circ F_w\right)\mathrm{d}\Gamma^{(c)}(f\circ F_w),
\end{align*}
hence
$$\int_A a \mathrm{d}\Gamma^{(c)}(f\circ F_w)=\int_A \left(a\circ F_w\right)\mathrm{d}\Gamma^{(c)}(f\circ F_w).$$

For $w\in W_*$, let $\varphi_w:K\to K_w$ be the folding map, and let $\theta_w=F_w^{-1}\circ \varphi_w:K\to K$ be the map given in Remark \ref{rmk_01law}. For any $f\in \mathcal{F}\cap C(K)$, let $g:=f\circ \theta_w=f\circ F_w^{-1}\circ\varphi_w$, then by \ref{item_SSEe}, $g\in \mathcal{F}\cap C(K)$. By (\ref{eq_varphiF}), $g\circ F_w=(f\circ F_w^{-1}\circ\varphi_w)\circ F_w=f$. By the preceding identity,
$$\int_A a \mathrm{d}\Gamma^{(c)}(f)=\int_A \left(a\circ F_w\right)\mathrm{d}\Gamma^{(c)}(f).$$
For any $f\in \mathcal{F}$, since $\mathcal{F}\cap C(K)$ is $\mathcal{E}^{(c)}_1$-dense in $\mathcal{F}$, there exists $\{f_n\}_n\subseteq \mathcal{F}\cap C(K)$ that $\mathcal{E}^{(c)}_1$-converges to $f$. For any $n$,
$$\int_A a \mathrm{d}\Gamma^{(c)}(f_n)=\int_A \left(a\circ F_w\right)\mathrm{d}\Gamma^{(c)}(f_n).$$
By the same argument used to prove (\ref{eq_temp_ffn}),
$$\int_A a \mathrm{d}\Gamma^{(c)}(f_n)\to \int_A a \mathrm{d}\Gamma^{(c)}(f),\qquad\int_A \left(a\circ F_w\right)\mathrm{d}\Gamma^{(c)}(f_n)\to \int_A \left(a\circ F_w\right)\mathrm{d}\Gamma^{(c)}(f),$$
hence
$$\int_A a \mathrm{d}\Gamma^{(c)}(f)=\int_A \left(a\circ F_w\right)\mathrm{d}\Gamma^{(c)}(f),$$
which implies $a=a\circ F_w$ $\Gamma^{(c)}(f)$-a.e. Since $f\in \mathcal{F}$ is arbitrary, we have $a=a\circ F_w$ $\mu$-a.e.

For any $\psi\in D_4$, $A\in \mathcal{B}(K)$, and $f\in \mathcal{F}$, let $B=\psi(A)\in \mathcal{B}(K)$, by Proposition \ref{prop_step1} and (\ref{eq_Gamma_D4}),
$$\Gamma(f\circ \psi^{-1})(B)=\int_B a \mathrm{d}\Gamma^{(c)}(f\circ \psi^{-1})=\int_{B} a \mathrm{d}(\psi_*(\Gamma^{(c)}(f)))=\int_{A} (a\circ \psi) \mathrm{d}\Gamma^{(c)}(f),$$
and
$$\Gamma(f\circ \psi^{-1})(B)=\Gamma(f)(\psi^{-1}(B))=\int_A a \mathrm{d}\Gamma^{(c)}(f).$$
hence
$$\int_A a \mathrm{d}\Gamma^{(c)}(f)=\int_A (a\circ \psi)\mathrm{d}\Gamma^{(c)}(f),$$
which implies $a=a\circ \psi$ $\Gamma^{(c)}(f)$-a.e. Since $f\in \mathcal{F}$ is arbitrary, we have $a=a\circ \psi$ $\mu$-a.e.
\end{proof}

We now use Theorem \ref{thm_01law} and Proposition \ref{prop_ss_a} to show that the function $a$ is indeed constant, as stated in the following result. This completes \ref{main_step2} and hence the proof of Theorem \ref{thm_unique}.

\begin{proposition}
There exists a constant $c\in [\frac{1}{C},C]$, where $C\ge1$ is the constant in Lemma \ref{lem_quasi_unique}, such that $a=c$ $\mu$-a.e.
\end{proposition}

\begin{proof}
For any $r\in \mathbb{R}$, $i=0,1,\ldots,7$, and $\psi\in D_4$, by Proposition \ref{prop_ss_a},
$$1_{\{a>r\}}=1_{\{a\circ F_i>r\}}=1_{\{a>r\}}\circ F_i\ \mu\text{-a.e.},\qquad 1_{\{a>r\}}=1_{\{a\circ \psi>r\}}=1_{\{a>r\}}\circ \psi\ \mu\text{-a.e.}$$
Then Theorem \ref{thm_01law} gives either $\mu(\{a>r\})=0$ or $\mu(K\backslash\{a>r\})=\mu(\{a\le r\})=0$. Since $\frac{1}{C}\le a\le C$ on $K$ by Proposition \ref{prop_step1}, $\mu(\{a>r\})=0$ for any $r\ge C$ and $\mu(\{a\le r\})=0$ for any $r<\frac{1}{C}$. Let
$$c:=\inf \left\{r\in \mathbb{Q}:\mu(\{a>r\})=0\right\},$$
then $c\in[\frac{1}{C},C]$. For any $r>c$, $\mu(\{a>r\})=0$, hence
$$\mu(\{a>c\})=\mu\left(\bigcup_{r\in \mathbb{Q}\cap (c,+\infty)}\{a>r\}\right)=0.$$
For any $r<c$, $\mu(\{a\le r\})=0$, hence
$$\mu(\{a<c\})=\mu\left(\bigcup_{r\in \mathbb{Q}\cap (-\infty,c)}\{a\le r\}\right)=0.$$
Therefore, $a=c$ $\mu$-a.e.
\end{proof}

\section{Proof of Proposition \ref{prop_Hin310}}\label{sec_Hin310}

In this section, we give some results in potential theory on Euclidean spaces. Since we have not found these results in the existing literature, we include their proofs for completeness. We write $L^p(\mathbb{R}^2)=L^p(\mathbb{R}^2;\mathcal{L}^2)$ and use $\|\cdot\|_p$ for the norm in $L^p(\mathbb{R}^2)$. Let $\xi:\mathbb{R}^2\to[0,+\infty)$ satisfy $\eta:=\xi^{\frac{1}{p}}\in W^{1,p}(\mathbb{R}^2)$ and $0\le \rho\le \xi$ $\mathcal{L}^2$-a.e. Since $\xi=\eta^p\in L^1(\mathbb{R}^2)$, the measure $(\xi+1)\mathcal{L}^2$ is finite on bounded sets and has the same null sets as Lebesgue measure $\mathcal{L}^2$. In what follows, ``a.e." means $\mathcal{L}^2$-a.e., which is equivalent to $(\xi+1)\mathcal{L}^2$-a.e.

For any $u\in C_c^1(\mathbb{R}^2)$, let
$$Q^\xi(u):=\int_{\mathbb{R}^2}|\nabla u|^p(\xi+1)\mathrm{d}\mathcal{L}^2,\qquad Q_1^\xi(u):=\int_{\mathbb{R}^2}(|\nabla u|^p+|u|^p)(\xi+1)\mathrm{d}\mathcal{L}^2.$$
Let $D=\mathrm{Dom}(Q^\xi)$ be the $Q_1^\xi$-closure of $C_c^1(\mathbb{R}^2)$. Recall that for any $u\in W^{1,p}(\mathbb{R}^2)$,
$$Q(u)=\int_{\mathbb{R}^2}|\nabla u|^p\mathrm{d}\mathcal{L}^2,\qquad Q_1(u)=\int_{\mathbb{R}^2}(|\nabla u|^p+\lvert u\rvert^p)\mathrm{d}\mathcal{L}^2.$$
For any open subset $U\subseteq\mathbb{R}^2$, let
\begin{align*}
\mathrm{Cap}^\xi(U)&:=\inf\left\{Q_1^\xi(u):u\in D,\ u\ge1\text{ a.e. in }U\right\},\\
\mathrm{Cap}^{1,p}(U)
&:=\inf\left\{Q_1(u):u\in W^{1,p}(\mathbb{R}^2),\ u\ge1\text{ a.e. in }U\right\}.
\end{align*}
For any subset $A\subseteq\mathbb{R}^2$, let
$$
\mathrm{Cap}^\xi(A):=\inf_{\substack{U\supseteq A\\U:\text{ open}}}\mathrm{Cap}^\xi(U),\qquad \mathrm{Cap}^{1,p}(A):=\inf_{\substack{U\supseteq A\\U\text{ open}}}\mathrm{Cap}^{1,p}(U),$$
here we use the convention that $\inf\emptyset=+\infty$. Since $1\le \rho+1\le \xi+1$, for any subset $A\subseteq \mathbb{R}^2$,
$$\mathrm{Cap}^{1,p}(A)\le \mathrm{Cap}^{\rho}(A)\le \mathrm{Cap}^{\xi}(A).$$
It suffices to prove the following result.

\begin{theorem}\label{thm_cap_main}
For any subset $A\subseteq\mathbb{R}^2$ and any $C\ge2$,
\begin{equation}\label{eq_comparison}
\mathrm{Cap}^{1,p}(A)\le\mathrm{Cap}^\xi(A)\le2^{p-1}\left(C^p\mathrm{Cap}^{1,p}(A)+\varepsilon(C)\right),\end{equation}
where
\begin{equation}\label{eq_tail}
\varepsilon(C):=\int_{\{\eta>C-1\}}\left(|\nabla\eta|^p+(\eta+1)^p\right)\mathrm{d}\mathcal{L}^2\to0\text{ as }C\to+\infty.
\end{equation}
Consequently, for any subset $B\subseteq\mathbb{R}^2$, $\mathrm{Cap}^{1,p}(B)=0$ if and only if $\mathrm{Cap}^\xi(B)=0$.
\end{theorem}

Since $\xi+1\ge1$, we have $Q^\xi(u)\ge Q(u)$ for any $u\in C_c^1(\mathbb{R}^2)$, which gives $D\subseteq W^{1,p}(\mathbb{R}^2)$.

Fix a non-negative mollifier $\varrho\in C_c^\infty(B(0,1))$ with $\int_{\mathbb{R}^2}\varrho\mathrm{d}\mathcal{L}^2=1$, and let
\begin{equation}\label{eq_mollifier}
\varrho^{(\delta)}(x):=\frac{1}{\delta^2}\varrho \left(\frac{x}{\delta}\right)\text{ for any }x\in \mathbb{R}^2\text{ and any }\delta>0.
\end{equation}

\begin{lemma}\label{lem_lipschitz}
For any Lipschitz function $u$ with compact support, we have $u\in D$.
\end{lemma}

\begin{proof}
Assume that $\mathrm{supp}(u)\subseteq B(0,R)$ for some $R>0$. Let $M:=\mathrm{sup}_{\mathbb{R}^2}\lvert u\rvert$ and $L:=\mathrm{Lip}(u)$. For any $\delta\in(0,1)$, let $u^{(\delta)}=u\ast\varrho^{(\delta)}$, then $u^{(\delta)}\in C_c^\infty(\mathbb{R}^2)\subseteq D$, $\mathrm{supp}(u^{(\delta)})\subseteq B(0,R+1)$, $\lvert u^{(\delta)}\rvert\le M$, $\lvert \nabla u^{(\delta)}\rvert\le L$, $u^{(\delta)}\rightrightarrows u$, and $\nabla u^{(\delta)}=(\nabla u)\ast\varrho^{(\delta)}\to\nabla u$ a.e. as $\delta\downarrow0$. Hence
$$(|u^{(\delta)}-u|^p+|\nabla u^{(\delta)}-\nabla u|^p)(\xi+1)\le 2^p \left(M^p+L^p\right)(\xi+1)1_{B(0,R+1)},$$
where the RHS is in $L^1(\mathbb{R}^2)$. By the dominated convergence theorem, $u\in D$.
\end{proof}

\begin{lemma}\label{lem_multiplier}
If $u\in D$ and $v$ is a bounded Lipschitz function, then $uv\in D$.
\end{lemma}

\begin{proof}
Take $\{u_n\}\subseteq C_c^1(\mathbb{R}^2)$ that $Q_1^\xi$-converges to $u$. Then for any $n$, $u_nv$ is a Lipschitz function with compact support, by Lemma \ref{lem_lipschitz}, $u_nv\in D$. Notice that
$$\lvert (u_n-u)v\rvert\le \sup_{\mathbb{R}^2}\lvert v\rvert\cdot \lvert u_n-u\rvert,$$
$$\lvert\nabla ((u_n-u)v)\rvert\le \lvert u_n-u\lvert \cdot\lvert\nabla v\rvert+\lvert v\rvert \cdot\lvert\nabla (u_n-u)\rvert\le \mathrm{Lip}(v)\lvert u_n-u\rvert+\sup_{\mathbb{R}^2}\lvert v\rvert\cdot\lvert\nabla (u_n-u)\rvert.$$
Thus $\{u_nv\}$ is $Q_1^\xi$-convergent, and hence $uv\in D$.
\end{proof}

\begin{lemma}\label{lem_weak}
Let $\{u_n\}_n\subseteq D$ satisfy $\sup_n Q_1^\xi(u_n)<+\infty$, and $u_n\to u$ a.e. Then $u\in D$ and
$$Q_1^\xi(u)\le\varliminf_{n\to+\infty}Q_1^\xi(u_n).$$
\end{lemma}

\begin{proof}
Passing to a subsequence, we may assume that $\lim_{n\to+\infty}Q_1^\xi(u_n)$ exists. Since $(D,(Q^\xi_1)^{\frac{1}{p}})$ is a reflexive Banach space and $\{u_n\}_n$ is $Q^\xi_1$-bounded, by the Banach--Alaoglu theorem (see \cite[Theorems 3.16 and 3.18]{Bre11}) together with Mazur’s lemma, for which we refer to \cite[Theorem 2 in Section V.1]{Yos95}, there exist a subsequence, still denoted by $\{u_n\}_n$, and $v\in D$, such that $\{u_n\}_n$ is $Q^\xi_1$-weakly-convergent to $v$. Moreover, for any $n$, there exist an integer $I_n\ge n$ and coefficients $\lambda^{(n)}_k\ge0$, $k=n,\ldots,I_n$, such that $\sum_{k=n}^{I_n}\lambda_k^{(n)}=1$, and $\{\sum_{k=n}^{I_n}\lambda_k^{(n)}u_k\}_n$ is $Q^\xi_1$-convergent to $v$, and hence also $L^p(\mathbb{R}^2;(\xi+1)\mathcal{L}^2)$-convergent to $v$, then it has a subsequence that converges to $v$ a.e. Since $u_n$ converges to $u$ a.e., the same subsequence also converges to $u$ a.e. Hence $u=v\in D$ and $\{u_n\}_n$ is $Q^\xi_1$-weakly-convergent to $u$, which gives $Q^\xi_1(u)\le \varliminf_{n\to+\infty}Q^\xi_1(u_n)$.
\end{proof}

For any $C\ge2$, let $F_C:\mathbb{R}\to \mathbb{R}$ be given by
$$
F_C(s):=\left(1-\frac{C}{s_++1}\right)_+
=\begin{cases}
0,&\text{if }s\le C-1,\\
1-\frac{C}{s+1},&\text{if }s>C-1,
\end{cases}
$$
see Figure \ref{fig_FC}.

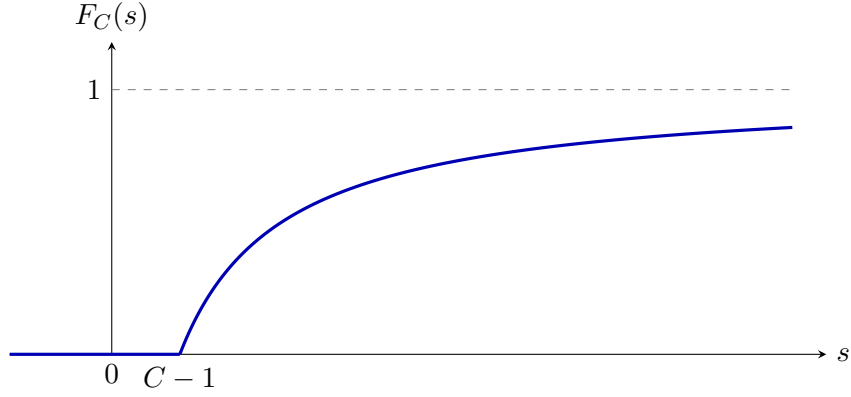
\begin{figure}[ht]
  \centering
  \begin{tikzpicture}[x=0.45cm,y=3.5cm,>=stealth]
    \draw[->] (-3,0) -- (21,0) node[right] {$s$};
    \draw[->] (0,0) -- (0,1.18) node[above] {$F_C(s)$};
    \node[below] at (0,0) {$0$};

    \draw[dashed,gray] (0,1) -- (20,1);
    \node[left] at (0,1) {$1$};

    \draw[very thick,blue!70!black] (-3,0) -- (2,0);
    \draw[very thick,blue!70!black,
          domain=2:20,samples=150,smooth]
      plot (\x,{1-3/(1+\x)});

    \node[below=1pt] at (2,0) {$C-1$};
  \end{tikzpicture}
  \caption{The function $F_C(s)=(1-\frac{C}{s_++1})_+$}
  \label{fig_FC}
\end{figure}

It is obvious that $F_C$ is a non-negative Lipschitz function, bounded above by $1$, and satisfies $F_C(0)=0$. We have the following elementary but useful properties. For any $t,s\ge0$,
\begin{gather}
(s+1)(1-F_C(s))\le C,\label{eq_cutoff_one}\\
0\le F_C(s)\le {1}_{\{s>C-1\}},\label{eq_cutoff_two}\\
(s+1)F_C'(s)\le {1}_{\{s>C-1\}},\label{eq_cutoff_three}\\
(t+1)(s+1)|F_C(t)-F_C(s)|\le C|t-s|.\label{eq_cutoff_difference}
\end{gather}
Let $f_C:=F_C(\eta)$, then $f_C\in W^{1,p}(\mathbb{R}^2)$ and $\nabla f_C=F_C'(\eta)\nabla\eta$ a.e. Notice also that
\begin{equation}\label{eq_weight_bound}
\xi+1\le(\eta+1)^p\le 2^{p-1}(\xi+1),\qquad(\eta+1)f_C\le(\eta+1) {1}_{\{\eta>1\}}\le2\eta.
\end{equation}

\begin{lemma}\label{lem_fCinD}
For any $C\ge2$, we have $f_C\in D$.
\end{lemma}

\begin{proof}
Denote $f=f_C$. Let $\chi\in C_c^\infty(\mathbb{R}^2)$ satisfy $0\le\chi\le1$ on $\mathbb{R}^2$ and $\chi=1$ in $B(0,1)$. For any $n\ge1$ and any $\delta>0$, let
$$\chi_n=\chi\left(\frac{\cdot}{n}\right),\qquad f^{(\delta)}=f\ast\varrho^{(\delta)},\qquad g_n=\chi_n f^{(1/n)},$$
where $\varrho^{(\delta)}$ is given in (\ref{eq_mollifier}), then $g_n\in C_c^\infty(\mathbb{R}^2)$ and $\{g_n\}_{n\ge1}$ converges to $f$ a.e. By Lemma \ref{lem_weak}, it suffices to show that $\sup_{n\ge1} Q_1^\xi(g_n)<+\infty$. Notice that $0\le g_n\le f^{(1/n)}$ and $\nabla g_n=\chi_n\nabla f^{(1/n)}+\frac{1}{n}f^{(1/n)}\nabla\chi(\frac{\cdot}{n})$. By (\ref{eq_weight_bound}), it suffices to show that
\begin{equation}\label{eq_fCinD1}
\sup_{\delta\in(0,1]}\lVert (\eta+1)f^{(\delta)}\rVert_p<+\infty,\qquad \sup_{\delta\in(0,1]}\lVert (\eta+1)\nabla f^{(\delta)}\rVert_p<+\infty.
\end{equation}
For any scalar- or vector-valued function $k$, let
$$S_k(x):=\int_{\mathbb{R}^2}|\eta(x-z)-\eta(x)|\cdot|k(z)|\mathrm{d}z.$$
Then for $\eta\in C_c^\infty(\mathbb{R}^2)$,
$$\eta(x-z)-\eta(x)=\int_0^1 \frac{\mathrm{d}}{\mathrm{d}t}\eta(x-tz)\mathrm{d}t=-\int_0^1z\cdot\nabla\eta(x-tz)\mathrm{d}t.$$
By Minkowski's inequality,
$$\lVert\eta(\cdot-z)-\eta\rVert_p\le \int_0^1 \lvert z\rvert\cdot \lVert \nabla \eta(\cdot-tz)\rVert_p \mathrm{d}t=\lvert z\rvert\cdot \lVert \nabla \eta\rVert_p.$$
Since $\eta\in W^{1,p}(\mathbb{R}^2)$, by the density of $C_c^\infty(\mathbb{R}^2)$ in $W^{1,p}(\mathbb{R}^2)$,
$$\lVert\eta(\cdot-z)-\eta\rVert_p\le\lvert z\rvert\cdot \lVert \nabla \eta\rVert_p.$$
Then by Minkowski's inequality again,
$$
\lVert S_k\rVert_p\le\int_{\mathbb{R}^2}\lVert \eta(\cdot-z)-\eta\rVert_p \cdot \lvert k(z)\rvert \mathrm{d}z\le\lVert\nabla\eta\rVert_p\int_{\mathbb{R}^2}|z|\cdot|k(z)|\mathrm{d}z.
$$
In particular,
\begin{equation}\label{eq_kernel_bounds}
\lVert S_{\varrho^{(\delta)}}\rVert_p\le\delta \rVert\nabla\eta\rVert_p,\qquad \lVert S_{\nabla\varrho^{(\delta)}}\rVert_p\le c_\varrho\lVert\nabla\eta\rVert_p,
\end{equation}
where $c_\varrho:=\int_{\mathbb{R}^2}|z|\cdot|\nabla\varrho(z)|\mathrm{d}z<+\infty$. Since $\int_{\mathbb{R}^2}\varrho^{(\delta)}\mathrm{d}\mathcal{L}^2=1$, by (\ref{eq_cutoff_difference}) and (\ref{eq_weight_bound}), for a.e. $x\in \mathbb{R}^2$,
\begin{align*}
&0\le(\eta(x)+1)f^{(\delta)}(x)=(\eta(x)+1)f(x)+(\eta(x)+1)\int_{\mathbb{R}^2}(F_C(\eta(x-z))-F_C(\eta(x)))\varrho^{(\delta)}(z)\mathrm{d}z\\
&\le2\eta(x)+C\int_{\mathbb{R}^2}\lvert\eta(x-z)-\eta(x)\rvert\varrho^{(\delta)}(z)\mathrm{d}z=2\eta(x)+CS_{\varrho^{(\delta)}}(x),
\end{align*}
then by (\ref{eq_kernel_bounds}),
\begin{equation}\label{eq_fCinD2}
\lVert (\eta+1)f^{(\delta)}\rVert_p\le 2 \lVert \eta\rVert_p+C\delta \lVert \nabla \eta\rVert_p.
\end{equation}
Since $\int_{\mathbb{R}^2}\nabla\varrho^{(\delta)}\mathrm{d}\mathcal{L}^2=0$, by (\ref{eq_cutoff_difference}), for a.e. $x\in \mathbb{R}^2$,
\begin{align*}
&(\eta(x)+1)\lvert \nabla f^{(\delta)}(x)\rvert=(\eta(x)+1)\lvert \int_{\mathbb{R}^2}f(x-z)\nabla \varrho^{(\delta)}(z)\mathrm{d}z\rvert\\
&=(\eta(x)+1)\lvert \int_{\mathbb{R}^2}(F_C(\eta(x-z))-F_C(\eta(x)))\nabla \varrho^{(\delta)}(z)\mathrm{d}z\rvert\\
&\le C \int_{\mathbb{R}^2}\lvert \eta(x-z)-\eta(x)\rvert\cdot \lvert \nabla \varrho^{(\delta)}(z)\rvert \mathrm{d}z=CS_{\nabla \varrho^{(\delta)}}(x),
\end{align*}
then by (\ref{eq_kernel_bounds}),
$$\lVert (\eta+1)\nabla f^{(\delta)}\rVert_p\le Cc_\varrho \lVert \nabla \eta\rVert_p.$$
Combining this with (\ref{eq_fCinD2}), we obtain (\ref{eq_fCinD1}).
\end{proof}

For any $g\in W^{1,p}(\mathbb{R}^2)$ with $0\le g\le1$, let
$$W_{g,C}:=f_C+g-f_Cg=g+(1-g)f_C=f_C+(1-f_C)g.$$
Since $f_C,g\in W^{1,p}(\mathbb{R}^2)\cap L^\infty(\mathbb{R}^2)$, we have $W_{g,C}\in W^{1,p}(\mathbb{R}^2)$ and
$$\nabla W_{g,C}=(1-f_C)\nabla g+(1-g)\nabla f_C\text{ a.e.}$$
Recall that $f_C=F_C(\eta)$ and $\nabla f_C=F'_C(\eta)\nabla \eta$, by (\ref{eq_cutoff_one})--(\ref{eq_cutoff_three}),
\begin{gather*}
0\le(\eta+1)W_{g,C}\le Cg+(\eta+1){1}_{\{\eta>C-1\}},\\
(\eta+1)|\nabla W_{g,C}|\le C|\nabla g|+|\nabla\eta|{1}_{\{\eta>C-1\}}.
\end{gather*}
Combining these estimates with (\ref{eq_weight_bound}) gives
\begin{equation}\label{eq_blend_energy}
\int_{\mathbb{R}^2}(|\nabla W_{g,C}|^p+|W_{g,C}|^p)(\xi+1)\mathrm{d}\mathcal{L}^2\le2^{p-1}\left(C^p Q_1(g)+\varepsilon(C)\right),
\end{equation}
where $\varepsilon(C)$ is given in (\ref{eq_tail}). Since $\eta\in W^{1,p}(\mathbb{R}^2)$, $\lim_{C\to+\infty}\mathcal{L}^2(\{\eta>C-1\})=0$. By the dominated convergence theorem, $\lim_{C\to+\infty}\varepsilon(C)$=0.

\begin{proposition}\label{prop_blend}
For any $g\in W^{1,p}(\mathbb{R}^2)$ with $0\le g\le1$, we have $W_{g,C}\in D$, and
\begin{equation}\label{eq_blend}
Q^\xi_1(W_{g,C})\le 2^{p-1}\left(C^pQ_1(g)+\varepsilon(C)\right).
\end{equation}
Moreover, $0\le W_{g,C}\le1$, and $W_{g,C}=1$ a.e. on $\{g=1\}$.
\end{proposition}

\begin{proof}
Take $\{h_n\}\subseteq C_c^\infty(\mathbb{R}^2)$ that $Q_1$-converges to $g$. By passing to a subsequence, we may assume that $\{h_n\}$ converges to $g$ a.e. For any $n$, let $g_n=(h_n\vee0)\wedge1$, then $g_n\in W^{1,p}(\mathbb{R}^2)$ is a Lipschitz function with compact support. By Lemma \ref{lem_lipschitz}, $g_n\in D$. By Lemmas \ref{lem_multiplier} and \ref{lem_fCinD}, $f_Cg_n\in D$, hence $W_{g_n,C}\in D$. By (\ref{eq_blend_energy}), $\sup_n Q_1^\xi(W_{g_n,C})<+\infty$. Since $\{W_{g_n,C}\}_n$ converges to $W_{g,C}$ a.e., by Lemma \ref{lem_weak}, $W_{g,C}\in D$ and (\ref{eq_blend_energy}) is exactly (\ref{eq_blend}). Obviously, $0\le W_{g,C}\le 1$, and for a.e. $x\in \{g=1\}$, $W_{g,C}(x)=f_C(x)+1-f_C(x)=1$.
\end{proof}

\begin{proof}[Proof of Theorem \ref{thm_cap_main}]
Since $\xi+1\ge1$, it follows directly that $\mathrm{Cap}^{1,p}(\cdot)\le \mathrm{Cap}^\xi(\cdot)$. To prove the other inequality, fix any open subset $U\subseteq \mathbb{R}^2$. Without loss of generality, we may assume that $\mathrm{Cap}^{1,p}(U)<+\infty$. For any $u\in W^{1,p}(\mathbb{R}^2)$ with $u\ge1$ a.e. in $U$, let $g=(u\vee0)\wedge1$. Then $g\in W^{1,p}(\mathbb{R}^2)$ satisfies $0\le g\le1$ and $g=1$ a.e. in $U$. Then by Proposition \ref{prop_blend}, $W_{g,C}\in D$, $0\le W_{g,C}\le1$, and $W_{g,C}=1$ a.e. on $\{g=1\}\supseteq U$. Hence (\ref{eq_blend}) gives
$$\mathrm{Cap}^{\xi}(U)\le Q^{\xi}_1(W_{g,C})\le 2^{p-1}\left(C^pQ_1(g)+\varepsilon(C)\right).$$
Taking the infimum over $u$, we have
$$\mathrm{Cap}^{\xi}(U)\le 2^{p-1}\left(C^p \mathrm{Cap}^{1,p}(U)+\varepsilon(C)\right).$$
For any subset $A\subseteq \mathbb{R}^2$, for any open set $U\supseteq A$, the above inequality holds. Taking the infimum over $U$, we have
$$\mathrm{Cap}^{\xi}(A)\le 2^{p-1}\left(C^p \mathrm{Cap}^{1,p}(A)+\varepsilon(C)\right).$$

Consequently, for any subset $B\subseteq \mathbb{R}^2$, if $\mathrm{Cap}^\xi(B)=0$, then (\ref{eq_comparison}) gives directly $\mathrm{Cap}^{1,p}(B)=0$; if $\mathrm{Cap}^{1,p}(B)=0$, then (\ref{eq_comparison}) gives $\mathrm{Cap}^{\xi}(B)\le 2^{p-1}\varepsilon(C)$. Letting $C\to+\infty$, by (\ref{eq_tail}), we have $\mathrm{Cap}^{\xi}(B)=0$.
\end{proof}

\section*{AI statement}

The use of the concept of index in \cite{Hin10} and the concept of $p$-independence in \cite{EM25}, which are central to this paper, was suggested by OpenAI's ChatGPT (GPT-5.6 Sol) in mid-August 2026. We later used our recent results \cite{Yan25a,Yan25c,Yan25d,Yan25e} to establish non-linear analogs of the main arguments in \cite{Hin13}. During the final preparation of this paper, ChatGPT (GPT-6 Astra) was also used to assist with proofreading, polishing the English, and generating TikZ code for figures. The author takes full responsibility for the content of this paper.

\bibliographystyle{plain}

\end{document}